\documentclass[12pt]{article}
\usepackage{amssymb,amsmath,amsthm,secdot,bbm,mathrsfs}

\usepackage[normalem]{ulem}
\usepackage[
bookmarks=true,
bookmarksnumbered=true,
colorlinks=true, pdfstartview=FitV, linkcolor=blue, citecolor=blue,hypertexnames=false,
urlcolor=blue]{hyperref}

\usepackage[shortlabels]{enumitem}
\usepackage{color}

\usepackage[nameinlink,capitalize]{cleveref}
\usepackage{microtype}

\usepackage{makecell}

\hypersetup{%
	bookmarksnumbered, bookmarksopen=true, bookmarksopenlevel=1,%
}
\newcommand*{\Gref}[1]{\hyperref[mainSDE]{SDE($#1$)}}
\newcommand*{\Gtag}[1]{\tag{SDE($#1$)}}

\newcommand*{\Sref}[1]{\hyperref[SPDE]{SHE($#1$)}}
\newcommand*{\Stag}[1]{\tag{SHE($#1$)}}

\DeclareMathOperator{\1}{\mathbbm{1}}
\DeclareMathOperator{\Leb}{Leb}

\DeclareMathOperator{\law}{Law}%

\def\E{\hskip.15ex\mathsf{E}\hskip.10ex}
\def\P{\mathsf{P}}
\DeclareMathOperator{\Var}{Var}

\def\eps{\varepsilon}
\def\phi{\varphi}

\newcommand{\dtv}{d_{TV}}

\newtheorem{theorem}{Theorem}[section]
\newtheorem{lemma}[theorem]{Lemma}
\newtheorem{proposition}[theorem]{Proposition}
\newtheorem{corollary}[theorem]{Corollary}

\newtheorem{convention}[theorem]{Convention}
\theoremstyle{definition}
\theoremstyle{definition}\newtheorem{remark}[theorem]{Remark}
\theoremstyle{definition}\newtheorem{definition}[theorem]{Definition}

\crefname{Lemma}{Lemma}{Lemmas}
\crefname{Theorem}{Theorem}{Theorems}

\crefrangeformat{proposition}{Propositions~#3#1#4--#5#2#6}

\numberwithin{equation}{section}

\renewcommand{\ge}{\geqslant}
\renewcommand{\le}{\leqslant}

\newcommand{\nn}{\nonumber}

\newcommand{\wt}{\widetilde}
\newcommand{\wh}{\widehat}

\renewcommand{\d}{\partial}

\newcommand{\Di}[1]{|#1|}

\newcommand{\per}{{per}}
\newcommand{\neu}{{Neu}}

\newcommand{\A}{\mathcal{A}}
\newcommand{\C}{\mathcal{C}}

\newcommand{\D}{\mathbf{D}}

\newcommand{\F}{\mathcal{F}}
\newcommand{\FF}{\mathbb{F}}
\newcommand{\G}{\mathcal{G}}
\newcommand{\M}{\mathcal{M}}
\newcommand{\N}{\mathbb{N}}

\newcommand{\PP}{\mathcal{P}}

\newcommand{\R}{\mathbb{R}}

\newcommand{\W}{W}
\newcommand{\Z}{\mathbb{Z}}

\newcommand{\V}{\mathcal{V}}
\newcommand{\Vcl}{\V(\frac{1+H}{2})}
\newcommand{\VV}{\mathcal{V}_{SHE}}
\newcommand{\Vscl}{\VV(5/8)}

\newcommand{\bes}{\mathcal{B}}

\newcommand{\var}{\textrm{var}}

\newcommand{\BB}{{\mathbf{B}}}

\newcommand{\Ctimespace}[3]{{\C^{#1,0}L_{#2}({#3})}}
\newcommand{\Ctimespacezerom}[1]{\Ctimespace{0}{m}{#1}}
\newcommand{\Ctimespaced}[3]{{\C_{\D}^{#1,0}L_{#2}({#3})}}
\newcommand{\Ctimespacedzerom}[1]{\Ctimespaced{0}{m}{#1}}

\newcommand{\lm}{L_m(\Omega)}

\newcommand{\Lip}{\operatorname{Lip}}

\newcommand{\uu}{\textbf{u}}
\newcommand{\vv}{\wt{\textbf{v}}}

\usepackage{array} 
\usepackage{booktabs,tabularx}
\begin{document}
	
\title{Weak uniqueness for singular stochastic equations}
	
	\author{
		Oleg Butkovsky%
		\thanks{Weierstrass Institute, Mohrenstrasse 39, 10117 Berlin, FRG. Email: \texttt{oleg.butkovskiy@gmail.com}}$^{\,\,\,,\!\!\!\!\!}$
		\setcounter{footnote}{3}
		\thanks{Institut für Mathematik, Humboldt-Universität zu Berlin, Rudower Chaussee 25,
			12489, Berlin,	FRG.}
		\and
		Leonid Mytnik%
		\thanks{Technion --- Israel Institute of Technology,
			Faculty of Data and Decision Sciences,
			Haifa, 3200003, Israel.  Email: \texttt{leonid@ie.technion.ac.il}
		}
	}
	
	\maketitle

\begin{abstract}
	We put forward a new method for proving weak uniqueness of 
stochastic equations with singular drifts driven by a non-Markov or infinite-dimensional noise.
We apply our method to study stochastic heat equation (SHE) driven by Gaussian space-time white noise
	$$
	\frac{\d}{\d t} u_t(x)=\frac12 \frac{\d^2}{\d x^2}u_t(x)+b(u_t(x))+\dot{W}_{t}(x),    \quad t>0,\, x\in D\subset\R,
	$$
	and  multidimensional stochastic differential equation (SDE) driven by fractional Brownian motion with the Hurst index  $H\in(0,1/2)$
	$$
	d X_t=b(X_t) dt +d B_t^H,\quad t>0.
	$$
	In both cases $b$ is a generalized function in the Besov space $\bes^\alpha_{\infty,\infty}$, $\alpha<0$. 
	Well-known pathwise uniqueness results for these equations do not cover the entire range of the parameter $\alpha$, for which weak existence holds. What happens in the range where weak existence holds but pathwise uniqueness is unknown has been an open problem.
	We settle this problem and show that 	 for SHE weak uniqueness holds for $\alpha>-3/2$, and for SDE  it  holds for $\alpha>1/2-1/(2H)$; thus, in both cases, it holds in the entire desired range of values of  $\alpha$. 		 This extends seminal results of Catellier and Gubinelli (2016) and  Gy\"ongy and Pardoux (1993) to the weak well-posedness setting.  
 To establish these results, we develop a new strategy, combining  ideas from ergodic theory (generalized couplings of Hairer-Mattingly-Kulik-Scheutzow) with  stochastic sewing  of L{\^e}.
\end{abstract}

\tableofcontents

\section{Introduction}
Strong regularization by noise for stochastic equations has been quite well understood by now for a large class of random driving noises. On the other hand, very few results are available about weak regularization by noise when the driving noise is non-Markovian or infinite-dimensional. The goal of this article is to develop techniques and methods to establish \textit{weak} regularization by noise theory for stochastic equations with irregular drift. More precisely, we consider a stochastic heat equation (SHE) with a drift  driven by Gaussian space-time white noise
\begin{align}\label{SPDE}\Stag{u_0;b}
	&\partial_t u_t(x)=\frac12 \partial^2_{xx}u_t(x) +b(u_t(x))+\dot{W}_{t}(x),\quad t\in(0,T],\, x\in D,\\
	&u(0,x)=u_0(x),\nn
\end{align}
where the domain $D=[0,1]$ and $u_0$ is a measurable bounded function $D\to\R$. Our second example is  a $d$-dimensional stochastic differential equation (SDE) driven by fractional Brownian motion (fBM) with the Hurst index $H\in(0,1/2]$
\begin{align}\label{mainSDE}	\Gtag{x;b}  
	&dX_t=b(X_t)dt + dB_t^H,\quad t\in[0,T]\\
	&X_0=x\in\R^d\nn.
\end{align}
In both equations,  $b$ is a generalized function in the Besov space $\C^\alpha:=\bes^\alpha_{\infty,\infty}$, $\alpha<0$. We show  weak uniqueness of solutions to \ref{SPDE} if $\alpha>-3/2$ (\cref{t:swu}) and weak uniqueness of solutions to \ref{mainSDE} if $\alpha>1/2-1/(2H)$ (\cref{t:wu}). In both cases, weak uniqueness is obtained for the same range  of values of $\alpha$ for which  weak existence is known.   Moreover, for $d=1$ we also obtain
{\it strong} existence and uniqueness of solutions to \ref{mainSDE} for a certain range of $\alpha$ (\cref{t:su}) improving upon the Catellier-Gubinelli  condition \cite{CG16}. 

Thus, the main contribution of this paper is the extension of the seminal results of Catellier and Gubinelli \cite{CG16} and Gy\"ongy and Pardoux \cite{GP93a, GP93b} to the weak well-posedness setting. Below, we compare our results with the current state of the art.

To appreciate the difficulty of proving \textit{weak} well-posedness for \ref{SPDE}, note that methods used to study the weak well-posedness of SDEs driven by Brownian motion rely on a good It\^o formula. This tool is not available for \ref{SPDE} under the above assumptions.
Recently, a number of strong well-posedness results for \ref{SPDE} or \ref{mainSDE} have been obtained using the sewing lemma or the stochastic sewing lemma \cite{CG16,D24,DGmult24,HP20}. However, this approach does not provide weak uniqueness results for \ref{SPDE} beyond the strong uniqueness regime. Indeed, using stochastic or deterministic sewing, one can show that the function
\begin{equation}\label{smoothfunpde}
	\widetilde b\colon x\mapsto \int_0^1 b(V_r(0)+x)\,dr,\quad x\in \R, 
\end{equation}
where $V$ solves the linear SHE $\partial_t V_t(x)=\frac12 \partial^2_{xx}V_t(x) +\dot{W}_{t}(x)$ (the stochastic convolution), is \textit{Lipschitz} for $b\in\C^\alpha$, $\alpha>-1$. 
Hence one can get pathwise uniqueness by a Gronwall-type argument. However, if $b\in\C^\alpha$, $\alpha<-1$, then $\wt b$ is in general no longer Lipschitz, and therefore, Gronwall-type arguments combined with sewing or stochastic sewing are not helpful for proving uniqueness.  Consequently, showing weak uniqueness of solutions to \ref{SPDE}  is a challenging problem. 

To overcome this problem we propose a new strategy that combines coupling and stochastic
sewing techniques.

{
	\begin{table}
		\centering
		\renewcommand{\arraystretch}{1.5}	 
		\begin{tabularx}{\textwidth}{l>{\raggedright\arraybackslash}X>{\raggedright\arraybackslash}X>{\raggedright\arraybackslash}X} 
			\toprule
			&\ref{SPDE} &\ref{mainSDE}, $H = \frac{1}{2}$ & \ref{mainSDE}, $H \in(0,\frac{1}{2})$\!\!\!\!\\
			\midrule
			Strong uniqueness\,\,\phantom{1}&\cite{ABLM}: $\alpha>-1$, critical stochastic sewing& \cite{ver80}: $\alpha\ge 0$, Zvonkin's transform &\cite{CG16,LeSSL}: $\alpha>1-\frac1{2H}$, sewing, stochastic sewing \\
			Weak uniqueness & No results beyond strong uniqueness& \cite{bib:zz17,FIR17}: $\alpha>-1/2$, Zvonkin's transform & No results beyond strong uniqueness
			\\
			\bottomrule
		\end{tabularx}
		\caption{State of the art. Main results on weak and strong uniqueness for  \Gref{x;b}, where $b\in\C^\alpha$.} 
		\label{tab:my_table}
	\end{table}
}

\subsection[Comments on the background, proof strategy, and literature]{Further comments on the background of the problem, proof strategy, and related literature}
It was observed long ago that an ill-posed differential equation can become well-posed when perturbed by random noise. Indeed, consider the reaction-diffusion equation
$$\d_t u_t(x)=\frac12\d^2_{xx} u_t(x)+b(u_t(x)),\quad t\in(0,T],\, x\in D.
$$
If the drift $b$ is not Lipschitz, then this equation might be ill-posed (it may have no solutions or more than one solution). On the other hand, Gy\"ongy and Pardoux \cite{GP93a,GP93b} showed that the corresponding stochastic equation \ref{SPDE} has a unique strong solution when the drift $b$ is bounded or locally integrable. This is known as \textit{regularization by noise}. 
This phenomenon has also been observed in other models: 
within the framework of rough path theory \cite{Tindel}; and in dispersive equations \cite{choukplus}.
An overview of the topic is given by Gess in \cite{bengess}, and a detailed exposition of regularization by Brownian noise for ODEs has been written by Flandoli \cite{F11,flandoli2013topics}.

Equations with drifts less regular than those in \cite{GP93a,GP93b} were studied by Bounebache and Zambotti \cite{BZ14}. Motivated by problems arising in the study of random interface models, they considered \ref{SPDE} with a measure-valued drift $b$ and established the existence of a weak solution. However, uniqueness was left open. 

Thus, one can pose a natural question: how irregular can the drift $b$ be so that well-posedness of \ref{SPDE} still holds? For example, does \ref{SPDE} remain well-posed when $b$ is a measure or a Schwartz distribution?
The proof strategy in \cite{GP93a,GP93b} relied on Girsanov’s theorem and the comparison principle. \cite{BZ14} used Dirichlet form methods to establish weak existence. However, it is not clear whether these arguments can be adapted to obtain the uniqueness of \ref{SPDE} when $b$ is a measure or a Schwartz distribution.

Using stochastic sewing, \cite{ABLM} extended the Gy\"ongy-Pardoux  results to distributional drifts and showed strong uniqueness of solutions to \ref{SPDE} for $b\in\C^\alpha$, $\alpha>-1$. The threshold $\alpha>-1$ is due to the Lipschitz property of $\wt b$ as discussed above after equation  \eqref{smoothfunpde}.
Strong uniqueness in the same regime $\alpha>-1$ for the multiplicative noise was obtained in \cite{D24}.

Note that weak existence for solutions to \ref{SPDE} is known for less regular $b$, namely for $b\in\C^{\alpha}$, $\alpha>-3/2$ \cite{ABLM}. However, there are no known uniqueness results (neither weak nor strong) for the SHE with drifts of regularity $\alpha\in(-\frac32,-1)$. In this region of $\alpha$, the function $\wt b$ from \eqref{smoothfunpde} is still more regular than $b$, but it is not clear at all how to exploit this gain in regularity.  This is discussed further in \cref{s:oview}.

It is worth mentioning that there is another technique, the so-called duality method, used for proving weak uniqueness for stochastic equations and the corresponding martingale problems. This method has been successfully implemented for proving weak uniqueness for certain SPDEs driven by multiplicative noise in the presence of irregular function-valued drift (see, e.g., \cite{AT00}, \cite{BMS24}). However, applying the duality method for SPDEs usually requires that the drift and diffusion coefficients have a very particular form. For example, in \cite{AT00} and \cite{BMS24}, the drift is supposed to be given via a power series with specific conditions on the coefficients. It seems extremely challenging to apply the duality method when the drift $b$ is an arbitrary generalized function in $\C^\alpha$ with $\alpha < 0$, and we are not aware of any results in this direction.

%
%
%
%

A very similar open problem related to weak well-posedness also appears in regularization by noise for SDEs driven by a fractional Brownian motion. Recall that if $H=\frac12$, then strong well-posedness of \Gref{x;b} holds  for bounded or  locally integrable drifts  \cite{zvonkin74,ver80,kr_rock05}.  
Furthermore, it turns out that weak well-posedness holds even for the drifts which are not functions but rather Schwartz distributions.   Zhang and Zhao \cite{bib:zz17} (see also the work of Flandoli, Issoglio, Russo \cite{FIR17}), showed that \ref{mainSDE} with $b\in \C^\alpha$, $\alpha>-1/2$,  $H=1/2$ has a unique weak solution. We refer also to the works of Gr\"afner, Kremp, and Perkowski for further refinements \cite{kremp2023,grafner2024}.

On the other hand, if $H\neq\frac12$, then only strong well-posedness theory has been developed. Using sewing and stochastic sewing arguments, Catellier and Gubinelli
\cite{CG16} and Le \cite{LeSSL}  showed that \ref{mainSDE} has a unique strong (and even path-by-path unique) solution for $b\in\C^\alpha$, $\alpha>1-1/(2H)$, 
$d\geq 1$. One can see that this condition exactly matches (up to an arbitrarily small $\eps$) the Zvonkin-Veretennikov condition $\alpha\ge0$ for $H=1/2$. Time-dependent drifts are treated in \cite{GG22} and an extension to the case of infinitely regularizing noises is due to Harang and Perkowski \cite{HP20}. 
However, for reasons similar to what have been discussed for SHE, 
there are no results on  weak uniqueness for \Gref{x;b} for  $\alpha<1-1/(2H)$.  


The current state of the art  for uniqueness for  \ref{mainSDE} and \ref{SPDE}  is summarized in \cref{tab:my_table}.
As we mentioned above, the only available results on weak uniqueness for \ref{SPDE} and \ref{mainSDE} with $H \in(0,\frac12)$ are in the regime where strong uniqueness is known, while weak existence for these equations has been verified for a wider set of parameters. This article fills this gap. We  show weak uniqueness for \ref{SPDE} for $b\in\C^\alpha$, $\alpha>-3/2$, which  corresponds to the entire regime where weak existence was known.  We also obtain weak uniqueness for  \ref{mainSDE} for $b\in\C^\alpha$, $\alpha>1/2-1/(2H)$, again in the same regime where weak existence of solutions to \ref{mainSDE} is known. For $H=1/2$, this exactly matches the Zhang-Zhao condition $\alpha>-1/2$ and provides an alternative way to prove the weak uniqueness results of \cite{bib:zz17}.   One can see that the condition $\alpha>-3/2$ for SHE corresponds to the condition $\alpha>1/2-1/(2H)$ for $H=1/4$. In dimension $1$, we combine these results with a new extension of the comparison principle for distributional drifts and show that strong existence and uniqueness hold for \ref{mainSDE} beyond the Catellier-Gubinelli condition $\alpha>1-1/(2H)$.

To show weak uniqueness we develop a new approach which combines stochastic sewing with certain ideas from ergodic theory (generalized couplings). First, let us recall that a powerful tool to establish the unique ergodicity of a Markov process is the coupling method pioneered by Doeblin in the 1930s \cite{doeblin}; see also \cite{lig} for a review of Doeblin's contributions and their subsequent development. Assume that we are given a Markov transition kernel $(P_t)_{t\ge0}$ and would like to show convergence of transition probabilities $P_t(x,\cdot)$ and $P_t(y,\cdot)$, where $x$ and $y$ are in the state space of the process. The main idea is to construct two copies of the Markov process, $(X_t)_{t\ge0}$ and $(Y_t)_{t\ge0}$, starting from $x$ and $y$, respectively, and such that $\law(X_t)=P_t(x,\cdot)$, $\law(Y_t)=P_t(y,\cdot)$. Then, clearly,
\begin{equation*}
d_{TV}(P_t(x,\cdot),P_t(y,\cdot))\le \P(X_t\neq Y_t),
\end{equation*}
where $d_{TV}$ stands for the total variation distance. Thus, if we can construct processes $X$ and $Y$ in such a way that the probability  they do not couple by time $t$ tends to $0$ as $t\to\infty$, we get weak convergence of the transition probabilities and we can even bound the convergence rate. Of course, the main challenge now is how to construct processes $X$ and $Y$ to make $\P(X_t\neq Y_t)$ small. For SDEs driven by Brownian motion, this strategy was successfully implemented in  \cite{ver88}, where exponential ergodicity  was established. We refer to \cite{Eberle} for recent advancements.

While this strategy works quite well for finite-dimensional Markov processes, it typically fails when the state space is infinite-dimensional. The main reason is that the transition probabilities $P_t(x,\cdot)$ and $P_t(y,\cdot)$ are often orthogonal, resulting in zero probability for the processes $X$ and $Y$ to meet for any coupling. This phenomenon is discussed further in \cite[Section~1]{HMS} and \cite[Section~4.1]{Kulikbook} and specific examples of processes with mutually orthogonal transition probabilities are provided also in  \cite{ScheutzowCE} (stochastic delay equation) and \cite[Theorem~4.8]{BScmp} (stochastic heat equation).

To study such processes, Hairer, Mattingly, Scheutzow, and Bakhtin in a series of works in the 2000s \cite{H02,M02,Yura,HMS} developed a new approach called the generalized coupling method. The main idea is to ``help'' the process $Y$ to become closer to the process $X$. Therefore, an intermediate process  $\wt Y$ is introduced, which starts at the same initial point as $Y$ and receives an additional push towards $X$, for example, in the form $\lambda (X-\wt{Y})$ for large $\lambda>0$. Consequently, the distance between $\wt{Y}$ and $X$ rapidly decreases. Of course, the pair $(X,\widetilde{Y})$ is not a true coupling of $P_t(x,\cdot)$ and $P_t(y,\cdot)$, hence the name ``generalized coupling''. Nevertheless, it is often possible to show (usually with the help of Girsanov's theorem) that the process $\wt{Y}$ is not too far from $Y$ in total variation distance, allowing one to conclude
\begin{equation*}
	W_{d}(P_t(x,\cdot),P_t(y,\cdot))\le d_{TV}(\law(Y_t),\law(\wt Y_t))+\E|\wt Y_t-X_t|,
\end{equation*}
where $W_d$ denotes an appropriate Wasserstein distance. We address an interested reader to \cite[Section~1.1]{H02}, where the generalized coupling method is illustrated on a toy model. The latest developments of the method can be found in \cite{GMR,BKS18}.

Kulik and Scheutzow \cite{Ksch} adapted this method to show weak uniqueness for stochastic delay equations with irregular drift. They take as $X$ a weak solution to a stochastic equation with irregular drift and as $Y^n$ a solution to a well-posed approximating equation. They introduce an intermediate process $\wt Y^n$, which serves a similar role as $\wt Y$ above: $\wt Y^n$  receives a kick towards $X$ and becomes close to $X$ for each $\omega$, yet it remains close to $Y^n$ in law. A similar triangular inequality as above implies
\begin{equation*}
	W_{d}(\law(X_t),\law(Y^n_t))\le d_{TV}(\law(Y^n_t),\law(\wt Y^n_t))+\E|\wt Y^n_t-X_t|.
\end{equation*}
Since both terms on the right-hand side are small, we see that the process $Y^n_t$ converges weakly to $X_t$. Since $X$ was an arbitrary weak solution, weak uniqueness holds. \cite{Ksch} coins this approach  ``Control and Reimburse strategy''.

Han \cite{Han,Hanwave} uses the Kulik-Scheutzow method to get weak uniqueness for solutions to SHE with H\"older drift and diffusion coefficients, as well as for the stochastic wave equation with H\"older drift.

However let us stress that the strategy of \cite{Ksch,Han}  requires the drift to be a H\"older continuous function. The direct application of their method does not work  if the drift is a  Schwartz distribution or even a bounded function. 

Thus, we see that neither the application of stochastic sewing alone nor the generalized coupling technique alone allows  to obtain  weak uniqueness results beyond the known pathwise uniqueness results for equations with distributional drift. The main novelty of this paper is that we show that a combination of these techniques allows  to break this barrier and obtain weak uniqueness in the same regime where weak existence is known. In 
\cref{s:oview}, we review the ``Control and Reimburse'' approach of \cite{Ksch}, explain where exactly it breaks down if applied directly to equations with singular drift, and show how we proceed with a combination of this approach with stochastic sewing.

For simplicity and to highlight our arguments, we treat \ref{SPDE} and \ref{mainSDE} with  time-homogeneous drifts belonging to the Besov space  $\bes^\alpha_{p,\infty}$  with $p=\infty$. We also consider \ref{mainSDE} only with  $H\le1/2$. It is absolutely clear that the theory can be extended to cover the cases  $p<\infty$ and $H\in(1/2,1)$ and time-inhomogeneous drifts. It is also clear that a similar approach would work to show weak well-posedness of rough SDEs introduced by Friz, Hocquet, and L{\^e} in \cite{friz2021}. We leave this generalization for future work.

The rest of the paper is organized as follows. In \cref{s:mrsde,s:mrspde}, we present the main results concerning weak uniqueness for SHE, as well as weak and strong uniqueness for SDEs. \cref{s:oview} provides a detailed overview of the proofs, offering heuristic insights into the main ideas of our proof strategy. Key technical bounds essential for the subsequent proofs are placed in  \cref{s:auxbound}. In \cref{s:sdeproofs}, we present the proofs of the main results for SDEs and in \cref{s:she} for SHE. Finally, \cref{a:app} contains the proofs of additional technical results that may be known and are provided for the sake of completeness.

\textbf{Convention on constants}.  Throughout the paper $C$ denotes a positive constant whose value may change from line to line; its dependence is always specified in the corresponding statement. For brevity, we will not explicitly state the dependencies of the constants on the parameters $H$, $d$, $\alpha$ which are considered to be fixed.

\textbf{Convention on integrals}. In this paper, all integrals with respect to the deterministic measure are understood in the Lebesgue sense.

\textbf{Acknowledgements}.  OB is  funded by the Deutsche Forschungsgemeinschaft (DFG, German Research Foundation) under Germany's Excellence Strategy --- The Berlin Mathematics Research Center MATH+ (EXC-2046/1, project ID: 390685689, sub-project EF1-22) and DFG CRC/TRR 388 ``Rough
Analysis, Stochastic Dynamics and Related Fields", Project B08. LM is supported in part by ISF grant No. ISF 1985/22. 

\section{Main results}\label{s:mr}
Let us introduce the main notation and recall the basic definitions. Let $d\in \N$. For a set $Q\subset\R^k$, $k\in\N$,  the spaces of all continuous (respectively bounded measurable) functions $Q\to\R^d$ equipped with the supremum norm are denoted by $\C(Q;\R^d)$ (respectively  $\BB(Q;\R^d)$). For $\beta\in\R$ we denote by $\C^\beta=\C^\beta(\R^d)=\bes^\beta_{\infty,\infty}(\R^d)$ the Besov space of regularity $\beta$ and integrability parameters $\infty$ and $\infty$. We recall that if $\beta\in(0,1)$, then the space $\bes^\beta_{\infty,\infty}$ is the space of all H\"older continuous functions with the exponent $\beta$. The space $\C^\infty(\R^d,\R^d)$ denotes the space of bounded continuous functions $\R^d\to\R^d$ having bounded continuous derivatives  
of any order.

We denote the set of all signed finite Radon measures on $\R$ by $\M(\R)$ and the set of all non-negative finite Radon measures on $\R$ by $\M_+(\R)\subset \M(\R)$.

If $(E,\rho)$ is a metric space, then the space of all probability measures on $E$ equipped with the Borel $\sigma$-algebra $\mathscr{B}(E)$ is denoted by $\PP(E)$. For two probability measures  $\mu,\nu\in\PP(E)$ we define the Wasserstein (Kantorovich) distance between them as
\begin{equation}\label{wrho}
\W_\rho(\mu,\nu):=\inf_{\lambda\in\mathscr{C}(\mu,\nu)}\int_{E\times E} \rho(x,y)\,\lambda(dx,dy),
\end{equation}
where the infimum is taken over all couplings $\mathscr{C}(\mu,\nu)$, that is, all probability measures on $(E\times E, \mathscr{B}(E\times E))$ with marginals $\mu$ and $\nu$. The choice  $\rho(x,y)=\1(x\neq y)$, $x,y\in E$, leads to  the total variation distance $d_{TV}$, which is given by
$$
d_{TV}(\mu,\nu):=\inf_{\lambda\in\mathscr{C}(\mu,\nu)}\int_{E\times E} \1(x\neq y)\,\lambda(dx,dy)=\sup_{A\in\mathscr{B}(E)}|\mu(A)-\nu(A)|.
$$
It is well known that if  $(E,\rho)$ is Polish and $\rho$ is bounded, then weak convergence of measures is equivalent to convergence in $\W_\rho$, see, e.g., \cite[Corollary~6.13]{Villani}.

Let $\Gamma^d_t$, $t>0$, be the density of a $d$-dimensional vector with independent Gaussian components each of mean zero and variance $t$:
$$
\Gamma^d(t,x) =(2\pi t)^{-d/2}e^{-\frac{|x|^2}{2t}},\quad  x\in\R^d,
$$
and let $G^d_t$ be the corresponding Gaussian semigroup. In the case $d=1$, this index will be dropped, and the density and semigroup will be denoted simply by $\Gamma_t$ and $G_t$, respectively.

First, we present our results regarding  well-posedness of  \ref{SPDE}, and then we move on to the results concerning \ref{mainSDE}.

\subsection{Weak uniqueness for stochastic heat equation with distributional drift}\label{s:mrspde}
Throughout this subsection we assume $d=1$. 
Let  $p_t^{\per}$ and $p_t^{\neu}$ be the  the heat kernels on $[0,1]$ with the periodic and Neumann boundary conditions, respectively. That is,
\begin{align*}
	&p_t^\per(x,y):=\sum_{n\in\Z} \Gamma(t,x-y+n),\quad t>0,\,x,y\in [0,1];\\
	&p_t^\neu(x,y):=\sum_{n\in\Z} (\Gamma(t,x-y+2n)+\Gamma(t,x+y+2n)),\quad t>0,\,x,y\in [0,1].
\end{align*}


We consider  \Sref{u_0;b} on the interval $[0,1]$ in two possible setups:  with periodic boundary conditions and  with Neumann boundary conditions. To simplify the presentation of our results we introduce the following notational convention.

\begin{convention}\label{c:conv}
	
	The pair  $(D,p)$ stands for one of the two options: $([0,1],p^\per)$, or $([0,1],p^\neu)$. The corresponding semigroup will be denoted by $P$.
\end{convention}

We fix the time interval $T>0$. Let $(\Omega,\F,\P)$ be a probability space. If $(\F_t)$ is a complete filtration on this space, then we recall that a Gaussian process $W\colon L_2(D)\times[0,T]\times\Omega\to\R$ is called \textit{$(\F_t)$-space-time white noise} if for any $\phi,\psi\in L_2(D)$ the process $(W_t(\phi))_{t\in[0,T]}$ is an $(\F_t)$--Brownian motion and $\E W_s(\psi)W_t(\phi)=(s\wedge t)\int_D\phi(x)\psi(x)\,dx$. 

Now let us define what exactly we mean by a solution to \Sref{u_0;b} where $b$ is a distribution. 
Similar to the discussion above, we note that the term $b(u)$ is not well-defined but one can make sense of its convolution with space-time heat kernel. 

\begin{definition}
	We say that a sequence of functions $f^n\colon\R^d\to\R^d$, $n\in\Z_+$, converges to a function $f$ in $\C^{\beta-}$, $\beta\in\R$,  if $\sup_n \|f_n\|_{\C^{\beta}}<\infty$ and  for any $\beta'<\beta$ we have $\|f_n-f\|_{\C^{\beta'}}\to0$ as $n\to\infty$.
\end{definition}

\begin{definition}[{\cite[Definition~2.3]{ABLM}}] \label{def:ssol}
	Let $b\in\C^\beta$, $\beta\in\R$. We say that a continuous process $u\colon(0,T]\times D\times\Omega\to\R$ is a mild solution to \Sref{u_0;b}  with the initial condition $u_0\in \BB(D,\R)$ 	if there exists a continuous process $K\colon[0,T]\times D\times\Omega\to\R$ such that
	\begin{enumerate}[(i)]
		\item $u_t(x)=P_tu_0(x)+K_t(x)+\int_0^t \int_D p_{t-r}(x,y)W(dr,dy)$, $x\in D$, $t\in(0,T]$ a.s.;
		\item for any sequence  $(b^n)_{n\in\Z_+}$ of $\C^\infty(\R,\R)$ functions  converging to $b$ in $\C^{\beta-}$
		we have for any $N>0$
		$$
		\sup_{t\in[0,T]}\sup_{x\in D}\Bigl|\int_0^t\int_D p_{t-r}(x,y) b^n(u_r(y))\,dy\,dr-K_t(x)\Bigr|\to0\quad \text{in probability as $n\to\infty$}.
		$$
	\end{enumerate}
\end{definition}
Here the stochastic integral is the Wiener integral, see, e.g., \cite[Section~1.2.4]{Marta}. We note that given $b\in\C^{\beta}$ such approximating sequence  $(b^n)_{n\in\Z_+}$ of smooth functions converging to $b$ in $\C^{\beta-}$ always exists: one can just take $b_n:=G^d_{1/n}b$,  see, e.g., \cite[Lemma~A.3]{ABLM}. If $b\in\C^{\beta}$ and $\beta>0$, then $b(u_t(x))$ is well-defined, and it is immediate that this notion of a solution coincides with the standard notion of a solution. This is also the case if $b$ is a bounded measurable function (under a certain technical condition), see \cite{ABLMmw}.

\begin{remark}
	In this article, we restrict ourselves to solutions that are mild in the PDE sense. Note that one can give a similar definition of a PDE-weak solution with the distributional drift. Under certain technical conditions  these notions are equivalent as shown in \cite{ABLMmw}. 
\end{remark}

A (probabilistically) weak (PDE) mild solution to \Sref{u_0;b} is a couple $(u,W)$ on a complete filtered probability space $(\Omega, \F, \P,\FF= (\F_t)_{t\in[0,T]})$ such that $W$ is an $\FF$-space-time white noise, $u$ is adapted to $\FF$, and $u$ is a mild solution to \Sref{u_0;b} in the sense  of \cref{def:ssol}. We say that \textit{weak uniqueness} holds for \Sref{u_0;b} if whenever    $(u,W)$ and $(\overline u,\overline W)$ are two  weak solutions of this equation (not necessarily defined on the same probability space), then   $\law(u)=\law(\overline u)$ on the path space  $\C([0,T],\C(D;\R))$.

As it is standard in the analysis of SDEs or SPDEs with distributional drift (see, e.g., \cite[Definition~2.1(iii)]{BC}, \cite[Definition~3.1 and Corollary~5.3]{bib:zz17}, \cite[Theorem~1.2, condition (1.6)]{hao2023sdes}, \cite[Definition~2]{perkowski2018energy}),  we do not consider all solutions of \ref{SPDE} but rather restrict ourselves to solutions having certain regularity. We will consider the following  class 
of solutions.

\begin{definition}[{\cite[Definition 2.4]{ABLM}}] Let $\kappa\in[0,1]$. We say that a solution $(u,W)$ to  \Sref{u_0;b} belongs to the
	class $\VV(\kappa)$ if for any $m\ge2$ 
	denoting  $K_t(x):=u_t(x)-\int_0^t \int_D p_{t-r}(x,y)W(dr,dy)$, $t\in[0,1]$, $x\in D$, we have
	\begin{equation*}
		\sup_{(t,x)\in(0,T]\times D}\|u_t(x)\|_{L_m}<\infty,\qquad \sup_{0< s\le t\le T}\sup_{x\in D}\frac{\|K_t(x)- P_{t-s} K_s(x)\|_{L_m(\Omega)}}{|t-s|^\kappa}<\infty.
	\end{equation*}	
\end{definition}

If $u_0\in \BB(D)$,  $b\in\C^\alpha$, where
\begin{equation}\label{spdecond}\tag{S}
	\alpha>-\frac32,
\end{equation}	
then \cite[Theorem~2.6]{ABLM} established weak existence of solution to \Sref{u_0;b} in class $\VV(1+\alpha/4)\subset\Vscl$. Our main result concerning SPDEs shows that weak uniqueness of solutions to \Sref{u_0;b} holds in exactly the same regime. 

\begin{theorem}\label{t:swu}
	Let  $\alpha\in\R$,  $u_0\in \BB(D)$, $b\in\C^\alpha$. Suppose that \eqref{spdecond} holds.
	Then 
	\begin{enumerate}[(i)]
		\item  weak uniqueness holds for solutions to equation \Sref{u_0;b}  in the class $\Vscl$;  
		\item let $\{b_n, n\in\Z_+\}$ be a sequence of $\C^\infty(\R^d,\R^d)$ functions converging to $b$ in $\C^{\alpha-}$ as $n\to\infty$. Let $u^n$ 
		be a strong solution to \Sref{u_0;b_n}. Then the sequence $(u^n)_{n\in\Z_+}$
		weakly converges in $\C([0,T],\C(D;\R))$ and its limit is a unique weak solution to  \Sref{u_0;b} in class $\Vscl$.
	\end{enumerate}
\end{theorem}

We remind that \cref{t:swu} establishes weak uniqueness in 
two settings, see \cref{c:conv}.

\subsection{Weak and strong well-posedness for SDEs driven by fBM}\label{s:mrsde}

Now let us present our results regarding  \ref{mainSDE}.

Let $T>0$. Let $(\Omega,\F,\P)$ be a probability space equipped with a complete filtration $\FF:=(\F_t)_{t\in[0,T]}$. Let $(B^H_t)_{t\in[0,T]}$ be a $d$-dimensional fractional Brownian motion of Hurst index $H\in(0,1)$ defined on this space. It is well-known (see, e.g., \cite[Section~2.2]{OuNu} and \cite[formula (5.8) and Proposition~5.1.3]{Nu})
that one can construct on the same probability space a standard $d$-dimensional Brownian motion $W$ such that
\begin{equation}\label{bhrepr}
B_t^H=\int_0^t K_H(t,s) dW_s,\quad 0\le s \le t\le T,
\end{equation}	
where the kernel $K_H$ is given by 
\begin{align*}
	&K_H(t,s):=Cs^{\frac12-H}\int_s^t (r-s)^{H-\frac32}r^{H-\frac12}\,dr\quad\text{when}\quad H>1/2,\\
	&K_H(t,s):=C\Bigl(t^{H-\frac12}s^{\frac12-H}(t-s)^{H-\frac12}
	\nonumber\\&\quad\quad \quad \quad \quad+(\frac12-H)s^{\frac12-H}\int_s^t (r-s)^{H-\frac12}r^{H-\frac32}\,dr\Bigr)\quad\text{when}\quad H<1/2,
\end{align*}
where $0\le s \le t \le T$, and $C$ is a certain positive constant. Following \cite[Definition~1]{OuNu}, we say that $B^H$ is an $\FF$-fractional Brownian motion, if 
there exists an $\FF$-Brownian motion $W$ so that \eqref{bhrepr} holds. Denoting by $\FF^{B^H}$ the natural filtration of $B^H$, we note that $B^H$ is always an $\FF^{B^H}$--fractional Brownian motion, because the natural filtrations generated by $W$ from \eqref{bhrepr} and $B^H$ coincide, see \cite[Section 5.1.3]{Nu}.

Now let us define precisely what we mean by a solution to equation \Gref{x;b}.  
As explained above, when $b$ is a Schwartz distribution, the drift term $b(X_t)dt$ in this equation is not well-defined. As suggested initially by Bass and Chen \cite[Definition~2.1]{BC}, \cite[Definition~2.5]{BC03} (see also \cite[Definition~3.9]{bib:zz17}, \cite[Definition~2.1]{ABM2020}) and became standard by now, one does not define $b(X_t)dt$ directly but rather works with the integral of this term, $\int b(X_t) dt$, which is understood as the limit of the corresponding approximations.

\begin{definition}\label{D:sol}
	Let $b\in\C^{\beta}$, $\beta\in\R$. We say that a continuous   process $(X_t)_{t\in[0,T]}$ taking values in $\R^d$ is a solution to \Gref{x;b} with the initial condition $x\in\R^d$, if there exists a continuous process $(\psi_t)_{t\in[0,T]}$ taking values in $\R^d$ such that:
	\begin{enumerate}[(i)]
		\item $X_t=x+\psi_t+B_t^H$, $t\in[0,T]$ a.s.;
		\item\label{cond2md} for \textit{any} sequence $(b^n)_{n\in\Z_+}$ of $\C^\infty(\R^d,\R^d)$ functions converging to $b$ in $\C^{\beta-}$ we have
		\begin{equation*}
			\lim_{n\to\infty}\sup_{t\in[0,T]}\Big|\int_0^t b^n(X_r)\,dr-\psi_t\Big|= 0\,\,\text{in probability}.
		\end{equation*}
	\end{enumerate}
\end{definition}
We note as above that in case $\beta>0$ \cref{D:sol} coincides with the usual definition of a  solution to \ref{mainSDE}.

As in \cite[Section~3.2]{OuNu}, we define a \textit{weak solution} to \Gref{x;b} as a couple $(X,B^H)$ on a complete filtered probability space $(\Omega, \F, \P,\FF= (\F_t)_{t\in[0,T]})$ such that $B^H$ is an $\FF$-fractional Brownian motion, $X$ is adapted to $\FF$, and $X$ is a solution to \eqref{mainSDE} in the sense  of \cref{D:sol}. A weak solution $(X,B^H)$ is called a \textit{strong solution}  if $X$ is adapted to $\FF^{B^H}$. We say that \textit{weak uniqueness} holds for \Gref{x;b} if whenever    $(X,B^H)$ and $(\overline X,\overline B^H)$ are two  weak solutions of this equation (not necessarily defined on the same probability space), then   $\law(X)=\law(\overline X)$ on the path space  $\C([0,T];\R^{d})$.  We say that \textit{pathwise uniqueness} holds for \Gref{x;b} if for any two weak solutions of \Gref{x;b}  $(X,B^H)$ and $(\overline X,B^H)$  with common noise $B^H$ on a common probability space (w.r.t. possibly different filtrations),  one has $\P(X_t=\overline X_t \text{ for all $t\in[0,T]$})=1$. 

We consider the following class of solutions.

\begin{definition} Let $\kappa\in[0,1]$. We say that a solution $(X,B^H)$ to \Gref{x;b} belongs to the
class $\V(\kappa)$ if for any $m\ge2$ we have for $\psi:=X-B^H$
\begin{equation*}
\sup_{0\le s \le t\le T} \frac{\|\psi_t-\psi_s\|_{L_m(\Omega)}}{|t-s|^\kappa}<\infty.
\end{equation*}	
\end{definition}


Clearly, if $b$ is a non-negative measure, then the process $X-B^H$ is nondecreasing, and thus automatically of finite variation.

Now we are ready to present our next main result. We recall that if $b\in\C^\alpha$, $x\in\R^d$, and $\alpha$ satisfies 
\begin{equation}\label{alphacond}\tag{A}
	\alpha>\frac12-\frac1{2H},
\end{equation}	
then  \cite[Theorem~8.2, Lemma 8.4]{GG22} implies that  \Gref{x;b} has a weak solution, and this solution lies in the class $\V(1+\alpha H)\subset \Vcl$. We are able to show that in the entire regime \eqref{alphacond} where weak existence holds, weak uniqueness also holds.

\begin{theorem}\label{t:wu}
Let  $\alpha\in\R$, $x\in\R^d$, $b\in\C^\alpha$, $H\in(0,\frac12]$. Suppose that \eqref{alphacond} holds.
Then 
\begin{enumerate}[(i)]
	\item  equation \Gref{x;b} has a unique weak solution  in the class $\Vcl$;  
	\item let $\{b_n, n\in\Z_+\}$ be a sequence of $\C^\infty(\R^d,\R^d)$ functions converging to $b$ in $\C^{\alpha-}$ as $n\to\infty$. Assume that the sequence $\{x_n, n\in\Z_+\}$, where $x_n\in\R^d$, converges to $x$ as $n\to\infty$. Let $X_n$ 
	be a strong solution to \Gref{x_n;b_n}. Then the sequence $(X_n,B^H)$
	weakly converges in $\C([0, T], \R^{2d})$ and its limit is a unique weak solution to \Gref{x;b} in class $\Vcl$.
\end{enumerate}
\end{theorem}	

\begin{remark}
For $H=1/2$, \cref{t:wu} establishes weak uniqueness and stability for solutions to SDEs driven by standard Brownian motion in the regime $\alpha > -1/2$. This provides an alternative proof of the results of Zhang and Zhao \cite{bib:zz17}, where weak uniqueness was established in the same range of the parameter $\alpha$. We note that our arguments are very different: we use generalized coupling methods and stochastic sewing, whilst \cite{bib:zz17} uses PDE techniques and the Zvonkin transformation.
\end{remark}

We also obtained the following stability result showing H\"older continuous dependence of the law of the solution to \Gref{x;b} on the initial condition $x$ and on the drift $b$. Hereafter, $\W_{\|\cdot\|\wedge1}$ denotes the Wasserstein distance,   as introduced above in \eqref{wrho}, on a metric space $\C([0,1],\R^{2d})$ equipped with the distance $\|\cdot\|_{\C([0,1],\R^{2d})}\wedge1$. Bounding a natural distance of the metric space by $1$ is standard in ergodic literature, see, e.g., \cite[Proposition~3.12]{HMAOM}, \cite[equation~(5.1) and Definition~4.6]{HMS}, \cite[p.~555, definition of $\rho$]{Yura} and so on. The main reason is that convergence in $\W_{\|\cdot\|}$ is not equivalent to weak convergence in $\C([0,1],\R^{2d})$ (consider $\mu_n:=(1-\frac1n)\delta_0+\frac1n\delta_n$), so this metric is somehow too strong. On the other hand,  convergence in $\W_{\|\cdot\|\wedge1}$ is equivalent to the  weak convergence in $\C([0,1],\R^{2d})$, cf.  \cite[Theorem~6.9]{Villani} vs \cite[Corollary~6.13]{Villani}.

\begin{theorem}\label{t:sb}
Let  $\alpha\in\R$, $x_1,x_2\in\R^d$, $b_1,b_2\in\C^\alpha$, $H\in(0,\frac12]$. Suppose that \eqref{alphacond} holds. Let $(X^i,B^{H,i})$ be a weak  solution to \Gref{x_i,b_i}, $i=1,2$. Then
there exists $\eps>0$ such that
\begin{equation}\label{mmmssde}
	\W_{\|\cdot\|\wedge1}(\law(X^1, B^{H,1}),\law(X^2,B^{H,2}))\le C \Gamma (\|b_1-b_2\|_{\C^{\alpha}}^\eps+|x_1-x_2|^\eps).
\end{equation}
for 
$$
\Gamma:=(1+\|b_1\|_{\C^\alpha}^{\frac{40}{\eps}}+\|b_2\|_{\C^\alpha}^{\frac{40}{\eps}})
(1+|x_1|+|x_2|).
$$
\end{theorem}

Up until now dimension $d$ was an arbitrary natural number. We recall however that the case $d=1$ is special. Here, one typically gets better results on the well-posedness of SDEs than if $d\ge2$.  Indeed, for SDEs driven by Brownian motion, strong well-posedness holds in $d=1$ for drift $b\in\C^\alpha$, $\alpha>-1/2$ \cite[Theorem~2.6]{BC}, while if $d\ge2$, then strong well-posedness is known only for $b$ being a bounded (or integrable) function \cite{ver80,kr_rock05} but not a general distribution.  Similarly, for SDEs driven by $\beta$-stable L\'evy processes with drift $b\in\C^{\alpha}$, strong well-posedness holds in $d=1$ for $\alpha>\frac12-\frac\gamma2$ \cite[Theorem~2.3]{ABM2020}, and $d\ge2$ for $\alpha>1-\frac\gamma2$ \cite[Theorem~1.1]{Pr12}.

However, to the best of our knowledge,  until now, no such results were available for SDEs driven by fractional Brownian motion. The best results on strong well-posedness for \Gref{x;b} with $b\in\C^\alpha$ are due to Catellier and Gubinelli \cite[Theorem~1.13]{CG16}, who showed that \Gref{x;b} has a unique strong solution whenever $\alpha>1-\frac{1}{2H}$. The next theorem shows that for $d=1$, this result can be improved.

Before stating the theorem let us just recall that there is an alternative way to define a solution to \Gref{x;b}, where $b$ is a distribution, via non-linear Young integrals \cite[Section~2]{CG16}, \cite{LucioNLY}. Namely, if \eqref{alphacond} is satisfied  and additionally one of the following holds:
\begin{align}\label{bcond1}\tag{B1}
&(1+\alpha H)(\alpha+\frac1{2H})>\frac12\qquad \text{or}\\
&b\in \M_+(\R),\tag{B2}\label{bcond2}
\end{align}
then  the drift term $\int b(X_r)\,dr$  can be defined directly without the need to consider approximating sequence $(b^n)$ and the new definition of solution coincides with the old one \cite[Remark~8.5]{GG22}. Our next result states that in $d=1$ equation \Gref{x;b} has a unique strong solution in exactly the same regime \eqref{bcond1} or \eqref{bcond2} where its solution can be written as a non-linear Young integral. We note that the regime  \eqref{bcond1} improves the condition $\alpha>1-1/(2H)$ from  \cite{CG16}  for every $H\in(0,1/2]$.

\begin{theorem}\label{t:su}
	Let $d=1$, $\alpha\in\R$, $x\in\R$, $b\in\C^\alpha$, $H\in(0,\frac12]$. Suppose that \eqref{alphacond} holds.  Then 
if \eqref{bcond1} or \eqref{bcond2} is satisfied, then  equation \Gref{x;b} has a   strong solution $X\in\Vcl$ and pathwise uniqueness holds in the class $\Vcl$.
\end{theorem}	




\subsection{Overview of the proofs of main results}\label{s:oview}

\textbf{Weak uniqueness}. For the convenience of the reader we demonstrate the heuristics of our proof strategy on the following simplified example. Then  we highlight the necessary changes which are needed in order to pass from this partially informal explanation to rigorous proofs of \cref{t:wu,t:swu}. We also explain the novelty of our ideas compared with \cite{Ksch}. To make our exposition clearer, in this subsection we skip many technical details and, in particular, omit the arbitrarily small exponents.

Thus, we would like to show weak uniqueness for solutions to SDE
\begin{equation}\label{toyeq}
d X_t= b(X_t) dt +dB_t^H,\quad	X_0=0,
\end{equation}
where $b\in \C^\alpha$, $\alpha<0$. Assume for simplicity that $b$ is a function. We fix a sequence $(b^n)_{n\in\Z_+}$  of smooth functions converging to $b$ in $\C^{\alpha-}$. Let  $X^n$ be the strong solution to \Gref{0;b^n}.

Now consider \textbf{any} weak solution $(X,B^H)$ to \eqref{toyeq}. 
Our goal is to show that the sequence $(X^n,B^H)_{n\in\Z_+}$ weakly converges to this solution $(X,B^H)$. Since  $(X,B^H)$ was an arbitrary solution to \eqref{toyeq}, this would imply that the sequence $(X^n,B^H)$ converges in law to any weak solution of  \eqref{toyeq}, and thus all of them have the same distribution, that is, weak uniqueness holds.

 A  naive direct approach, which fails here, would be simply to write 
\begin{equation*}
\|X_t-X^n_t\|_{L_2(\Omega)}= \Bigl\|\int_0^t (b(X_r)-b^n(X^n_r)) \,dr \Bigr\|_{L_2(\Omega)},
\end{equation*}
and then use stochastic sewing. Under the condition \eqref{alphacond}, one would get (see \cref{l:finfin})
\begin{equation}\label{directbound:s}
	\sup_{t\in[0,t_0]}\|X_t-X^n_t\|_{L_2(\Omega)}\le C t_0^{\frac12}\sup_{t\in[0,t_0]}\|X_t-X^n_t\|_{L_2(\Omega)}^{(\alpha+\frac1{2H})\wedge1}+\text{good small terms}.
\end{equation}
If $\alpha+\frac1{2H}>1$, one can remove the bad term $\sup_{t\in[0,t_0]}\|X_t-X^n_t\|_{L_2(\Omega)}$ on the right-hand side for small $t_0>0$, and even get strong uniqueness of solutions to SDE \eqref{toyeq}, see \cite[Theorem~1.4]{GG22}, \cite[Theorem~1.13]{CG16}. However, we are interested in the regime where weak existence of \eqref{toyeq} holds, that is,  $\alpha>\frac12-\frac1{2H}$ \cite[Theorem~8.2]{GG22}. In this regime, the exponent $\alpha+\frac1{2H}$ is less than $1$ and \eqref{directbound:s} does not imply that $\sup_{t\in[0,t_0]}\|X_t-X^n_t\|_{L_2(\Omega)}$ is small. Therefore this direct approach does not work. 
Nevertheless, as we demonstrate below, bound \eqref{directbound:s} is not entirely useless.

Therefore, we avoid comparing the solutions $X$ and $X^n$ directly. Instead, in the spirit of the ``Control-and-Reimburse''  strategy of \cite{Ksch}, we introduce now an auxiliary process $\wt Y^n$, which is defined on the same space as $X$ and solves SDE
\begin{equation*}
d \wt Y^n_t =b^n(\wt Y^n_t) dt+\lambda(X_t-\wt Y^n_t)dt+d B_t^H,
\end{equation*}
where the parameter $\lambda>1$ will be fixed later. The main idea now is to show that the process $\wt Y^n$ is close  to $X^n$ in law for $\lambda$ not too large, and close to $X$ in distance for  $\lambda$ not too small. Then by picking the best $\lambda$ from these two opposite requirements, we derive that $X^n$ and $X$ are close to each other in law. By the triangle inequality, 
\begin{equation}\label{key}
\W_{\|\cdot\|\wedge1}(\law(X^n),\law(X))\le d_{TV}(\law(X^n),\law(\wt Y^n))+\E \|\wt Y^n-X\|_{\C([0,1])}.
\end{equation}
We bound the first term of \eqref{key} in the same spirit as in \cite{Ksch}. The key difference between our approach   and that in \cite{Ksch} lies  in treating the second term of \eqref{key}. This is precisely where stochastic sewing plays a key role.

Introduce a process $\wt B_t^H:=\int_0^t\lambda(X_r-\wt Y^n_r)dr+ B_t^H$. Then $\wt Y^n$ satisfies 
\begin{equation*}
	d \wt Y^n_t =b^n(\wt Y^n_t) dt+d\wt  B_t^H.
\end{equation*}
Thus $(\wt Y^n,\wt B^H)$ is a solution to \Gref{0;b^n}. 
Since $b^n$ is a smooth function, weak uniqueness holds for this equation. Therefore, Girsanov's theorem and Pinsker's inequality yield (recall that $H\le\frac12$)
\begin{equation}\label{tvbound:s}
d_{TV}(\law(X^n),\law(\wt Y^n))\le d_{TV}(\wt B^H,B^H)\le  C \lambda\bigl\|\|X-\wt Y^n\|_{\C([0,1])}\bigr\|_{L_2(\Omega)}.
\end{equation}

Next, to compare $\wt Y^n$ and $X$, we benefit from the control term $\lambda(X-\wt Y^n)$ which pushes $\wt Y^n$ towards $X$. After an easy calculation, we deduce for $t\in[0,1]$
\begin{equation}\label{keyint:s}
|X_t-\wt Y^n_t|\le \Bigl|\int_0^t e^{-\lambda(t-r)}(b(X_r)-b_n(\wt Y_r^n))\,dr\Bigr|.
\end{equation}	
Carefully bounding the right-hand side of the above inequality is absolutely crucial for the entire weak uniqueness proof. 
First, let us review how \cite{Ksch} deals with bounds of a similar type, and then we explain how we improve  their approach. 

Assume additionally for a moment that $\alpha>0$. Taking supremum over $t\in[0,1]$ in the above bound, estimating $|b(x)-b^n(y)|$ as $\|b\|_{\C^\alpha}|x-y|^\alpha+\|b-b^n\|_{\C(\R)}$, $x,y\in\R^d$, and using Young's inequality, one gets
for arbitrary $\eps>0$
\begin{align}\label{demoestimate}
	\|X-\wt Y^n\|_{\C([0,1])}&\le C \lambda^{-1}\|b-b^n\|_{\C(\R)}+C \lambda^{-1}\|b\|_{\C^\alpha}\|X-\wt Y^n\|_{\C([0,1])}^\alpha \\
	&\le C \lambda^{-1}\|b-b^n\|_{\C(\R)}+C \lambda^{-\eps}\|b\|_{\C^\alpha}\|X-\wt Y^n\|_{\C([0,1])}+C\lambda^{-\frac{1}{1-\alpha}}\|b\|_{\C^\alpha}.\nn
\end{align}	
Hence, for $\lambda$ large enough the following inequaltiy holds:
\begin{equation}\label{fa:s}
	\|X-\wt Y^n\|_{\C([0,1])}\le  C \lambda^{-1}\|b-b^n\|_{\C(\R)}+C\lambda^{-\frac{1}{1-\alpha}}\|b\|_{\C^\alpha}.
\end{equation}
Combining this with the total variation bound \eqref{tvbound:s} and substituting this into \eqref{key}, one finally derives
\begin{equation*}
	\W_{\|\cdot\|\wedge1}(\law(X^n),\law(X))\le C \|b-b^n\|_{\C(\R)}+C\lambda^{-\frac{\alpha}{1-\alpha}}\|b\|_{\C^\alpha}.
\end{equation*}
By taking the limit as $\lambda\to\infty$, one gets the desired weak convergence of $X^n$ to an arbitrary weak solution $X$ and thus weak uniqueness.

We see, however, that this approach does not work for $\alpha < 0$, as neither of the terms on the right-hand side of the last inequality converges as $n, \lambda \to \infty$. Our main insight is to utilize stochastic sewing to bound \eqref{keyint:s} more carefully. Applying a modification of \eqref{directbound:s}, we obtain (see \cref{c:finfin})
\begin{equation*}
\sup_{t\in[0,1]}\|X_t-\wt Y_t^n\|_{L_2(\Omega)}\le C \lambda^{-1-(\alpha-\frac12)H}
\sup_{t\in[0,1]} \|X_t-\wt Y^n_t \|_{L_2(\Omega)}^{\frac12}+\text{good small terms}.
\end{equation*}
Of course, similar to \eqref{directbound:s}, the term $\|X_t - \widetilde{Y}^n_t \|_{L_2(\Omega)}$ appears with an exponent less than $1$, and we cannot eliminate it directly. However, this is not a problem anymore due to the presence of the additional factor $\lambda^{-1-(\alpha-\frac{1}{2})H}$, which we can make very small.
Comparing this with the way \cite{Ksch} treats the right-hand side of \eqref{keyint:s}  (see \eqref{demoestimate} above), we see that a certain tradeoff takes place. The term $\|X_t-\wt Y_t^n\|$ now has a much bigger exponent: $1/2$ instead of $\alpha$ (recall that $\alpha$ is negative now). The price to pay is that the parameter $\lambda^{-1}$ now has a smaller exponent $1+(\alpha+\frac12)H$ instead of $1$. An additional disadvantage is that we are now unable to obtain a direct almost sure bound, but rather a weaker $L_2(\Omega)$ bound. Nevertheless, this tradeoff is sufficient for us to close the buckling loop. By Young's inequality
\begin{equation*}
	\sup_{t\in[0,1]}\|X_t-\wt Y_t^n\|_{L_2(\Omega)}\le C \lambda^{-2-(2\alpha-1)H} +C \lambda^{-\eps}
	\sup_{t\in[0,1]} \|X_t-\wt Y^n_t \|_{L_2(\Omega)}+\text{good small terms},
\end{equation*}
which implies 
\begin{equation}\label{onemore:s}
	\sup_{t\in[0,1]}\|X_t-\wt Y_t^n\|_{L_2(\Omega)}\le C \lambda^{-2-(2\alpha-1)H} +\text{good small terms}.
\end{equation}
Recall that the Girsanov bound \eqref{tvbound:s}  includes an additional  factor $\lambda$; therefore, in order to succeed, we must impose 
$$
\lambda \lambda^{-2-(2\alpha-1)H}\to0,\qquad \text{as $\lambda\to\infty$},
$$
which implies exactly our condition \eqref{alphacond}. Thus, substituting \eqref{onemore:s} into  \eqref{tvbound:s} and  \eqref{key} and taking care of ``good small terms'', we get the  bound 
\begin{equation*}
	\W_{\|\cdot\|\wedge1}(\law(X^n),\law(X))\le C    \lambda^{-\frac12-H(\alpha-\frac12)}
	+C \lambda^2  \|b-b^n\|_{\C^{\alpha-\eps}},
\end{equation*}
which significantly improves \cite{Ksch}-type bound \eqref{fa:s}. We see that in our regime  \eqref{alphacond}, the exponent of $\lambda$, that is, $-\frac12-H(\alpha-\frac12)$, is negative. Therefore we can pick $\lambda:=\|b-b^n\|_{\C^{\alpha-\eps}}^{-1/3}$ to get 
\begin{equation*}
	\W_{\|\cdot\|\wedge1}(\law(X^n),\law(X))\le C   \|b-b^n\|_{\C^{\alpha-\eps}}^\rho
\end{equation*}
for certain  $\rho=\rho(\alpha,H)>0$. This yields the desired weak convergence of the fixed sequence $(X^n)_{n\in\Z_+}$ to any weak solution $X$, establishing weak uniqueness. This completes the proof.

Now it remains to make this informal description of our proof strategy rigorous. This is the subject of the remaining part of the paper. The derivation of an analogue of \eqref{keyint:s} is done in \cref{L:ks2}. Bound \eqref{tvbound:s} is justified in \cref{l:ks1}, and they are combined in \cref{l:ml}. It is important to note that in our setting $b$ is not function but rather a Schwartz distribution belonging to $\C^\alpha$. As a result, the integral in the right-hand side of \eqref{keyint:s} is not well-defined, and even \eqref{keyint:s} does not hold. Therefore, one must consider an appropriate approximation. 

The proof of weak uniqueness for SHE goes along the same lines; however, some modifications due to the infinite-dimensional nature of the problem are necessary. Most importantly, instead of the $\sup$ distance $\|\cdot\|_{\C([0,1])}$, we have to work with the weighted in time and space norm defined in \eqref{stwd} and rely on the backward uniqueness-type results. This is the subject of \cref{s:she}.

\vskip2ex
\textbf{Strong uniqueness for SDEs in case $d=1$}. As before, we first explain the heuristics for the case when $b$ is a function, and then proceed to the general case.

Let $d=1$, and let $(X^1,B^H)$, $(X^2,B^H)$ be two weak solutions to SDE \eqref{toyeq} defined on the same probability space with common noise. Consider a process $Y_t:=X^1_t\wedge X^2_t$. If we show that $Y$ is also a weak solution to this equation, then by weak uniqueness we get for any bounded continuous strictly increasing function $f\colon\R\to\R$
$$
\E f(X^1_t\wedge X^2_t)=\E f(X^1_t),\quad t\ge0,
$$ 
which implies $X^1_t\wedge X^2_t=X^1_t$ a.s. Similarly, $X^1_t\wedge X^2_t=X^2_t$ a.s., and thus strong uniqueness holds. 

To show that $Y$ solves \eqref{toyeq}, we fix $\omega\in\Omega$.  Let $A:=\{s\in[0,1]: X^{1}_s(\omega)>X^{2}_s(\omega)\}$. We see that $A$ is an open set, and therefore it is a union of countably many mutually disjoint intervals $A=\cup (s_k,t_k)$. Further, for any $k\in\N$, by the continuity of $X^1$ and $X^2$,
\begin{align*}
\int_{[s_k,t_k]}b(X_r^1)\,dr&=X^1_{t_k}-X^1_{s_k}-(B^H_{t_k}-B^H_{s_k})\\
&=X^2_{t_k}-X^2_{s_k}-(B^H_{t_k}-B^H_{s_k})=\int_{[s_k,t_k]}b(X_r^2)\,dr.
\end{align*}
This implies  
\begin{equation}\label{idd:s}
	\int_{A }b(X_r^1)\,dr=\int_{A}b(X_r^2)\,dr.
\end{equation}
Let $t\in[0,1]$. Suppose without loss of generality that  $X^1_t(\omega)\le X^2_t(\omega)$.
Using \eqref{idd:s} and that $t\notin A$ by definition of $A$, we get
\begin{align*}
Y_t-B_t^H&=X^1_t-B_t^H=\int_0^t b(X_r^1)\,dr=
\int_A b(X_r^1)\,dr+\int_{[0,t]\setminus A} b(X_r^1)\,dr\\
&=\int_A b(X_r^2)\,dr+\int_{[0,t]\setminus A} b(X_r^1)\,dr\\
&=\int_{[0,t]}b(Y_r)\,dr,
\end{align*}	
and thus $Y$ solves \eqref{toyeq}.

Unfortunately, this argument becomes inapplicable  when $b$ is a distribution. Indeed, let $(b^n)_{n\in\Z_+}$ be an arbitrary sequence of functions converging to $b$ in $\C^{\alpha-}$. Then it is easy to show that for a certain subsequence $(n_m)_{m\in\Z_+}$
$$
\lim_{m\to \infty }\int_{[s_k,t_k]} b^{n_m}(X_r^1)\,dr=
\lim_{m\to \infty }\int_{[s_k,t_k]} b^{n_m}(X_r^2)\,dr.
$$
However, this does not imply that $
\lim_{m\to \infty }\int_{A} b^{n_m}(X_r^1)\,dr$ equals 
$\lim_{m\to \infty }\int_{A} b^{n_m}(X_r^2)\,dr$. Indeed, a sum of countably many numbers each converging to $0$ does not necessarily converge to $0$.

Therefore, we fix additionally $K\in\N$ and consider a set $A^K:=\cup_{k\le K}(s_k,t_k)\subset A$. We introduce a process $Y^K_t:=X^1_t\1(t\in [0,1]\setminus A^K)+X^2_t\1(t\in A^K)$. Since $Y^K$ differs from $X^1$ (and hence from $X^2$) only on finitely many intervals, the above argument remains valid, and one can show that $Y^K$ solves \eqref{toyeq}. Note, though, that $Y^K$ is not necessarily adapted; therefore, we cannot conclude that $Y^K$ is a weak solution to \eqref{toyeq}. Nevertheless, we have  for any fixed $K\in\Z_+$
\begin{equation}\label{conv:s} 
\int_0^t b^{n_m}(Y_r^K)\,dr\to Y^K_t-B_t^H  \quad \text{as $m\to\infty$}.
\end{equation}
By definition we have $Y^K_t\to Y_t$ a.s. as $K\to\infty$. Since for any fixed $m$ we have also 
 \begin{equation*}
\int_0^t b^{n_m}(Y_r^K)\,dr\to \int_0^t b^{n_m}(Y_r)\,dr  \quad \text{as $K\to\infty$},
\end{equation*}
we can conclude that 
 \begin{equation*}
\int_0^t b^{n_m}(Y_r)\,dr\to Y_t-B_t^H  \quad \text{as $m\to\infty$},
\end{equation*}
as long as we can show that the convergence in \eqref{conv:s} is uniform over $K$. This is not entirely trivial as the process $Y^K$ is not adapted, and thus the stochastic sewing lemma is not applicable. Therefore, we rely here on the deterministic sewing lemma, which allows us to treat non-adaptive processes and nonlinear Young integral-type bounds. The key bound is established in \cref{l:noreg}, and this strategy is formalized in \cref{l:sud1}.

\section{Preliminaries and auxiliary bounds}\label{s:auxbound}

We  introduce further necessary notation which will be used in the article.  We denote by $\Lip(\R^d,\R^d)$ the space of all Lipschitz functions $\R^d\to\R^d$ equipped with the usual norm and seminorm 
\begin{equation*}
	[f]_{\Lip(\R^d)}:=\sup_{x,y\in\R^d}\frac{|f(x)-f(y)|}{|x-y|};\qquad 
	\|f\|_{\Lip(\R^d)}:=\|f\|_{\C(\R^d)}+[f]_{\Lip(\R^d)}.
\end{equation*}	

 For $0\le S< T$ let  $\Delta_{[S,T]}$ be the simplex
\begin{equation*}
	\Delta_{[S,T]}:=\{(s,t)\in[S,T]^2\colon s\le t\}.
\end{equation*}

If $A_{\cdot,\cdot}$ is a function $\Delta_{[S,T]}\to\R^d$, where $0\le S\le T$, then we put as standard
\begin{equation*}
	\delta A_{s,u,t}:= A_{s,t}-A_{s,u}-A_{u,t},\quad  S\le s \le u \le t \le T. 
\end{equation*}

If $f$ is a function $[S,T]\to\R^d$, where $0\le S\le T$, then its $\ell$-variation, $\ell\in[1,\infty)$, is denoted as usual  by
$$
[f]_{\ell-\var;[S,T]}:=\bigl(\sup_\Pi\sum_{i=0}^{n-1} |f(t_{i+1})-f(t_i)|^\ell\bigr)^{\frac1\ell},
$$
where the supremum is taken over all partitions $\Pi=\{t_0=S,t_1,\dots, t_n=T\}$ of the interval $[S,T]$. The space of functions $f\colon [S,T]\to\R^d$ with finite $\ell\in[1,\infty)$ variation will be denoted by $\C^{\ell-\var}([S,T],\R^d)$. For a partition $\Pi$ as above we denote its diameter by $|\Pi|:=\max_i |t_{i+1}-t_i|$.

A continuous function $w\colon \Delta_{[S,T]}\to\R_+$ is called a \textit{control} if $w(s,s)=0$ for any $s\in[S,T]$ and for any $(s,t)\in\Delta_{[S,T]}$, $u\in[s,t]$ one has $w(s,u)+w(u,t)\le w(s,t)$, see \cite[Definition~1.6]{FV2010}. It is easy to see that the functions $w(s,t):=t-s$ and $w(s,t):=[f]_{\ell-\var;[s,t]}^\ell$, where $f\in\C^{\ell-\var}([S,T])$, are examples of controls.

If a filtration $(\F_t)_{t\in[0,T]}$ is given, then the conditional expectation given $\F_t$, $t\in[0,T]$, will be denoted by $\E^t[\cdot]:=\E[\cdot|\F_t]$.

As in \cite{BDG,ABLM}, we introduce the following family of norms and seminorms to analyze solutions of \Gref{x;b}. If $f\colon[S,T]\times\Omega\to\R^d$ is a measurable function, then for 
$\gamma\in(0,1]$, $m\ge1$ we define
\begin{align}\label{fullsnormsde}
&\|f\|_{\C^0L_m([S,T])}:=\sup_{s\in[S,T]}\|f_s\|_{\lm};\qquad 
[f]_{\C^\gamma L_m([S,T])}:=\sup_{(s,t)\in\Delta_{[S,T]}}\frac{\|f_t-f_s\|_{L_m(\Omega)}}{|t-s|^\gamma};\\
&\|f\|_{\C^\gamma L_m([S,T])}=\|f\|_{\C^0 L_m([S,T])}+[f]_{\C^\gamma L_m([S,T])}.\nn
\end{align}	

The space-time convolution of the heat kernel with the white noise will be denoted by~$V$:
\begin{equation}\label{def.V}
	V_t(x):=\int_0^t\int_{D}p_{t-r}(x,y)W(dr,dy),\quad t\ge0,\,x\in D.
\end{equation}

The following norms and seminorms will be important for the analysis of \Sref{u_0;b}. For a measurable function $f\colon[S,T]\times D\times \Omega\to\R$ and $\gamma\in(0,1]$, 
$m\ge1$ we write
\begin{align}\begin{split}
	&\|f\|_{\Ctimespacezerom{[S,T]}}:=\sup_{t\in[S,T]}\sup_{x\in D}\|f(t,x)\|_{L_m(\Omega)};\label{fullsnorm}\\
	&[f]_{\Ctimespace{\gamma}{m}{[S,T]}}:=\sup_{(s,t)\in\Delta_{[S,T]}}\sup_{x\in D} \frac{\|f_t(x)-P_{t-s}f_s(x)\|_{L_m(\Omega)}}{|t-s|^\gamma};\\	
	&\|f\|_{\Ctimespace{\gamma}{m}{[S,T]}}:=\|f\|_{\Ctimespacezerom{[S,T]}}+[f]_{\Ctimespace{\gamma}{m}{[S,T]}}.
	\end{split}
\end{align}

In the proofs we often use a standard heat kernel estimate
\begin{equation}\label{gaussian}
	\|G_t^d f\|_{\C(\R^d)}\le C t^{\frac{\beta}2}\|f\|_{\C^\beta},
\end{equation}
where $\beta<0$, $f\in\C^\beta$ and $C=C(\beta)$, see, e.g., \cite[Proposition~3.7(i)]{BDG}.

\subsection{Sewing lemmas}

As mentioned earlier, our proofs frequently utilize both the deterministic and stochastic sewing lemmas. For the convenience of the reader, we recall them here. 

We begin with the deterministic sewing lemma with controls.	
\begin{proposition}[{Sewing lemma, \cite[Theorem~2.2]{FZ18}}]\label{t:dSL}
	Let  $0\le S\le T$.  Let  $A$ be a function $\Delta_{[S,T]}\to\R^d$. Assume that there exist a control $w$ on the same simplex and a constant $\eps>0$ such that for every $(s,t)\in\Delta_{[S,T]}$, $u\in[s,t]$ one has
	\begin{equation}	
		| \delta A_{s,u,t}|\le w(s,t)^{1+\eps}.\label{con:ds2}
	\end{equation}
	
	Further, suppose that there exists a process $\A\colon [S,T]\to\R^d$ such that for any $t\in[S,T]$ and any sequence of partitions $\Pi_N:=\{S=t^N_0,...,t^N_{k(N)}=t\}$ of $[S,t]$ with $\lim_{N\to\infty}\Di{\Pi_N}\to0$  one has
	\begin{equation}\label{conssl:ds3}
		\sum_{i=0}^{k(N)-1} A_{t^N_i,t_{i+1}^N}\to  \A_{t}-\A_{S}\quad\text{ as $N\to\infty$}.
	\end{equation}
	
	Then there exists a constant $C=C(\eps)$ independent of $S,T$, $w$ such that for every $(s,t)\in\Delta_{S,T}$  we have 
	\begin{equation}
		|\A_{t}-\A_{s}| \le |A_{s,t}|+ 	Cw(s,t)^{1+\eps}.
		\label{est:dssl1}
	\end{equation}
\end{proposition}

Next we recall the stochastic version of the above proposition.	
\begin{proposition}[{Stochastic sewing lemma, \cite[Theorem~2.1]{LeSSL}}]\label{T:SSLst}
	Let $m\in[2,\infty)$, $0\le S\le T$. Let $(\F_t)_{t\in[S,T]}$ be a filtration.  Let  $A$ be a function $\Delta_{[S,T]}\to L_m(\Omega,\R^d)$ and assume additionally that $A_{s,t}$ is $\F_t$--measurable. 	Suppose that there exist  constants $\Gamma_1, \Gamma_2\ge0$, $\eps_1, \eps_2 >0$ such that the following conditions hold for every $(s,t)\in\Delta_{[S,T]}$ and $u\in[s,t]$
	\begin{align}	
		&\| A_{s,t}\|_{L_m(\Omega)}\le \Gamma_1|t-s|^{\frac12+\eps_1},\label{con:s1}	\\
		&\|\E^s[ \delta A_{s,u,t}]\|_{L_m(\Omega)}\le \Gamma_2|t-s|^{1+\eps_2}.\label{con:s2}
	\end{align}
	
	Further, suppose that there exists a process $\A\colon\Omega\times[S,T]\to\R^d$ such that for any $t\in[S,T]$ and any sequence of partitions $\Pi_N:=\{S=t^N_0,...,t^N_{k(N)}=t\}$ of $[S,t]$ with $\lim_{N\to\infty}\Di{\Pi_N}\to0$  one has
	\begin{equation}\label{con:s3}
		\sum_{i=0}^{k(N)-1} A_{t^N_i,t_{i+1}^N}\to  \A_{t}-\A_{S}\quad\text{in probability as $N\to\infty$}.
	\end{equation}
	
	Then there exists a constant $C=C(\eps_1,\eps_2,m)$ independent of $S,T$, $\Gamma_i$ such that for every $(s,t)\in\Delta_{S,T}$  we have 
	\begin{equation}
		\|\A_{t}-\A_{s}\|_{L_m(\Omega)} \le C \Gamma_1 |t-s|^{\frac12+\eps_1}+
		C \Gamma_2 |t-s|^{1+\eps_2}.
		\label{est:ssl1}
	\end{equation}
\end{proposition}

\subsection{Bounds for integrals of a general stochastic process}\label{s:gen}
In this subsection, we collect a number of technical estimates for integral functionals of 
a general stochastic process $Y$ satisfying certain smoothing conditions. Later we apply these results in order 
to study functionals of fractional Brownian motion and of a convolution of space-time white noise with the Gaussian kernel. The bounds presented here are obtained using stochastic sewing lemma and the calculations are similar in spirit to the bounds found in \cite{CG16}, \cite{LeSSL}, \cite{GG22}, \cite{BDG}, \cite{ABLM} and so on. Since we were not able to find these exact statements in the literature, we provide short proofs.

We fix $d\in\N$ and the length of the time interval $T>0$. Let the space domain $\D\subset\R$ be either a single point, closed interval, or $\R$. We fix the convolution  kernel $s\colon [0,T]\times\D\times\D\to\R_+$ which for any $t>0$ is bounded, continuous, and satisfies
\begin{equation}\label{convprop}
	\int_\D s_t(x,y)\,dy =1,\qquad 	\int_\D s_r(x,y)s_{t-r}(y,z)\,dy =s_{t}(x,z),
\end{equation}
where $(r,t)\in\Delta_{[0,T]}$, $x,z\in\D$. 

\begin{convention}\label{c:convdelta}
In case when $\D$ is a single point $x$, $dy$ stands for the Dirac delta measure concentrated at $x$.
\end{convention}	

We introduce  the corresponding semigroup $S_t\phi(x):=\int_\D s_t(x,y) \phi(y)\,dy$, where $\phi$ is a measurable function $\D\to\R^d$. Later we will take $\D=\{0\}$, $s_t(0,0)=1$ (see \cref{c:convdelta}) for the fractional Brownian motion case and $\D=D$, $s_t(x,y)=p_t(x,y)$ (see \cref{c:conv}) for the space-time white noise case.

Fix a smoothing parameter $\kappa\in(0,1)$.  
Let $Y\colon[0,1]\times\D\times\Omega\to\R^d$ be a stochastic process adapted to the given filtration 
$(\F_t)_{t\in[0,T]}$ (that is, $\sigma\bigl(Y(s,x), s\le t, x\in\D\bigr)\subset \F_t$ for $t\in[0,T]$). In the whole subsection we suppose that $Y$ satisfies the following condition: for any $\beta<0$, $\mu\in(0,1]$ there exists a constant $C=C(\beta,\mu)$ such that for any bounded continuous function 
$f\colon\R^d\to\R$, $(s,t)\in\Delta_{[0,T]}$, $x\in \D$, $\F_s$-measurable random vectors $\xi$, $\eta$ one has a.s.
\begin{align}\label{Ycond}\tag{$\mathbb{Y}_\kappa$}\begin{split}
		&|\E^s f(Y_t(x)+\xi)|\le C \|f\|_{\C^\beta}(t-s)^{\beta \kappa};\\
		&	\bigl|\E^s [f(Y_t(x)+\xi)-f(Y_t(x)+\eta)]\bigr|\le C \|f\|_{\C^\beta}|\xi-\eta|^{\mu}(t-s)^{(\beta-\mu)\kappa}.	
\end{split}\end{align}

In this subsection we use the following norms and seminorms which extend both \eqref{fullsnormsde} and \eqref{fullsnorm}.
For $0\le S<T$, a measurable function $f\colon[S,T]\times \D\times \Omega\to\R$ and $\gamma\in(0,1]$ we define
\begin{align*}
	&\|f\|_{\Ctimespacedzerom{[S,T]}}:=\sup_{t\in[S,T]}\sup_{x\in \D}\|f(t,x)\|_{L_m(\Omega)};\\
	&[f]_{\Ctimespaced{\gamma}{m}{[S,T]}}:=\sup_{(s,t)\in\Delta_{[S,T]}}\sup_{x\in \D} \frac{\|f_t(x)-S_{t-s}f_s(x)\|_{L_m(\Omega)}}{|t-s|^\gamma}.
\end{align*}

First, let us present a simple condition which ensures that \eqref{Ycond} holds.
\begin{lemma}\label{l:gycond}
	Assume that for any $(s,t)\in\Delta_{[0,T]}$, $x\in\D$ we have
	\begin{equation}\label{ind}
		Y_t(x)-\E^s Y_t(x) \text{ is Gaussian and independent of $\F_s$}.
	\end{equation}
	If additionally all $d$ components of $Y_t(x)-\E^s Y_t(x)=(Y^i_t(x)-\E^s Y^i_t(x), i=1,\hdots,d)$ are i.i.d with 
	\begin{equation}\label{sigmaY}
		\sigma^2(s,t,x):=\Var (Y_t^{i}(x)-\E^s Y_t^{i}(x))\ge C(t-s)^{2\kappa},\quad i=1,...,d,
	\end{equation}	
	for some $\kappa>0$, then \eqref{Ycond} holds with this $\kappa$.
\end{lemma}
\begin{proof}
	Fix $\beta<0$, $\mu\in(0,1]$. 
	It follows from the assumptions of the lemma that for any bounded continuous function 
	$f\colon\R^d\to\R$, $(s,t)\in\Delta_{[0,T]}$, $x\in \D$, $\F_s$-measurable random vector $\xi$ we have
	\begin{equation}\label{reprm}
		\E^s f(Y_t(x)+\xi)=G^d_{\sigma^2(s,t,x)} f(\E^s Y_t(x)+\xi),
	\end{equation}		
	where we used that the vector $Y_t(x)-\E^s Y_t(x)$ is Gaussian and independent from $\F_s$.
	Therefore, using this identity, \eqref{sigmaY} and \eqref{gaussian}, we deduce for any $(s,t)\in\Delta_{[0,T]}$
	\begin{equation*}
		| \E^s f(Y_t(x)+\xi)|\le \|G^d_{\sigma^2(s,t,x)}f\|_{\C(\R^d)}\le C \|f\|_{\C^\beta}(t-s)^{\beta \kappa},
	\end{equation*}
	where $C=C(\beta)>0$.
	
	Similarly, if $\eta$ is another $\F_s$-measurable $d$-dimensional random vector, we denote $h(x):=f(x+\xi)-f(x+\eta)$. Using again \eqref{reprm}, we have
	\begin{align*}
		\bigl| \E^s [f(Y_t(x)+\xi)-f(Y_t(x)+\eta)]\bigr|&=\bigl| G^d_{\sigma^2(s,t,x)} f(Y_t(x)+\xi)-G^d_{\sigma^2(s,t,x)} f(Y_t(x)+\eta)|\\
		&=\bigl| G^d_{\sigma^2(s,t,x)} h(\E^s Y_t(x))\bigr|\le \| G^d_{\sigma^2(s,t,x)} h\|_{\C(\R^d)}\\\
		&\le C \|h\|_{\C^{\beta-\mu}}(t-s)^{(\beta-\mu) \kappa}\le C \|f\|_{\C^\beta}(t-s)^{(\beta-\mu) \kappa}|\xi-\eta|^\mu,
	\end{align*}
	which establishes  \eqref{Ycond}.
\end{proof}

Next, we obtain a bound showing that the integral involving the process $Y$ can be approximated by a certain Riemann sum.

\begin{proposition}\label{l:sappr}
	Let $f\colon\R^d\to\R$, $w\colon[0,T]\to\R_+$ be bounded continuous functions. Let $\phi\colon 
	[0, T]\times \D \times \Omega\to\R^d$ be a measurable process adapted to the filtration $(\F_t)_{t\in[0,T]}$. Suppose that for some $\eps>0$, $(s,t)\in\Delta_{[0,T]}$ one has 
	\begin{equation}\label{stechncond}
		\|f\|_{\C^\eps}<\infty;\qquad  \|w\|_{\C([s,t])}<\infty;\qquad  [\phi]_{\Ctimespaced{\eps}{1}{[s,t]}} <\infty.	
	\end{equation}
	Then for any  $x\in\D$, and any sequence of partitions $\Pi_N:=\{s=t^N_0,...,t^N_{k(N)}=t\}$ of $[s,t]$ with $\lim_{N\to\infty}\Di{\Pi_N}\to0$  one has
	\begin{align}\label{spde:s3}
		&\sum_{i=0}^{k(N)-1}\int_{t^N_i}^{t_{i+1}^N}\int_\D s_{t-r}(x,y)w_r \E^{t^N_i}[f(Y_r(y)+S_{r-t^N_i}\phi_{t^N_i}(y))]\,dydr\nn\\
		&\qquad \to \int_s^t\int_\D s_{t-r}(x,y)w_r f(Y_r(y)+\phi_r(y))\,dydr\quad\text{in probability as $N\to\infty$}.
	\end{align}
\end{proposition}	
The proof of this proposition is elementary though a bit technical. We place it in the Appendix.

Now we are ready to prove the first regularization lemma. 

\begin{lemma}\label{l:finfin}
	Let  $m\ge2$, $\beta\in(-\frac1{2\kappa},0)$, $\tau\in(0,1]$. Let $f\colon\R^d\to\R$, $w\colon[0,T]\to\R_+$ be bounded continuous functions.
	Let $\phi,\psi\colon [0, T]\times\D\times\Omega\to\R^d$ be measurable processes adapted to the filtration $(\F_t)_{t\in[0,T]}$.
	Assume that
	\begin{equation}\label{paramcond}
		\beta>1-\frac{\tau}\kappa.
	\end{equation}
	Then there exists a constant $C=C(m,\beta,\tau)$ such that for any $(s,t)\in\Delta_{[0,T]}$, $x\in\D$ we have 	
	\begin{align}\label{sb1}
		&\Bigl\|\int_s^t\int_\D s_{t-r}(x,y) w_rf(V_r(y)+\phi_r(y))\,dydr\Bigr\|_{L_m(\Omega)}\nn\\
		&\qquad\le
		C \|f\|_{\C^\beta}\|w\|_{\C([s,t])}
		(t-s)^{1+\beta\kappa}\bigl(1+ [\phi]_{\Ctimespaced{\tau}{m}{[s,t]}}(t-s)^{\tau-\kappa}\bigr).
	\end{align}
	
	If, additionally, for some $\mu\in(0,1]$ we have 
	\begin{equation}\label{paramcond2}
		\beta>\mu-\frac1{2\kappa},
	\end{equation}
	then there exists a constant $C=C(m,\beta,\mu,\tau)$ such that for any $(s,t)\in\Delta_{[0,T]}$, $x\in\D$ we have 
	\begin{align}
		&\Bigl\|\int_s^t\int_\D  s_{t-r}(x,y)w_r(f(V_r(y)+\phi_r(y))-f(V_r(y)+\psi_r(y)))\,dydr\Bigr\|_{L_m(\Omega)}\nn\\
		&\quad\le
		C \|f\|_{\C^\beta}\|w\|_{\C([s,t])}
		(t-s)^{1+(\beta-\mu)\kappa}\|\phi-\psi\|^\mu_{\Ctimespacedzerom{[s,t]}}\nn\\
		&\qquad+C \|f\|_{\C^\beta}\|w\|_{\C([s,t])}
		(t-s)^{1+(\beta-1)\kappa+\tau}\bigl([\phi]_{\Ctimespaced{\tau}{m}{[s,t]}}+[\psi]_{\Ctimespaced{\tau}{m}{[s,t]}}\bigr).\label{sb2}
	\end{align}
\end{lemma}

\begin{proof}
	We begin with the proof of \eqref{sb1}. Fix $0\le S\le T$, $x\in\D$. Without loss of generality, we assume that $ [\phi]_{\Ctimespaced{\tau}{m}{[S,T]}}<\infty$ and 
	$ \|w\|_{\C([S,T])}<\infty$. Otherwise, the right-hand side of \eqref{sb1} is infinite and 
	\eqref{sb1} automatically holds. 
	
	First, we suppose additionally that $f$ is smooth enough so that $\|f\|_{\C^1}<\infty$.	
	We apply the stochastic sewing lemma (\cref{T:SSLst}) to the processes
	\begin{align}\label{Aprocess}\begin{split}
			A^\phi_{s,t}&:=\int_s^t\int_\D s_{t-r}(x,y)w_r \E^sf(V_r(y)+S_{r-s}\phi_s(y))\,dydr;\\
			\A^\phi_{t}&:=\int_s^t\int_\D s_{t-r}(x,y)w_r f(V_r(y)+\phi_r(y))\,dydr.\end{split}
	\end{align}
	where $(s,t)\in\Delta_{[S,T]}$. 
	Since $\phi_s$ is $\F_s$-measurable, we can use \eqref{Ycond} with $\xi=S_{r-s}\phi_s(y)$ and  to derive	
	\begin{equation*}
		|A^\phi_{s,t}|\le  C \|f\|_{\C^{\beta}}\|w\|_{\C([S,T])}\int_s^t\int_\D s_{t-r}(x,y)	(t-s)^{\beta\kappa}\,drdy\le  C \|f\|_{\C^{\beta}}\|w\|_{\C([S,T])}(t-s)^{1+\beta\kappa},
	\end{equation*}	
	where $(s,t)\in \Delta_{[S,T]}$ and $C=C(\beta)$. Here we used that the kernel $s_{t-r}$ integrates to $1$ by \eqref{convprop}. This implies 
	\begin{equation}\label{astb1}
		\|A^\phi_{s,t}\|_{\lm}\le 	C \|f\|_{\C^{\beta}}\|w\|_{\C([S,T])}(t-s)^{1+\beta\kappa},\quad  (s,t)\in \Delta_{[S,T]}.
	\end{equation}	
	Hence condition \eqref{con:s1} holds thanks to the assumption $1+\beta \kappa >1/2$.
	
	In a similar manner, we derive 
	\begin{align}\label{asutb0}
		|\E^s \delta A^\phi_{s,u, t}|&=\Bigl|\E^s \int_u^t\int_\D s_{t-r}(x,y)w_r   \E^u[f(V_r(y)+S_{r-s}\phi_s(y))-f(V_r(y)+S_{r-u}\phi_u(y))]\,dydr\Bigr|\nn\\
		&\le C \|f\|_{\C^{\beta}}\|w\|_{\C([S,T])}\int_u^t\int_\D s_{t-r}(x,y) |S_{r-s}\phi_s(y)-S_{r-u}\phi_u(y)|	(r-u)^{(\beta-1)\kappa}\,dydr
	\end{align}
	for $C=C(\beta)$. Here we used that both random variables $S_{r-s}\phi_s(y)$ and $S_{r-u}\phi_u(y)$ are $\F_u$-measurable and applied \eqref{Ycond} with $\mu=1$. By Jensen's inequality and the convolution property \eqref{convprop}
	\begin{equation*}
		\|S_{r-s}\phi_s(y)-S_{r-u}\phi_u(y)\|_{\lm}\le\sup_{z\in\D}\|S_{u-s}\phi_s(z)-\phi_u(z)\|_{\lm}\le [\phi]_{\Ctimespaced{\tau}{m}{[S,T]}}(u-s)^{\tau}
	\end{equation*}
	for $S\le s\le u\le r\le T$  and $y\in\D$. Substituting this back into \eqref{asutb0} and applying Jensen's inequality once again, we get
	\begin{equation}\label{asutb1}
		\|\E^s \delta A^\phi_{s,u, t}\|_{L_m(\Omega)}\le  C \|f\|_{\C^{\beta}}\|w\|_{\C([S,T])}[\phi]_{\Ctimespaced{\tau}{m}{[S,T]}}	(t-s)^{1+(\beta-1)\kappa+\tau}.
	\end{equation}
	Here we used that the integral of the kernel over the domain $\D$ is $1$ and that singularity $(r-u)^{(\beta-1)\kappa}$ is integrable by \eqref{paramcond}. Applying \eqref{paramcond} once again, we get $1+(\beta-1) H+\tau>1$ and thus condition \eqref{con:s2} is satisfied.
	
	Finally, the last condition of the stochastic sewing lemma, condition \eqref{con:s3}, holds by \cref{l:sappr}; recall that we assumed  $\|f\|_{\C^1}\!<\!\infty$, $\|w\|_{\C([S,T])}\!<\infty$, and 
	${[\phi]_{ \Ctimespaced{\tau}{m}{[S,T]}}<\infty}$
	which yields \eqref{stechncond}.
	
	Thus, all the conditions of \cref{T:SSLst} are satisfied and \eqref{sb1} follows from \eqref{est:ssl1} and \eqref{astb1}, \eqref{asutb1}. 
	
	To obtain  \eqref{sb1} for general $f$, we note that for any $\eps>0$ the function 
	$f_\eps:=G^d_{\eps}f$, satisfies  $\|f_\eps\|_{\C^1}<\infty$. Therefore, \eqref{sb1} holds for $f_\eps$. Since $f_\eps$ converges pointwise to $f$ as $\eps\to0$  and $\|f_\eps\|_{\C^\beta}\le\|f\|_{\C^\beta}$, by Fatou's lemma we see  that \eqref{sb1} holds also for $f$.
	
	Now we move on to the proof of \eqref{sb2}. We use very similar ideas as in the proof above. 
	We also note that, without loss of generality, we can assume that $[\phi]_{ \Ctimespaced{\tau}{m}{[S,T]}}<\infty$, $[\psi]_{\Ctimespaced{\tau}{m}{[S,T]}}<\infty$ and $ \|w\|_{\C([S,T])}<\infty$. We suppose additionally that $\|f\|_{\C^1}<\infty$ and  apply \cref{T:SSLst} to the processes  
	\begin{equation*}
		A_{s,t}:=A^\phi_{s,t}-A^\psi_{s,t};\quad 	\A_{t}:=A^\phi_t-A^\psi_t,\qquad (s,t)\in \Delta_{[S,T]},
	\end{equation*}
	where the processes $A^\phi$, $\A^\phi$ were introduced in \eqref{Aprocess} and $A^\psi$, $\A^\psi$ are defined in exactly the same way with $\phi$ replaced by $\psi$. 
	
	We note that $S_{r-s}\phi_s(y)$ and $S_{r-s}\psi_s(y)$ are $\F_s$ measurable. Therefore, we use \eqref{Ycond} to derive for $(s,t)\in\Delta_{[S,T]}$
	\begin{equation*}
		|A_{s,t}|\le C\|f\|_{\C^\beta}\|w\|_{\C([S,T])}\int_s^t (r-s)^{(\beta-\mu)\kappa} \int_\D s_{t-r}(x,y)|S_{r-s}\phi_s(y)-S_{r-s}\psi_s(y)|^\mu\,dydr,
	\end{equation*}
	where $C=C(\beta,\mu)>0$. 
	Thanks  to \eqref{paramcond2}, the singularity $(r-s)^{(\beta-\mu)\kappa}$ is integrable. Hence, applying Jensen's inequality and using \eqref{paramcond}, we get for $m\ge2$
	\begin{equation}\label{s:spdec1}
		\|A_{s,t}\|_{L_m(\Omega)}\le  C\|f\|_{\C^\beta}\|w\|_{\C([S,T])}(t-s)^{1+(\beta-\mu)\kappa}\|\phi-\psi\|^\mu_{\Ctimespacedzerom{[S,T]}},
	\end{equation}
	where $C=C(\beta,\mu)>0$. We stress that $C$ does not depend on $x\in D$. By \eqref{paramcond2}, we have  $1+(\beta-\mu)\kappa>1/2$. Therefore condition \eqref{con:s1} of \cref{T:SSLst} holds.

	Next, using \eqref{asutb1}, we easily get for $(s,t)\in \Delta_{[S,T]}$, $u\in[s,t]$
	\begin{align}\label{asutb2}
		\|\E^s \delta A_{s,u, t}\|_{L_m(\Omega)}&\le \|\E^s \delta A^\phi_{s,u, t}\|_{L_m(\Omega)}+\|\E^s \delta A^\psi_{s,u, t}\|_{L_m(\Omega)}\nn\\
		&\le  C \|f\|_{\C^{\beta}}\|w\|_{\C([S,T])}	(t-s)^{1+(\beta-1)\kappa+\tau}([\phi]_{\Ctimespaced{\tau}{m}{[S,T]}}+[\psi]_{\Ctimespaced{\tau}{m}{[S,T]}}),
	\end{align}
	where $C=C(\beta,\mu)$. We see from \eqref{paramcond} $1+(\beta-1) \kappa+\tau>1$ and thus \eqref{con:s2} holds.

	Condition \eqref{con:s3} follows from the definitions of $A_{s,t}$, $\A_t$ and \cref{l:sappr}.
	
	Thus all the conditions of  \cref{T:SSLst} holds and the desired bound \eqref{sb2} follows from \eqref{est:ssl1} and \eqref{s:spdec1}, \eqref{asutb2}.
	
	The proof of \eqref{sb2}  for a general bounded continuous $f$ follows from the above argument by approximation, exactly as in the proof of \eqref{sb1}.
\end{proof}

In the paper we will use as a weight $w$ in the above lemma only two options: either the trivial weight $w\equiv1$ or an exponential weight $w_r:=\exp(\lambda(t-r))$, where $\lambda\gg1$ is large. The next lemma shows how in the latter case the right-hand sides of \eqref{sb1}, \eqref{sb2} can be modified to become small for large $\lambda$.

\begin{corollary}\label{c:finfin}
	Let  $m\ge2$, $\beta\in(-\frac1{2\kappa},0)$, $\tau\in[0,1]$, $\eps\in(0,1/2)$. Let $f\colon\R^d\to\R$ be a bounded continuous function.
	Let $\phi,\psi\colon [0, 1]\times\D\times\Omega\to\R^d$ be measurable processes adapted to the filtration $(\F_t)_{t\in[0,1]}$.
	Assume that \eqref{paramcond} holds. Then there exists a constant $C=C(m,\beta, \eps,\tau)>0$ such that for any $\lambda>1$, $(s,t)\in\Delta_{[0,1]}$ and $x\in \D$ we have
	\begin{align}\label{bound1terml}
		&\Bigl\|\int_s^t\int_\D e^{-\lambda(t-r)}s_{t-r}(x,y)f(Y_r(y)+\phi_r(y))\,dydr\Bigr\|_{L_m(\Omega)}\nn\\
		&\qquad\le
		C \|f\|_{\C^\beta}(t-s)^\eps\lambda^{-1-\beta \kappa+3\eps}\Bigl(1+[\phi]_{\Ctimespaced{\tau}{m}{[s,t]}}\lambda^{\kappa-\tau}\Bigr).
	\end{align}
	If, additionally, for $\mu\in[0,1]$ condition \eqref{paramcond2} is satisfied, 
	then there exists a constant $C=C(m,\beta,\eps, \mu,\tau)>0$ such that for any $(s,t)\in\Delta_{[0,1]}$,
	$x\in \D$ we have 
	\begin{align}\label{ourboundasregl}
		&\Bigl\|\int_s^t\int_\D e^{-\lambda(t-r)}s_{t-r}(x,y)(f(Y_r(y)+\phi_r(y))-f(Y_r(y)+\psi_r(y)))\,dydr\Bigr\|_{L_m(\Omega)}\nn\\
		&\quad\le
		C \|f\|_{\C^\beta}
		(t-s)^\eps\lambda^{-1-(\beta-\mu)\kappa+3\eps}\|\phi-\psi\|^\mu_{\Ctimespacedzerom{[s,t]}}\nn\\
		&\qquad+C\|f\|_{\C^\beta} 	(t-s)^\eps \lambda^{-1-(\beta-1)\kappa-\tau+3\eps}\bigl([\phi]_{ \Ctimespaced{\tau}{m}{[s,t]}}+[\psi]_{ \Ctimespaced{\tau}{m}{[s,t]}}\bigr).
	\end{align}
\end{corollary}
\begin{proof}
	We deduce  only \eqref{bound1terml} from \eqref{sb1}. Bound \eqref{ourboundasregl} is derived from \eqref{sb2} using exactly the same argument.
	
	We fix $\eps\in(0,1/2)$, $\lambda>1$ and split the integral into two
	\begin{align}\label{difftwoguys}
		&\Bigl\|\int_s^t\int_\D e^{-\lambda(t-r)}s_{t-r}(x,y)f(Y_r(y)+\phi_r(y))\,dydr\Bigr\|_{L_m(\Omega)}\nn\\
		&\qquad\le
		\Bigl\|\int_s^{(t-\lambda^{-1+\eps})\vee s }\int_\D e^{-\lambda(t-r)}s_{t-r}(x,y)f(Y_r(y)+\phi_r(y))\,dydr\Bigr\|_{L_m(\Omega)}\nn\\
		&\qquad\phantom{\le}+\Bigl\|\int_{(t-\lambda^{-1+\eps})\vee s}^t \int_\D e^{-\lambda(t-r)}s_{t-r}(x,y)f(Y_r(y)+\phi_r(y))\,dydr\Bigr\|_{L_m(\Omega)}\nn\\
		&\qquad=:I_1+I_2.
	\end{align}
	If $r\le t-\lambda^{-1+\eps}$, then $e^{-\lambda(t-r)}\le e^{-\lambda^\eps}$ and $\|e^{-\lambda (t-\cdot)}\|_{\C([s,(t-\lambda^{-1+\eps})\vee s])}\le e^{-\lambda^\eps}$. Therefore, since $1+\beta \kappa>1/2>\eps$, \eqref{sb1} implies 
	\begin{equation}\label{ionel}
		I_1\le C e^{-\lambda^\eps} \|f\|_{\C^\beta}
		(t-s)^{\eps}(1+  [\phi]_{\Ctimespaced{\tau}{m}{[s,t]}}).
	\end{equation}
	for $C=C(m,\beta,\tau)$.
	
	To bound $I_2$, we use the obvious inequality  $\|e^{-\lambda (t-\cdot)}\|_{\C([t-\lambda^{-1+\eps},t])}\le1$. 
	Clearly, we also have for any $\gamma\in[\eps,2]$
	$$
	(t-((t-\lambda^{-1+\eps})\vee s))^{\gamma}\le (t-s)^\eps  \lambda^{(-1+\eps)(\gamma-\eps)}\le
	(t-s)^\eps  \lambda^{-\gamma+3\eps}.
	$$
	Applying this inequality with $\gamma=1+\beta\kappa$ and $\gamma=1+\beta\kappa+\tau-\kappa$, we derive from \eqref{sb1} 
	\begin{equation*}
		I_2\le C \|f\|_{\C^\beta}(t-s)^{\eps}\lambda^{-1-\beta \kappa+3\eps}\bigl(1+ [\phi]_{\Ctimespaced{\tau}{m}{[s,t]}}\lambda^{\kappa-\tau}\bigr)
	\end{equation*}
	for $C=C(m,\beta,\tau)$. Combining this with \eqref{ionel} and substituting into \eqref{difftwoguys}, we obtain \eqref{bound1terml}; here we used that the function $\lambda\mapsto e^{-\lambda^\eps}$ decreases to $0$ much faster than any negative polynomial function as $\lambda\to\infty$.
\end{proof}

The final technical bound is an a priori bound on the norm of a solution to an SDE with additional forcing.
\begin{lemma}\label{l:apriori}
Let $m\ge2$, $\beta\in(\frac12-\frac1{2\kappa},0)$. Let $z\colon \D\to\R^d$ be a bounded measurable function. Let $(f_n)_{n\in\Z_+}$ be a sequence of bounded continuous functions $\R^d\to\R^d$.  Let $Z,A, \rho\colon[0,1]\times \D\times\Omega\to\R$ be measurable processes adapted to the filtration $(\F_t)_{t\in[0,1]}$ which satisfy 
\begin{equation}\label{zequ}
Z(t,x)=S_t z(x)+ A(t,x) +\int_0^t \int_{\D} s_{t-r}(x,y) \rho(r,y)\,dydr + Y(t,x), 
\end{equation}
where $t\in[0,1]$, $x\in \D$ and
\begin{equation*}
	\sup_{t\in[0,T]}\sup_{x\in \D}\Bigl|A(t,x)- \int_0^t\int_\D p_{t-r}(x,y) f_n(Z_r(y))\,dy dr\Bigr|\to0 \quad\text{in probability as } n\to\infty.
\end{equation*}	
Suppose that
\begin{equation}\label{zcond}
[Z-Y]_{\Ctimespaced{\frac12+\frac{\kappa}2}{m}{[0,1]}}<\infty.
\end{equation} 
Then there exists $C=C(m,\beta)>0$ such that
\begin{equation}\label{claim}
[Z-Y]_{\Ctimespaced{1+\beta\kappa}{m}{[0,1]}}\le  C (1+\sup_n\|f_n\|_{\C^\beta}^2)(1+\|\rho\|_{ \Ctimespacedzerom{[0,1]}}).	
\end{equation}
\end{lemma}

\begin{proof}
	Denote $\phi_t(x):=Z_t(x)-Y_t(x)$, $t\in[0,1]$, $x\in\D$. 
	It is clear that for any $(s,t)\in\Delta_{[0,1]}$, $x\in\D$ we have by Fatou's lemma
		\begin{align}\label{z1step}	
		&\|\phi_t(x)-S_{t-s}\phi_s(x)\|_{L_m(\Omega)}\nn\\
		&\qquad\le \|A(t,x)-S_{t-s}A(s,x)\|_{L_m(\Omega)} + \int_s^t\int_\D s_{t-r}(x,y)\|\rho(r,y)\|_{L_m(\Omega)}\,dydr\nn\\
		&\qquad \le \liminf_{n\to\infty}\Bigl\|\int_s^t\int_\D S_{t-r}(x,y) f_n(Z_r(y))\,dydr\Bigr\|_{L_m(\Omega)} + (t-s) \|\rho\|_{\Ctimespacedzerom{[0,1]}}.
	\end{align}
	To estimate the first term in the right-hand side of the above inequality, we use \cref{l:finfin}
	with $\tau=\frac12+\frac{\kappa}2$. We see that condition \eqref{paramcond} is satisfied since $\beta>\frac12-\frac1{2\kappa}$. We get from \eqref{sb1} for any $n\in\Z_+$
	\begin{equation*}
		\Bigl\|\int_s^t\int_\D s_{t-r}(x,y) f_n(Z_r(y))\,dydr\Bigr\|_{L_m(\Omega)}\le C F(t-s)^{1+\beta\kappa}(1+[\phi]_{\Ctimespaced{\frac12+\frac\kappa2}{m}{[s,t]}}(t-s)^{\frac12-\frac\kappa2})
	\end{equation*}		
	for $C=C(m,\beta)$ and $F:=\sup_n \|f\|_{\C^\beta}$. 	
	Substituting this into \eqref{z1step} we get for any $(s',t')\in\Delta_{[s,t]}$, $x\in\D$
	\begin{equation*}
	\|\phi_{t'}(x)-S_{t'-s'}\phi_s'(x)\|_{L_m(\Omega)} \le C(t'-s')^{1+\beta\kappa}\bigl(F+\|\rho\|_{\Ctimespacezerom{[0,1]}}+ F[\phi]_{\Ctimespaced{\frac12+\frac\kappa2}{m}{[s,t]}}(t-s)^{\frac12-\frac\kappa2}\bigr).
\end{equation*}

	Let $\ell>0$. Dividing both sides of the above inequality by $(t'-s')^{1+\beta \kappa}$ and taking supremum in the above inequality over all $s',t'\in \Delta_{[s,t]}$, we deduce for any $(s,t)\in\Delta_{[0,1]}$ such that $|t-s|\le \ell$
		\begin{align}\label{xminyhn}
		[\phi]_{\Ctimespaced{1+\beta\kappa}{m}{[s,t]}}&\le  C (F+\|\rho\|_{\Ctimespacezerom{[0,1]}})+C F \ell^{\frac12-\frac{\kappa}2}	[\phi ]_{\Ctimespaced{\frac12+\frac{\kappa}2}{m}{[s,t]}}\\
		&\le  C (F+\|\rho\|_{\Ctimespacezerom{[0,1]}})+C F \ell^{\frac12-\frac{\kappa}2}[\phi ]_{\Ctimespaced{1+\beta\kappa}{m}{[s,t]}},\label{xminyhn1}
	\end{align}	
	where $C=C(m,\beta)$ and we used the inequality $\frac12+\frac{\kappa}2<1+\beta\kappa$ which follows from the assumption $\beta>\frac12-\frac1{2\kappa}$. Since by assumptions of the lemma $[\phi ]_{\Ctimespaced{\frac12+\frac{\kappa}2}{m}{[s,t]}}<\infty$, it follows from \eqref{xminyhn}  that  
	\begin{equation}\label{ffff}
		[\phi]_{\Ctimespaced{1+\beta\kappa}{m}{[s,t]}}<\infty.
	\end{equation}
	
	Take now $\ell>0$ such that 
	\begin{equation*}
		C F \ell^{\frac12-\frac{\kappa}2}=\frac12.
	\end{equation*}
	Using \eqref{ffff}, we deduce from 
	\eqref{xminyhn1} that for any $(s,t)\in\Delta_{[0,1]}$ such that $|t-s|\le \ell$ we have 
	\begin{equation*}
		[\phi]_{\Ctimespaced{1+\beta\kappa}{m}{[s,t]}}\le  C (F+\|\rho\|_{\Ctimespacezerom{[0,1]}}).
	\end{equation*}	
	If $\ell>1$, then there is nothing more to prove, this inequality implies \eqref{claim}. Otherwise, if $\ell<1$, for a general $(s,t)\in\Delta_{[0,1]}$ we choose $N\ge2$ such that $1/N\le \ell < 1/(N-1)$
	and define $t_k:=s+k(t-s)/N$ for each $k=0,1,\hdots, N$. Then $t_{k+1}-t_k\le \ell$ and therefore by the above inequality and Jensen's inequality
	\begin{align*}
		\|\phi_t(x)-S_{t-s}\phi_s(x)\|_{L_m(\Omega)}&\le \sum_{i=0}^{N-1} \|S_{t-t_{i+1}}(\phi_{t_{i+1}}-S_{t_{i+1}-t_{i}}\phi_{t_{i}})(x)\|_{L_m(\Omega)}\\
		&\le \sum_{i=0}^{N-1} \sup_{z\in\D}\|\phi_{t_{i+1}}(z)-S_{t_{i+1}-t_{i}}\phi_{t_{i}}(z)\|_{L_m(\Omega)}\\
		&\le CN^{-\beta\kappa} (t-s)^{1+\beta\kappa}  (F+\|\rho\|_{\Ctimespacedzerom{[0,1]}}),
	\end{align*}	
	for $C=C(m,\beta)$. 
	Since $N^{-\beta \kappa}\le C(1+F^{\frac{-2\beta \kappa}{1-\kappa}})\le C(1+F)$, we get the desired bound \eqref{claim}.
\end{proof}	

We end this section with the following technical result. It shows that a certain auxiliary equation with a Lipschitz drift driven by a general noise has an adapted solution and its norm can be controlled uniformly. This is of course not surprising at all, the proof uses the standard Picard iterations argument and we place it  in the appendix.
\begin{lemma}\label{l:ebound}
	Let  $z\in \BB(\D,\R^d)$. Let $h$ be a continuous function $\R^d\to\R^d$ with $[h]_{\Lip(\R^d)}<\infty$.  Let $R\colon[0,1]\times \D\times\Omega\to\R$ be a measurable process  adapted to the filtration $(\F_t)_{t\in[0,1]}$.  Consider the following equation
	\begin{equation}\label{SPDEgeneric}
		Z(t,x)=S_t z(x) +\int_0^t\int_{\D} s_{t-r}(x,y) h( Z_r(y))\,dydr+R(t,x),\quad t\in[0,1],\,\, x\in \D.
	\end{equation}	 
	Suppose that  $\|R\|_{\Ctimespacedzerom{[0,1]}}<\infty$. Then this equation has an adapted solution on the same space $(\Omega,\F,(\F_t)_{t\in[0,1]},\P)$. Moreover, ${\|Z-R\|_{\Ctimespaced{1}{m}{[0,1]}}<\infty}$.
\end{lemma}

\subsection{Integral bounds for SHE}

In this subsection we provide bounds of integral functionals of a solution to the linear stochastic heat equation (process $V$ defined in \eqref{def.V}) which are necessary for the arguments of the weak uniqueness proof. 

We fix the length of the time interval $T>0$, and the filtration $(\F_t)_{t\in[0,T]}$ such that $W$ is an $(\F_t)$–space-time white noise. In this section we have $d=1$. First, we check that the process $V$ satisfies \eqref{Ycond}. Recall \cref{c:conv}.

\begin{lemma}\label{l:spde}
	The convolution kernel $s_t:=p_t$ for $t\in[0,T]$ defined on the domain $\D:=D$ satisfies condition \eqref{convprop}.  The process  $\{V_{t,x},t\in[0,T], x\in D\}$ satisfies \eqref{Ycond} with $\kappa:=1/4$.
\end{lemma}
\begin{proof} The first statement follows directly from the definition of $p$ in \cref{c:conv}. To verify that $V$ satisfies \eqref{Ycond}, we use \cref{l:gycond}. Let $(s,t)\in\Delta_{[0,T]}$, $x\in D$. Then, clearly
\begin{equation*}
V_t(x)-\E^s V_t(x)=\int_s^t\int_D p_{t-r}(x,y)W(dr,dy).
\end{equation*}	
We see that this random variable is Gaussian and independent from $\F_s$. Next, 
by \cite[Lemma~C.3]{ABLM}, $\Var (V_t(x)-\E^s V_t(x))\ge C(t-s)^{1/2}$, for $C>0$ independent of $t,x$. Hence, conditions \eqref{ind} and \eqref{sigmaY} holds with $\kappa=1/4$. Therefore,  \cref{l:gycond} implies now that $V$ satisfies \eqref{Ycond} with $\kappa=1/4$.
\end{proof}

We present now the key technical bound for the weak uniqueness proof of equation \Sref{u_0;b}.

\begin{corollary}\label{c:sfinfin}
	Let  $m\ge2$, $\beta\in(-2,0)$, $\tau\in[0,1]$, $\eps\in(0,1/2)$. Let $f\colon\R\to\R$ be a bounded continuous  function.
	Let $\phi,\psi\colon [0, 1]\times D\times \Omega\to\R$ be measurable processes adapted to the filtration $(\F_t)_{t\in[0,1]}$.
	Assume that
	\begin{equation}\label{sparam}
		\beta>1-4 \tau.
	\end{equation}
Then there exists a constant $C=C(m,\beta, \eps,\tau)>0$ such that for any $\lambda>1$ and $(s,t)\in\Delta_{[0,1]}$ we have
	\begin{align}\label{spdew}
		&\Bigl\|\int_s^t\int_D e^{-\lambda(t-r)}p_{t-r}(x,y) f(V_r(y)+\phi_r(y))\,dydr\Bigr\|_{L_m(\Omega)}\nn\\
		&\qquad\le
		C \|f\|_{\C^\beta}(t-s)^\eps\lambda^{-1-\frac\beta4+3\eps}\Bigl(1+[\phi]_{\C^{\tau,0}L_m ([s,t])}\lambda^{\frac14-\tau}\Bigr)
	\end{align}
	If, additionally, for $\mu\in[0,1]$ we have
	\begin{equation}\label{sparamcond2}
		\beta>\mu-2,
	\end{equation}
	then there exists a constant $C=C(m,\beta,\eps, \mu,\tau)>0$ such that for any $(s,t)\in\Delta_{[0,1]}$ we have 
	\begin{align}\label{spde2w}
		&\Bigl\|\int_s^t\int_D e^{-\lambda(t-r)} p_{t-r}(x,y) (f(V_r(y)+\phi_r(y))-f(V_r(y)+\psi_r(y)))\,dydr\Bigr\|_{L_m(\Omega)}\nn\\
		&\quad\le
		C \|f\|_{\C^\beta}
		(t-s)^\eps\lambda^{-1-\frac{\beta-\mu}4+3\eps}\|\phi-\psi\|^\mu_{\C^{0,0}L_m([s,t])}\nn\\
		&\qquad+C\|f\|_{\C^\beta} 	(t-s)^\eps \lambda^{-1-\frac{\beta-1}4-\tau+3\eps}\bigl([\phi]_{ \C^{\tau,0}L_m([s,t])}+[\psi]_{ \C^{\tau,0}L_m([s,t])}\bigr).
	\end{align}
\end{corollary}
\begin{proof}
By \cref{l:spde}, the process $V$ satisfies \eqref{Ycond} with $\kappa=1/4$. The result follows now from \cref{c:finfin}.	
\end{proof}

Next, we derive an a priori bound of the norm of the solution to equation \Sref{u_0;b}.
\begin{corollary}\label{l:apspde}
Let $u_0\in\BB(D)$, $\alpha\in(-\frac3{2},0)$, $b\in\C^\alpha$. Let $u$ be a weak solution to \Sref{u_0;b} and suppose that $u\in\Vscl$. Then for any $m\ge2$ there exists $C=C(m)>0$ such that
	\begin{equation}\label{claimspdesol}
		[u-V]_{\Ctimespace{1+\frac\alpha4}{m}{[0,1]}}\le  C (1+\|b\|_{\C^\alpha}^2).
	\end{equation}
\end{corollary}
\begin{proof}
	We apply \cref{l:apriori} with $\beta=\alpha$, $\kappa=1/4$, $\rho\equiv0$, $\D=D$, $S=P$, $z=u_0$, $Y=V$, $Z=u$, $f_n=G_{1/n}b$. By \cref{l:fbm}, $V$ satisfies \eqref{Ycond} with $\kappa=1/4$. Therefore our assumption $\alpha>-3/2$ coincides with the  condition $\beta\in(1/2-1/2\kappa,0)$ of \cref{l:apriori}. Recalling the definition of a solution  to \Sref{u_0;b} in \cref{def:ssol}, we see that the pair $(u,V)$ satisfies \eqref{zequ}. We see also that \eqref{zcond} holds thanks to the assumption $u\in\Vscl$. Thus all the conditions of \cref{l:apriori} are satisfied and \eqref{claimspdesol} follows from  \eqref{claim}.
%
\end{proof}

Finally, we need to bound the norm of a solution to a stochastic heat equation with additional forcing.
\begin{corollary}\label{l:sapriori}
	Let $m\ge2$, $\alpha\in(-\frac3{2},0)$, $u_0\in\BB(D)$. Let $f$ be a bounded continuous function $\R\to\R$. Let $\rho,Z\colon[0,1]\times D\times \Omega\to\R$ be measurable process adapted to the filtration $(\F_t)_{t\in[0,1]}$. Suppose that 
	\begin{equation*}
		Z(t,x)=P_t u_0(x)+\int_0^t\int_D p_{t-r}(x,y) f(Z_r(y))\,dydr +\int_0^t \int_D p_{t-r}(x,y) \rho(r,y)\,dydr + V(t,x), 
	\end{equation*}	
	where $t\in[0,1]$, $x\in D$.
	Assume that
	\begin{equation}\label{zspdecond}
	[Z-V]_{\Ctimespace{\frac58}{m}{[0,1]}}<\infty.
	\end{equation} 
	
	Then there exists $C=C(m)>0$ such that
	\begin{equation}\label{claimspde}
		[Z-V]_{\Ctimespace{1+\frac\alpha4}{m}{[0,1]}}\le  C (1+\|f\|_{\C^\alpha}^2)(1+\|\rho\|_{ \Ctimespacezerom{[0,1]}}).	
	\end{equation}
\end{corollary}	
\begin{proof}
Inequality \eqref{claimspde} follows directly from \cref{l:apriori} and \cref{l:spde}.
\end{proof}

\subsection{Bounds for integrals of fractional Brownian motion}

Using general results from \cref{s:gen}, we derive bounds for  integral functionals of fractional Brownian motion  with a drift, which will be crucial for our weak and strong uniqueness arguments. 
We fix $d\in\N$, $H\in(0, 1)$, the length of the time interval $T>0$, and the filtration $(\F_t)_{t\in[0,T]}$ such that $B^H$ is an $(\F_t)$–fractional Brownian motion. 

First we show that fractional Brownian motion satisfies conditions of \cref{s:gen} with $\kappa:=H$.

\begin{lemma}\label{l:fbm}
	The convolution kernel $s_t(0):=1$ for $t\in[0,T]$ defined on the domain $\D:=\{0\}$ satisfies condition \eqref{convprop} (recall \cref{c:convdelta}). 
	The proces $Y_t(0):=B^H_t$, $t\in[0,T]$, satisfies \eqref{Ycond} with $\kappa:=H$.
\end{lemma}
\begin{proof} The first statement is obvious. To show the second statement we apply \cref{l:gycond}. We write $B^H=(B^{H,1},\hdots, B^{H,d})$. Recall the representation \eqref{bhrepr}.  For any $(s,t)\in\Delta_{[0,T]}$, 
	the vector $B_t^{H}-\E^s B_t^{H}=\int_s^t K_H(t,r)\,dW_r$ is Gaussian and independent from $\F_s$. Therefore \eqref{ind} holds. We also see that all of the components of this vector are i.i.d and 
	\begin{equation*}
		\Var (B_t^{H,i}-\E^s B_t^{H,i})=\int_s^t (K_H(t,r))^2\,dr\ge C(t-s)^{2H},\quad i=1,...,d,
	\end{equation*}	
	where the last inequality follows from a direct calculation, see, e.g., \cite[Proposition~B.2(iii)]{blm23}. Therefore condition \eqref{sigmaY} holds as well with $\kappa=H$.
	
	Thus, all the conditions of \cref{l:gycond} hold and therefore $B^H$ satisfies \eqref{Ycond} with $\kappa=H$.
\end{proof}

Now we are ready to bound of a weighted integral of a fractional Brownian motion with a drift. 	
\begin{corollary}\label{c:sde}
	Let  $m\ge2$, $\beta\in(-\frac1{2H},0)$, $\tau\in[0,1]$, $\eps\in(0,1/2)$. Let $f\colon\R^d\to\R$ be a bounded continuous function.
	Let $\phi,\psi\colon [0, 1]\times\Omega\to\R^d$ be measurable processes adapted to the filtration $(\F_t)_{t\in[0,1]}$.
	Assume that
	\begin{equation}\label{parsde}
		\beta>1-\frac{\tau}{H}.
	\end{equation} 
	Then there exists a constant $C=C(m,\beta, \eps,\tau)>0$ such that for any $\lambda>1$ and $(s,t)\in\Delta_{[0,1]}$ we have
	\begin{equation}\label{sdeb1}
		\Bigl\|\int_s^t e^{-\lambda(t-r)}f(B_r^H+\phi_r)\,dr\Bigr\|_{L_m(\Omega)}\le
		C \|f\|_{\C^\beta}(t-s)^\eps\lambda^{-1-\beta H+3\eps}\Bigl(1+[\phi]_{\C^{\tau}L_m ([s,t])}\lambda^{-\tau+H}\Bigr)
	\end{equation}
	If, additionally, for $\mu\in[0,1]$ we have
	\begin{equation}\label{parsde2}
		\beta>\mu-\frac{1}{2H},
	\end{equation} 
	then there exists a constant $C=C(m,\beta,\eps, \mu,\tau)>0$ such that for any $(s,t)\in\Delta_{[0,1]}$ we have 
	\begin{align}\label{sdeb2}
		&\Bigl\|\int_s^t e^{-\lambda(t-r)}(f(B_r^H+\phi_r)-f(B_r^H+\psi_r))\,dr\Bigr\|_{L_m(\Omega)}\nn\\
		&\quad\le
		C \|f\|_{\C^\beta}
		(t-s)^\eps\lambda^{-1-(\beta-\mu)H+3\eps}\|\phi-\psi\|^\mu_{\C^0L_m([s,t])}\nn\\
		&\qquad+C\|f\|_{\C^\beta} 	(t-s)^\eps \lambda^{-1-(\beta-1)H-\tau+3\eps}\bigl([\phi]_{ \C^{\tau}L_m([s,t])}+[\psi]_{ \C^{\tau}L_m([s,t])}\bigr).
	\end{align}
\end{corollary}
\begin{proof}
	Thanks to \cref{l:fbm}, fractional Brownian $B^H$ satisfies \eqref{Ycond} with $\kappa:=H$. \cref{c:sde} follows now immediately from \cref{c:finfin}.
\end{proof}	

The next bound is a version of \cref{l:finfin} for the case when the drift process $\phi$ is measurable but not necessarily adapted to the filtration $(\F_t)$. It will be used to show strong uniqueness when we will have to deal with non-adapted drifts, for reasons explained in \cref{s:oview}.

\begin{lemma}\label{l:noreg}
	Let  $\beta<0$, $p\in[1,2)$. Let $f\colon\R^d\to\R$ be a bounded continuous function.
	Let $\phi\colon [0, 1]\times\Omega\to\R^d$ be a continuous process.
	Assume that
	\begin{equation}\label{dparamcond}
		\beta>\frac{p}2-\frac1{2H}.
	\end{equation}
	Then there exists a nonnegative  random variable $\xi$ such that 
	\begin{equation}\label{bound1termnoreg}
		\sup_{t\in[0,1]}\Bigl| \int_0^t f(B_r^H+\phi_r)\,dr\Bigr| \le
		\xi (1+|\phi_0|)(1+[\phi]_{p-\var;[0;1]})
	\end{equation}
	and for any $m\ge 2$ there exists a constant $C=C(m,\beta,p)$ such that
	\begin{equation*}
		\|\xi\|_{L_m(\Omega)}\le C \|f\|_{\C^\beta}.
	\end{equation*}
\end{lemma}

\begin{remark}
	We stress once again that \cref{l:noreg} does not require $\phi$ to be adapted to the filtration $(\F_t)$. The price to pay is a more restrictive range of $\beta$ where this lemma is applicable compared with \cref{l:finfin}. Indeed, in the regime $H\in(0,1/2]$,  for $\phi\in\C^{p-\var}$, condition \eqref{paramcond} can be rewritten as $\beta>1-\frac{1}{pH}$ and $1-\frac{1}{pH}\le \frac{p}2-\frac1{2H}=\frac{p}{2}(1-\frac1{pH})$ because $p\le 2$ and $1-1/(pH)\le0$. This leads to the fact that not every weak solution to \Gref{x;b} can be rewritten as a non-linear Young integral, we refer to
	\cite[Chapter~8]{GG22} for further discussions.
\end{remark}

\begin{proof}[Proof of \cref{l:noreg}]
	Without loss of generality, we assume that $\beta\le 1-\frac1{2H}$. Indeed, we can always replace $\beta$ with $\beta':=\beta \wedge (1-\frac1{2H})$ and use the embedding $\C^\beta\subset \C^{\beta'}$ and note that, since $p<2$, if the main condition \eqref{dparamcond} is satisfied for $\beta$, it is also satisfied for $\beta'$.
	
	Fix  $\eps>0$ small enough so that 
	\begin{equation}\label{epscondgg}
		(\beta+\frac1{2H})\frac1p-\eps>\frac12;
	\end{equation}
	this is possible thanks to the main assumption \eqref{dparamcond}.
	
	First, suppose additionally that $f\in\C^\infty(\R^d)$. We apply 
	\cite[Lemma~4.3 and Corollary~4.5]{BDG}
	(see also \cite[Proposition~7.3]{LeSSL}).
	It follows that for any $m\ge1$ there exists a constant $C=C(m,\beta,\eps)$ such that for any $x,y\in\R^d$, $s,t\in[0,1]$ we have
	\begin{align*}
		&\Bigl\|\int_s^t (f(B_r^H+x)-f(B_r^H+y))\,dr\Bigr\|_{L_m(\Omega)}\le C \|f\|_{\C^\beta} |x-y|^{\beta+\frac1{2H}-\frac\eps{2}}|t-s|^{\frac12+\frac{H\eps}2};\\
		&\Bigl\|\int_s^t f(B_r^H)\,dr\Bigr\|_{L_m(\Omega)}\le C \|f\|_{\C^\beta} |t-s|^{1+\beta H}.
	\end{align*}
	Take now $m\ge \frac{8d}{H\eps}$ and apply the Kolmogorov continuity theorem in the form of  \cite[Corollary~A.5]{GG22}	with $w(s,t)\equiv1$, $\alpha=\beta+\frac1{2H}-\frac\eps2$, $\beta_1=\frac12+\frac{H\eps}2$, $\gamma=1/2$, $\eta=\beta+\frac1{2H}-\eps$,
	$\lambda=\eps$. We get that there exists a random variable $\xi$ such that for any $x,y\in\R^d$, $s,t\in\Delta_{[0,1]}$ one has 
	\begin{align}\label{GGbound}
		&\Bigl|\int_s^t (f(B_r^H+x)-f(B_r^H+y))\,dr\Bigr|\le \xi(\omega) |x-y|^{\beta+\frac1{2H}-\eps}(1+|x|^{\eps}+|y|^{\eps})(t-s)^{\frac12};\\
		&\Bigl|\int_s^t f(B_r^H)\,dr\Bigr|\le \xi(\omega) (t-s)^{1+\beta H-\eps}\label{GGbound2}
	\end{align}
	and 
	\begin{equation*}
		\|\xi\|_{L_m(\Omega)}\le C \|f\|_{\C^\beta}.
	\end{equation*}	
	for $C=C(m,\beta,\eps)$. Thanks to  assumption \eqref{epscondgg}, the exponent $\beta+\frac1{2H}-\eps$ is positive. Fix now $\omega\in\Omega$ and put
	\begin{equation*}
		A_{s,t}:=\int_s^t  f(B_r^H+\phi_s)\,dr,\qquad \A_t:=\int_s^t  f(B_r^H+\phi_r)\,dr,\quad(s,t)\in\Delta_{[0,1]}.
	\end{equation*}
	We apply deterministic sewing lemma (\cref{t:dSL}) to these processes. Let us check that the conditions of the lemma are satisfied. First, note that for any $s,t\in\Delta_{[0,1]}$, $u\in[s,t]$ we have
	by \eqref{GGbound} (we apply this bound with $x:=\phi_s(\omega)$, $y:=\phi_u(\omega)$)
	\begin{align*}
		|\delta A_{s,u, t}|&=\Bigl|\int_u^t  [f(B_r^H+\phi_s)-f(B_r^H+\phi_u)]\,dr\Bigr|\\
		&\le 2\xi(\omega)(t-s)^{\frac12}|\phi_s-\phi_u|^{\beta+\frac1{2H}-\eps}(1+\|\phi\|_{\C([0,1])}^\eps)	\\
		&\le 2 \xi(\omega)(t-s)^{\frac12}[\phi]_{p-\var;[s,t]}^{\beta+\frac1{2H}-\eps}(1+\|\phi\|_{\C([0,1])}^\eps).
	\end{align*}
	We see that the function $(s,t)\mapsto[\phi]_{p-\var;[s,t]}$, $(s,t)\in\Delta_{[0,1]}$ is a control to the power $1/p$. Therefore, by \cite[Exercise~1.10(iii)]{FV2010} the function 
	\begin{equation*}
		(s,t)\mapsto (t-s)^{\frac12}[\phi]_{p-\var;[s,t]}^{\beta+\frac1{2H}-\eps}	
	\end{equation*}	
	is a control to the power $\frac12+(\beta+\frac1{2H}-\eps)\frac1p$. 
	Condition \eqref{epscondgg} implies that $\frac12+(\beta+\frac1{2H})\frac1p-\eps>1$ and thus condition \eqref{con:ds2} is satisfied. It is also immediate to see, that by Lipschitz continuity of $f$ and uniform continuity of $\phi$ condition \eqref{conssl:ds3} is also satisfied. Thus, all the conditions of \cref{t:dSL} holds and we get from \eqref{est:dssl1} for any $t\in[0,1]$
	\begin{equation}\label{drez1}
		\Bigl|\int_0^t  f(B_r^H+\phi_r)\,dr\Bigr|\le |A_{0,t}|+C\xi(\omega)(1+\|\phi\|_{\C([0,1])}^{\eps})[\phi]_{p-\var;[0,1]}^{\beta+\frac1{2H}-\eps}, 
	\end{equation}
	for some $C=C(\beta,\eps,p)$. To bound $A_{0,t}$ we use \eqref{GGbound}--\eqref{GGbound2} once again. We get
	\begin{equation*}
		|A_{0,t}|\le \Bigl|\int_0^t  [f(B_r^H+\phi_0)-f(B_r^H)]\,dr\Bigr|+\Bigl|\int_0^t  f(B_r^H)\,dr\Bigr|\le \xi(\omega)(1+|\phi_0|).
	\end{equation*}	
	Combining this with \eqref{drez1} and taking supremum over $t\in[0,1]$ we obtain \eqref{bound1termnoreg} for $f\in\C^\infty(\R^d)$
	
	Now we show  \eqref{bound1termnoreg} for a general bounded continuous $f$. As above, we approximate $f$  by a sequence of smooth functions $f_n := G_{1/n}^d f$, which converge pointwise to $f$ as $n \to \infty$. We get 
	for any $t\in[0,1]$, $n\in\N$
	\begin{equation*}
		\Bigl| \int_0^t f_n(B_r^H+\phi_r)\,dr\Bigr|\le \xi_n(1+|\phi_0|)(1+[\phi]_{p-\var;[0;1]}), 
	\end{equation*}
	and $\|\xi_n\|_{L_m}\le C \|G^d_{1/n}f\|_{\C^\beta}\le C \|f\|_{\C^\beta}$. Hence, 
	\begin{equation}\label{pre-bound1termnoreg}
		\Bigl| \int_0^t f(B_r^H+\phi_r)\,dr\Bigr|=\liminf_{n\to\infty }\Bigl| \int_0^t f_n(B_r^H+\phi_r)\,dr\Bigr|\le \xi(1+|\phi_0|)(1+[\phi]_{p-\var;[0;1]}), 
	\end{equation}
	where we put $\xi:=\liminf \xi_n$. By Fatou's lemma, $\|\xi\|_{L_m}\le  C \|f\|_{\C^\beta}$. Taking supremum over $t\in[0,1]$ in \eqref{pre-bound1termnoreg} we get \eqref{bound1termnoreg}. 
\end{proof}	

Finally, we need a couple of bounds of the norm of the solution to \Gref{x;b} and its relatives.
\begin{corollary}\label{l:apsde}
	Let $x\in\R^d$, $\alpha<0$, $b\in\C^\alpha$. Suppose that \eqref{alphacond} holds. Let $X$ be a weak solution to \Gref{x;b} and suppose that $X\in\Vcl$. Then for any $m\ge2$ there exists $C=C(m)>0$ such that
	\begin{equation}\label{claimsdesol}
		[X-B^H]_{\C^{1+\alpha H}L_m ([0,1])}\le  C (1+\|b\|_{\C^\alpha}^2).	
	\end{equation}
\end{corollary}
\begin{proof}
	We apply \cref{l:apriori} with $\beta=\alpha$, $\kappa=H$, $\rho\equiv0$, $\D=\{0\}$, $S=Id$, $z=x$, $Y=B^H$, $Z=X$, $f_n=G^d_{1/n}b$. We see from \cref{l:fbm} that $B^H$ satisfies \eqref{Ycond} with $\kappa=H$. Further we see from \eqref{alphacond} that condition $\beta\in(1/2-1/2\kappa,0)$ is also satisfied. The definition of a solution  to \Gref{x;b}, \cref{D:sol}, implies that the pair $(X,B^H)$ satisfies \eqref{zequ}. Finally \eqref{zcond} holds because $X\in\Vcl$. Thus all the conditions of \cref{l:apriori} are satisfied and \eqref{claimsdesol} follows from  \eqref{claim}.
\end{proof}

\begin{corollary}\label{l:apforced}
	Let $m\ge2$, $\alpha<0$, $z\in\R^d$. Suppose that \eqref{alphacond} holds. Let $f$ be a bounded continuous function  $\R^d\to\R^d$. Let $\rho,Z\colon[0,1]\times\Omega\to\R^d$ be measurable processes adapted to the filtration $(\F_t)_{t\in[0,1]}$ such that 
	\begin{equation*}
		Z_t=z+\int_0^t f(Z_r)\,dr +\int_0^t \rho_r dr + B_t^H,\quad t\in[0,1].
	\end{equation*}	
	Suppose that
	\begin{equation}\label{zcondm}
		[Z-B^H]_{\C^{\frac12+\frac{H}2}L_m ([0,1])}<\infty.
	\end{equation} 
	Then there exists $C=C(m)>0$ such that
	\begin{equation}\label{claimforced}
		[Z-B^H]_{\C^{1+\alpha H}L_m ([0,1])}\le  C (1+\|f\|_{\C^\alpha}^2)(1+\|\rho\|_{ \C^{0}L_m([0,1])}).	
	\end{equation}
\end{corollary}	
\begin{proof}
	The result follows from  \cref{l:apriori} applied with $\beta=\alpha$, $\kappa=H$, $\D=\{0\}$, $S=Id$,  $Y=B^H$,  $f_n=f$. We see that all the conditions of this lemma are satisfied thanks to \cref{l:fbm}.
\end{proof}

\section{Proofs of the main results for SDEs}\label{s:sdeproofs}

In this section we fix $d\in\N$, $H\in(0, \frac12]$ and $\alpha<0$ which satisfies our standing assumption \eqref{alphacond}. Without loss of generality we assume that the length of the time interval $T=1$. 

\subsection{Weak uniqueness}\label{s:sde}

The key element in the proofs of \cref{t:wu,t:sb} is the following stability bound.

\begin{lemma}\label{l:ml}
Let $x,y\in\R^d$, $b\in\C^\alpha(\R^d)$, $g\in\Lip(\R^d,\R^d)$. Let $(\Omega,\F,(\F_t)_{t\in[0,1]},\P)$ be a probability space. Let $(X,B^H)$ be \textbf{any} weak solution  to \Gref{x;b} in the class $\Vcl$ on this space. Let $(Y,B^H)$ be a strong soluton to  \Gref{y;g} defined on the same space. Then there exist constants $C,\eps>0$ which depends only on $\alpha$, $H$, $d$ such that 
\begin{equation}\label{mainressde}
	\W_{\|\cdot\|\wedge1}(\law(X,B^H),\law(Y,B^H))\le C \Gamma (\|b-g\|_{\C^{\alpha-\eps}}^{\frac\eps8}+|x-y|^{\frac\eps8}).
\end{equation}
for 
$$
\Gamma:=(1+\|g\|_{\C^{\alpha}}^{\frac{20}{\eps}}+\|b\|_{\C^{\alpha}}^{\frac{20}{\eps}})
(1+|x|+|y|).
$$
\end{lemma}

To obtain \cref{l:ml}, we fix $x,y\in\R^d$, $b\in\C^\alpha$, $g\in\Lip(\R^d,\R^d)$. Let $(X,B^H)$ be any weak solution to \Gref{x;b} on a space $(\Omega,\F,(\F_t)_{t\in[0,1]},\P)$ in the class $\Vcl$. 
Let $Y$ be the strong solution to \Gref{y;g}. Now we construct the generalized coupling process. For a parameter $\lambda\ge1$ to be specified later, consider the following equation:
\begin{equation}\label{SDEwty}
d\wt Y_t=g(\wt Y_t)dt+\lambda(X_t-\wt Y_t)dt + dB_t^H,\quad \wt Y_0=y.
\end{equation}	 
It is easy to see that this equation is well-defined.
\begin{lemma}\label{L:ks0}
SDE \eqref{SDEwty} has an adapted solution on the same space $(\Omega,\F,(\F_t)_{t\in[0,1]},\P)$. Moreover, ${\|\wt Y-B^H\|_{\C^1 L_m([0,1])}<\infty}$.
\end{lemma}
\begin{proof}
We apply \cref{l:ebound} with $\D=\{0\}$, $s_t(0)=1$ (recall \cref{c:convdelta}), $h(x)=g(x)-\lambda x$, $R(t)=B_t^H+\lambda \int_0^t X_r\,dr$, $Z=\wt Y$. We see that $h$ is Lipschitz and 
\begin{equation*}
\|R\|_{\C^0L_m([0,1])}\le (1+\lambda)\|B^H\|_{\C^0L_m([0,1])}+\lambda \|X-B^H\|_{\C^0L_m([0,1])}<\infty
\end{equation*}
since $X\in\Vcl$. Thus, all the conditions of \cref{l:ebound} are satisfied and applying this lemma we see that SDE \eqref{SDEwty} has an adapted solution $\wt Y$ and 
\begin{align*}
	\|\wt Y -B^H\|_{\C^1L_m([0,1])}&\le \|\wt Y -R\|_{\C^1L_m([0,1])}+\|R-B^H\|_{\C^1L_m([0,1])}\\
	&\le \|\wt Y -R\|_{\C^1L_m([0,1])}+ \lambda \|B^H\|_{\C^0L_m([0,1])}+
	\lambda \|X-B^H\|_{\C^0L_m([0,1])}<\infty.\qedhere
\end{align*}
\end{proof}

The next two lemmas are crucial for the whole generalized coupling strategy. We show that $\wt Y$ is close to $Y$ in law and to $X$ in space.

Define
\begin{equation}\label{bhtilde}
	\wt B^H_t:=B^H_t+\int_0^t \lambda (X_r-\wt Y_r)\,dr,\quad t\in[0,1].
\end{equation}
We see that $\wt Y$ solves the same equation as $Y$ with $\wt B^H$ in place of $B^H$. This allows to show that $\wt Y$ and $Y$ are close in law.

\begin{lemma}\label{l:ks1}
There exists a constant $C>0$ such that 
\begin{equation}\label{dtvbound}
\dtv(\law( Y,B^H),\law(\wt Y,\wt B^H))\le C \lambda \| \|X-\wt Y\|_{\C([0,1])} \|_{L_2(\Omega)}.
\end{equation}	
\end{lemma}
\begin{proof}
We use a localization argument and the Pinsker inequality together with the Girsanov theorem extending the strategy of \cite[proof of~Theorem~A.2]{BKS18} to our setting.

For $N>0$, $t\in[0,1]$ define
\begin{equation*}
\beta_N(t):=\lambda (X_t-\wt Y_t)\1_{ |X_t-\wt Y_t|\le N},\qquad \wt B_t^{H,N}:=B_t^H+
\int_0^t \beta_N(r)\,dr.
\end{equation*}
Obviously, $|\beta_N|\le \lambda N$. Hence, by the Girsanov theorem for the fractional Brownian motion (recall that $H\le1/2$) \cite[Theorem~2]{OuNu} 
(see also \cite[Proposition~3.10(i)]{BDG}), there exists an equivalent probability measure $\wt \P^{N}$ such that the process $\wt B^{H,N}$ is an $(\F_t)_{t\in[0,1]}$-fractional Brownian motion under this measure and 
\begin{equation}\label{denratio}
	\frac{d \wt \P^{N}}{d \P}=\exp\Bigl(-\int_0^1 v^N_s \,d W_s-\frac12\int_0^1 |v^N_s|^2\,ds\Bigr),
\end{equation}	
where $W$ is a Brownian motion as in \eqref{bhrepr},  
\begin{equation*}
	v_t^N:=c_H t^{H-\frac12}\int_0^t (t-s)^{-H-\frac12}s^{\frac12-H}\beta_N(s)\,ds,\quad t\in[0,1],
\end{equation*}	
$c_H$ is a positive constant \cite[formula~(13)]{OuNu}. It is straightforward that
\begin{equation*} 
	|v_t^N|\le C \|\beta_N\|_{\C([0,1])}\le C \lambda ( \|X-\wt Y\|_{\C([0,1])}\wedge N),\quad t\in[0,1],
\end{equation*}	
for $C>0$.

Let $\wt Y_N$ be a strong solution to \Gref{y;g} on the space $(\Omega,\F,\wt \P^N)$. It exists since $g$ is Lipschitz. That is, $\wt Y_N$ solves 
\begin{equation}\label{Yneq}
\wt Y_N(t)=y+ \int_0^t g(\wt Y_N(t))\,dr +\wt B_t^{H,N},\quad t\in[0,1].
\end{equation}

Now we are ready to bound the left-hand side of \eqref{dtvbound}. 	
We clearly have
\begin{align}\label{treq}
	\dtv(\law_\P( Y,B^H),\law_\P(\wt Y,\wt B^H))&\le \dtv(\law_\P( Y,B^H),\law_\P(\wt Y_N,\wt B^{H,N}))\nn\\
	&\phantom{\le}+\dtv(\law_\P(\wt Y_N,\wt B^{H,N}),\law_\P(\wt Y,\wt B^H)).
\end{align}

We see that  $(\wt Y_N,\wt B^{H,N})$ is a strong solution to \Gref{y;g} on $(\Omega,\F,\wt \P^N)$. 
Clearly,  $(Y,B^{H})$ is a strong solution to \Gref{y;g} on $(\Omega,\F,\P)$ Therefore, by weak uniqueness for \Gref{y;g}  (recall that $g\in\Lip(\R^d,\R^d)$) we have $\law_{\P}((Y,B^H))=\law_{\wt \P^N}((\wt Y_N,\wt B^{H,N})) $

Thus, by Pinsker's inequality (see, e.g., \cite[Lemma 2.5.(i)]{Tsy}) and the explicit formula for the density \eqref{denratio}
\begin{align}\label{term1gir}
\dtv(\law_{\P}( Y,B^H),\law_{\P}(\wt Y_N,\wt B^{H,N}))&=
\dtv(\law_{\wt \P^N}(\wt Y_N,\wt B^{H,N}),\law_{\P}(\wt Y_N,\wt B^{H,N}))\nn\\
&\le \dtv(\P,\wt\P^N) \nn\\
&\le \frac1{\sqrt2}\Bigl(\E^{\P}\log \frac{d \P}{d \wt \P^{N}}\Bigr)^{1/2}=\frac12\Bigl(\int_0^1 \E^\P|v_s^N|^2\,ds\Bigr)^{1/2}\nn\\
&\le C \lambda \| \|X-\wt Y\|_{\C([0,1])} \|_{L_2(\Omega)}.
\end{align}	
Here we used the bound $|v_t^N|\le C \lambda N$ to conclude that the expected value of the stochastic integral is $0$. 

To bound the second term in the right-hand side of \eqref{treq}, we note that on the set $\{\|X-\wt Y\|_{\C([0,1])}\le N\}$
the process $\wt Y$ satisfies
\begin{equation*}
	\wt Y(t)=y+ \int_0^t g(\wt Y(t))\,dr +\wt B_t^{H,N},\quad t\in[0,1].
\end{equation*}
Comparing this with equation for $\wt Y_N$ \eqref{Yneq} and using Gronwall lemma (recall again that $g$ is Lipschitz),
we see that $\wt Y_N=\wt Y$ on $\{\|X-\wt Y\|_{\C([0,1])}\le N\}$. Obviously,  $\wt B^{H,N}=\wt B^{H}$ on $\{\|X-\wt Y\|_{\C([0,1])}\le N\}$. Therefore,
\begin{equation*}
	\dtv(\law_\P(\wt Y_N,\wt B^{H,N}),\law_\P(\wt Y,\wt B^H))\le \P(\|X-\wt Y\|_{\C([0,1])}>N).
\end{equation*}
Now we combine this with \eqref{term1gir} and substitute into \eqref{treq}. We get 
\begin{equation*}
	\dtv(\law_\P( Y,B^H),\law_\P(\wt Y,\wt B^H))\le  C \lambda \| \|X-\wt Y\|_{\C([0,1])} \|_{L_2(\Omega)}+
	\P(\|X-\wt Y\|_{\C[0,1]}>N), 
\end{equation*}
for $C>0$. Since $N$ was arbitrary, by passing to the limit as $N\to\infty$, we get \eqref{dtvbound}.
\end{proof}

Now we are ready to start working on the key element of the entire proof: we must show that $\wt Y$ is close to $X$. Heuristically, this occurs due to the presence of the term $\lambda (X-\wt Y)$, which pushes $\wt Y$ towards $X$. However, as explained in \cref{s:oview}, the strategy used in \cite[proof of Proposition~3.1]{Ksch} is insufficient to achieve this result when the drift is a distribution. Therefore, we have developed a different approach that allows us to obtain more delicate bounds.

\begin{lemma}\label{L:ks2}
For any $\eps>0$ such that 
\begin{equation}\label{epscond}
6\eps<\frac12+H(\alpha-\frac12)
\end{equation}
there exist constants $C=C(\eps),C_0=C_0(\eps)>1$ such that for any
\begin{equation}\label{lambdacond}
\lambda> C_0(1+\|g\|_{\C^\alpha}^{\frac6\eps}+ \|b\|_{\C^{\alpha}}^{\frac6\eps})
\end{equation}
one has 
\begin{equation}\label{mainboundxmy}
 \bigl\| \sup_{t\in[0,1]}|X_t-\wt Y_t |\bigr\|_{L_2(\Omega)}\le   C \lambda^{-1-\frac\eps2}
 +C\lambda^{2\eps} (\|b-g\|_{\C^{\alpha-\eps}}+|x-y|).
\end{equation}	
\end{lemma}

We see that assumption \eqref{alphacond} implies that the set of all positive $\eps$ satisfying \eqref{epscond} is non-empty. 

\begin{proof}
Fix  $\eps>0$ satisfying \eqref{epscond}. Denote $m:=2/\eps\ge2$ and 
\begin{equation}\label{gmdef}
	\Theta:=1+\|g\|_{\C^\alpha}^3+ \|b\|_{\C^{\alpha}}^3.
\end{equation}
Put
\begin{equation*}
\phi:=X-B^H,\qquad \psi:=\wt Y- B^H.
\end{equation*}
Let $b_n:=G^d_{1/n}b$, $n\in\N$.  For $n\in\N$ define
\begin{equation*}
Z^n_t:=x+\int_0^t  b_n(X_t)\,dt +B_t^H, \quad t\in[0,1].
\end{equation*}
By definition of a solution to \Gref{x;b} and by passing to a subsequence if necessary, we have 
\begin{equation}\label{conv}
\|Z^n - X\|_{\C([0,1])}\to0\quad a.s.,\quad  \text{as $n\to\infty$}. 	
\end{equation}	

\textbf{Step~1}. 	It is immediate to see that 
	\begin{equation*}
		d(Z^n_t-\wt Y_t)=-\lambda (X_t-\wt Y_t)dt+ (b_n(X_t)-g(\wt Y_t))\,dt.
	\end{equation*}
	
	Hence, by the chain rule
	\begin{equation*}
		d[e^{\lambda t}(Z^n_t-\wt Y_t)]=e^{\lambda t} (b_n(X_t)-g(\wt Y_t))\,dt+
		e^{\lambda t} \lambda (Z_t^n-X_t)\,dt,
	\end{equation*}
	which implies for $(s,t)\in\Delta_{[0,1]}$
	\begin{align*}
		&|(Z^n_t-\wt Y_t)-e^{-\lambda(t-s)}(Z^n_s-\wt Y_s)|\\
		&\qquad\le\Bigl|\int^t_s e^{-\lambda (t-r)} (b_n(X_r)-g(\wt Y_r))\,dr\Bigr|+\lambda
		\int^t_s|Z_r^n-X_r|\,dr\\
		&\qquad\le \Bigl|\int^t_s e^{-\lambda (t-r)} (g(X_r)-g(\wt Y_r))\,dr\Bigr|		+\Bigl|\int_s^t e^{-\lambda (t-r)}(b_n(X_r)-g(X_r)) \, dr\Bigr|\\
		&\qquad\phantom{\le}+\lambda \|Z^n - X\|_{\C[0,1]}\\
				&\qquad=: I_{1}(s,t)+I_{2,n}(s,t)+\lambda \|Z^n - X\|_{\C[0,1]}.
	\end{align*}
	We pass to the limit in the above inequality as $n\to\infty$. Recalling \eqref{conv}, we get	
	\begin{equation*}
		|(X_t-\wt Y_t)-e^{-\lambda(t-s)}(X_s-\wt Y_s)|\le  I_{1}(s,t)+\liminf_{n\to\infty} I_{2,n}(s,t).
	\end{equation*}
	Therefore,
	\begin{align}\label{limeq}
		|(X_t-\wt Y_t)-(X_s-\wt Y_s)|&\le |(X_t-\wt Y_t)-e^{-\lambda(t-s)}(X_s-\wt Y_s)|+(1-e^{-\lambda(t-s)})|X_s-\wt Y_s|\nn\\
		&\le I_1(s,t)+\liminf_{n\to\infty} I_{2,n}(s,t)+\lambda^{\eps}|t-s|^{\eps}|X_s-\wt Y_s|,
	\end{align}
where we used the elementary inequality $1-e^{-a}\le a^\rho$ valid for any $a\ge0$, $\rho\in[0,1]$. 

	\textbf{Step~2}. 	Let us now analyze each of the terms in the above inequality. 	

	We begin with $I_{1}$. We apply  \cref{c:sde} with $\beta=\alpha$, $\tau=(1+ H)/2$, $\mu=\frac12$,  $f=g$. We see that conditions \eqref{parsde} and 
	\eqref{parsde2} hold thanks to our choice of parameters and the standing assumption \eqref{alphacond}. Therefore we get by \eqref{sdeb2}
	\begin{align}\label{ivanpred}
		\|I_{1}(s,t)\|_{L_m(\Omega)}&\le C \|g\|_{\C^\alpha}
	(t-s)^\eps\lambda^{-1-H(\alpha-\frac12)+3\eps}\|X-\wt Y\|_{\C^0L_m([0,1])}^{\frac12}\nn\\
		&\quad+C\|g\|_{\C^\alpha}(t-s)^\eps\lambda^{-\frac32-H(\alpha-\frac12)+3\eps}\bigl([\phi]_{ \C^{\frac{1+H}2}L_m([0,1])} +[\psi]_{ \C^{\frac{1+H}2}L_m([0,1])}\bigr )\nn\\
		&\le C \|g\|_{\C^\alpha}
		(t-s)^\eps\lambda^{-\frac12-2\eps}\|X-\wt Y\|_{\C^0L_m([0,1])}^{\frac12}\nn\\
		&\quad+C\|g\|_{\C^\alpha}(t-s)^\eps\lambda^{-1-2\eps}\bigl([\phi]_{ \C^{\frac{1+H}2}L_m([0,1])} +[\psi]_{ \C^{\frac{1+H}2}L_m([0,1])}\bigr ).
	\end{align}
	for $C=C(\eps)$ independent of $\lambda$, where in the last inequality we used \eqref{epscond}. To bound $[\phi]_{ \C^{\frac{1+H}2}L_m([0,1])}$
	in the right-hand side of \eqref{ivanpred} we apply \cref{l:apsde}. Recalling that $X\in\Vcl$ we get 
	\begin{equation}\label{phib}
	[\phi]_{ \C^{\frac{1+H}2}L_m([0,1])}\le [\phi]_{ \C^{1+\alpha H}L_m([0,1])}\le C(1+\|b\|^2_{\C^\alpha}).
\end{equation}	
	To bound $[\psi]_{ \C^{\frac{1+H}2}L_m([0,1])}$ we apply  \cref{l:apforced} 	 with $z=y$, $f=g$, $\rho=\lambda (X-\wt Y)$, $Z=\wt Y$. By \cref{L:ks0}, we have $[\wt Y-B^H]_{\C^1 L_m([0,1])}<\infty$ and hence condition \eqref{zcondm} holds. We get from \eqref{claimforced}
	\begin{equation}\label{psib}
		[\psi]_{ \C^{\frac{1+H}2}L_m([0,1])}\le
			[\psi]_{ \C^{1+\alpha H}L_m([0,1])}\le C(1+\|g\|^2_{\C^\alpha})(1+\lambda \|X-\wt Y\|_{\C^0 L_m([0,1])}).
	\end{equation}
	Now we substitute \eqref{phib}, \eqref{psib}  back into \eqref{ivanpred} and use the  inequality $ xy\le x^2+y^2$ valid for any $x,y\in\R$ for the first term there. Recalling the definition of $\Theta$ in \eqref{gmdef}, we obtain
		\begin{align}\label{ivanpred2}
		\|I_{1}(s,t)\|_{L_m(\Omega)}&\le C(t-s)^\eps \lambda^{-2\eps }\|X-\wt Y\|_{\C^0L_m([0,1])}+ C (t-s)^\eps\|g\|_{\C^\alpha}^2 \lambda^{-1-2\eps}\nn\\
		&\phantom{\le}+C \|g\|_{\C^\alpha}(1+\|b\|_{\C^\alpha}^2+\|g\|_{\C^\alpha}^2)(t-s)^\eps\lambda^{-1-2\eps}\nn\\
		&\phantom{\le}+C\|g\|_{\C^\alpha}(1+\|g\|_{\C^\alpha}^2)(t-s)^\eps\lambda^{-2\eps}\|X-\wt Y\|_{\C^0L_m([0,1])}\nn\\
		&\le C \Theta(t-s)^\eps \lambda^{-2\eps}\|X-\wt Y\|_{\C^0L_m([0,1])}+ C\Theta (t-s)^\eps \lambda^{-1-2\eps}.
	\end{align}
	
	Next, we move to $I_{2,n}$. Let us apply  \cref{c:sde} with $\beta=\alpha-\eps$, $\tau=(1+H)/2$, $f=b_n-g$. We see that $\alpha-\eps>\frac12-\frac1{2H}$ thanks to \eqref{epscond}, and thus condition \eqref{parsde2} holds. Therefore, we get  
	from \eqref{sdeb1} and \eqref{phib}
	\begin{align*}
		\|I_{2,n}(s,t)\|_{L_m(\Omega)}&\le 	C \|b_n-g\|_{\C^{\alpha-\eps}}(t-s)^\eps		\lambda^{-1-H\alpha+4\eps}\Bigl(1+[\phi]_{ \C^{\frac{1+H}2}L_m([0,1])}\lambda^{-\frac12+\frac{H}2}\Bigr)\nn\\
		&\le C \|b_n-g\|_{\C^{\alpha-\eps}}(t-s)^\eps	\nn\\
		&\phantom{\le}+C (\|b_n\|_{\C^{\alpha-\eps}}+\|g\|_{\C^{\alpha-\eps}})(t-s)^\eps(1+\|b\|^2_{\C^\alpha})\lambda^{-1-2\eps}.
	\end{align*}
	for $C=C(\eps)$ independent of $\lambda$ and $n$. Here we used \eqref{epscond} once again.

Now we substitute this together with \eqref{ivanpred2} into \eqref{limeq}. We note that $\|b_n -g\|_{\C^{\alpha-\eps}}\to \|b -g\|_{\C^{\alpha-\eps}}$ as $n\to\infty$. Therefore, by Fatou's lemma we derive
for any $(s,t)\in\Delta_{[0,1]}$
	\begin{align}\label{xmy1}
\!\!\!\!	\|(X_t-\wt Y_t)-(X_s-\wt Y_s)\|_{L_m (\Omega)}&\le C \Theta(t-s)^\eps \lambda^{-2\eps}\|X-\wt Y\|_{\C^0L_m([0,1])}\nn\\
	&\phantom{\le}+ C\Theta(t-s)^\eps  \lambda^{-1-2\eps}\nn\\
	&\phantom{\le}+C(t-s)^\eps \|b-g\|_{\C^{\alpha-\eps}}+\lambda^{\eps}|t-s|^{\eps}\|X_s-\wt Y_s\|_{L_m(\Omega)},
\end{align}
where $C=C(\eps)$ and we used again the definition of  constant $\Theta$ in \eqref{gmdef}.	

\textbf{Step~3: buckling for supremum norm}. 
 Choose now any $\lambda\ge1$ such that 
\begin{equation}\label{xmy2}
C\Theta\lambda^{-\eps/2}\le\frac12,
\end{equation}
where $C$ is as in \eqref{xmy1}.
By choosing  now $s=0$ in \eqref{xmy1} and taking supremum over all $t\in[0,1]$ we deduce with the help of \eqref{xmy2}
\begin{equation}\label{diffeq}
\|X-\wt Y\|_{\C^{0}L_m ([0,1])}\le \frac12\|X-\wt Y\|_{\C^{0}L_m ([0,1])}+ \lambda^{-1-\frac32\eps}+C \|b-g\|_{\C^{\alpha-\eps}}	+\lambda^\eps|x-y|.
\end{equation}	
Note that 	$\|X-B^H\|_{\C^{0}L_m ([0,1])}\le [X-B^H]_{\C^{\frac{1+H}2}L_m ([0,1])}+|x|<\infty$ because $X$ belongs to the class $\Vcl$. Furthermore, we have 	$\|\wt Y-B^H\|_{\C^{0}L_m ([0,1])}<\infty$ by \cref{L:ks0}. Hence ${\|X-\wt Y\|_{\C^{0}L_m ([0,1])}<\infty}$. 
Therefore, we can put the term $\frac12\|X-\wt Y\|_{\C^{0}L_m ([0,1])}$ in the left-hand side of \eqref{diffeq} and derive 	
	\begin{equation}\label{supfinal}
	\|X-\wt Y\|_{\C^{0}L_m ([0,1])}\le  C \lambda^{-1-\frac32\eps}+C \|b-g\|_{\C^{\alpha-\eps}}	+\lambda^\eps|x-y|.
\end{equation}
for $C=C(\eps)$.
	
\textbf{Step~4: bounding the H\"older norm}. 
Now we substitute \eqref{supfinal} into \eqref{xmy1}. We use \eqref{xmy2} to replace $\Theta$ by $\lambda^{\eps/2}$ and note that $\|X_s-\wt Y_s\|_{L_m(\Omega)}\le 	\|X-\wt Y\|_{\C^{0}L_m ([0,1])}$. We derive
 for any $(s,t)\in\Delta_{[0,1]}$
	\begin{equation*}
	\|(X_t-\wt Y_t)-(X_s-\wt Y_s)\|_{L_m (\Omega)}\le C  (t-s)^\eps \lambda^{-1-\eps/2}
	+C (t-s)^\eps\lambda^{2\eps}  (\|b-g\|_{\C^{\alpha-\eps}}+|x-y|)
\end{equation*}
for $C=C(\eps)$.
Dividing both sides of the above inequality by $(t-s)^\eps$ and taking the supremum over all $(s,t)\in\Delta_{[0,1]}$, we finally obtain
	\begin{equation*}
	[X-\wt Y]_{\C^{\eps}L_m ([0,1])}\le C\lambda^{-1-\eps/2}
	+C\lambda^{2\eps} (\|b-g\|_{\C^{\alpha-\eps}}+|x-y|).
\end{equation*}
Since $m\eps>1$ and the processes $X$ and $\wt Y$ are continuous, we can apply the Kolmogorov continuity theorem to derive	
	\begin{equation*}
	\|\sup_{t\in[0,1]}|X_t-\wt Y_t|\|_{L_m (\Omega)}\le C\lambda^{-1-\eps/2}
	+C\lambda^{2\eps} (\|b-g\|_{\C^{\alpha-\eps}}+|x-y|).
\end{equation*}
for $C=C(\eps)$, which is  the desired bound \eqref{mainboundxmy}.
\end{proof}

Now we are ready to prove \cref{l:ml}.

\begin{proof}[Proof of \cref{l:ml}] Recall the definition of $\wt B^H$ in \eqref{bhtilde}. 
Fix $\eps>0$ such that   \eqref{epscond}.	
Let $\lambda>1$ satisfy  \eqref{lambdacond}.
By the triangle inequality and \cref{l:ks1} we have
\begin{align}\label{mainproof1}
&\W_{\|\cdot\|\wedge1}(\law(X,B^H),\law(Y,B^H))\nn\\
&\quad\le \W_{\|\cdot\|\wedge 1}(\law(X,B^H),\law(\wt Y,\wt B^H))+
\W_{\|\cdot\|\wedge 1}(\law(\wt Y,\wt B^H),\law(Y,B^H))\nn\\
&\quad\le \bigl\|\|X-\wt Y\|_{\C([0,1])}\bigr\|_{L_1(\Omega)}+\bigl\|\|B^H-\wt B^H\|_{\C([0,1])}\bigr\|_{L_1(\Omega)}+ \dtv(\law(\wt Y,\wt B^H),\law(Y,B^H))\nn\\
&\quad\le C \lambda \bigl\|\|X-\wt Y\|_{\C([0,1])}\bigr\|_{L_2(\Omega)}\nn\\
&\quad\le C \lambda^{-\frac\eps2}
+C\lambda^{2} (\|b-g\|_{\C^{\alpha-\eps}}+|x-y|).
\end{align}
 Let $C_0=C_0(\eps)$ be as in \cref{L:ks2}.  Take now 
\begin{equation*}
\lambda:=(\|b-g\|_{\C^{\alpha-\eps}}+|x-y|)^{-1/4}+C_0 (1+\|g\|_{\C^\alpha}^{\frac6\eps}+ \|b\|_{\C^{\alpha}}^{\frac6\eps}).
\end{equation*}
Clearly such $\lambda$ satisfies \eqref{lambdacond}.  Substituting this $\lambda$ into \eqref{mainproof1} we get 
\begin{align*}
	\W_{\|\cdot\|\wedge1}(\law(X,B^H),\law(Y,B^H))&\le C       (\|b-g\|_{\C^{\alpha-\eps}}+|x-y|)^{\frac\eps8}\\
	&\phantom{\le} 	+C  (\|b-g\|_{\C^{\alpha-\eps}}+|x-y|)^{\frac12}\\
	&\phantom{\le} 	+C_0^2 (1+\|g\|_{\C^\alpha}^{\frac{12}\eps}+ \|b\|_{\C^{\alpha}}^{\frac{12}\eps})  (\|b-g\|_{\C^{\alpha-\eps}}+|x-y|),
\end{align*}
which implies \eqref{mainressde}.
\end{proof}	

With \cref{l:ml} in  hand, we are able to proof our main results concerning weak uniqueness.
\begin{proof}[Proof of \cref{t:wu}]
Fix $b\in\C^\alpha$, $x\in\R^d$ and let $(b^n)_{n\in\Z_+}$ be a sequence of $\C^\infty(\R^d,\R^d)$ functions converging to $b$ in $\C^{\alpha-}$, let $x_n$ be a sequence converging to $x$ in $\R^d$. 
Let $X^n,W^H)$ be the strong solution to \Gref{x_n;b^n}. Let $(X,B^H)$ be a weak solution to \Gref{x;b} in the class $\Vcl$.

By \cref{l:ml},  there exist constants  $C,\eps>0$ such that for any $n\in\Z_+$
\begin{align*}
	&\W_{\|\cdot\|\wedge1}(\law(X,B^H),\law(X^n,W^H))\\
	&\qquad\le C (1+\|b\|_{\C^\alpha}^{\frac{20}\eps}+\|b^n\|_{\C^\alpha}^{\frac{20}\eps})(1+|x|+|x_n|)(\|b-b^n\|_{\C^{\alpha-\eps}}^{\frac{\eps}8}+|x-x_n|^{\frac{\eps}8}).
\end{align*}
By the definition of convergence in $\C^{\alpha-}$, $\sup_n \|b^n\|_{\C^\alpha}<\infty$.  Hence, by passing to the limit as $n\to\infty$ in the above inequality, we get
\begin{equation}\label{eqconv}
\lim_{n\to\infty}\W_{\|\cdot\|\wedge1}(\law(X,B^H),\law(X^n,W^H))=0.
\end{equation}

(i). Now if  $(\overline{X},\overline{B}^H)$ is  another weak solution to \Gref{x;b},
then by above 
\begin{equation*}
	\lim_{n\to\infty}\W_{\|\cdot\|\wedge1}(\law(\overline{X},\overline{B}^H),\law(X^n,W^H))=0,
\end{equation*}
which implies that $\law(X,B^H)=\law(\overline{X},\overline{B}^H)$ and thus weak uniqueness holds for \Gref{x;b}.

(ii). By \eqref{eqconv} the sequence $\law(X^n,W^H)$ weakly converges in $\C([0,1],\R^{2d})$ and its limit is $\law(X,B^H)$, which is a unique weak solution to \Gref{x;b} in class $\Vcl$.
\end{proof}

\begin{proof}[Proof of \cref{t:sb}]
Fix $m\ge1$, $b_1,b_2\in\C^\alpha$, $x_1,x_2\in\R^d$ and let $b^{n,1}:=G^d_{1/n}b_1$, $n\in\N$. Let $(X^{n,1},W^H)$ be the strong solution to \Gref{x_1;b^{n,1}}. 
We apply \cref{l:ml} twice: first to the pairs $(X^1,B^{H,1})$ and $(X^{n,1},W^H)$, and then to the pairs 
$(X^2,B^{H,2})$ and $(X^{n,1},W^H)$. We get that there exists constants $C,\eps>0$
\begin{align*}
&\W_{\|\cdot\|\wedge1}(\law(X^1,B^{H,1}),\law(X^2,B^{H,2}))\\
&\qquad\le
\W_{\|\cdot\|\wedge1}(\law(X^1,B^{H,1}),\law(X^{n,1},W^H))+
\W_{\|\cdot\|\wedge1}(\law(X^2,B^{H,2}),\law(X^{n,1},W^H))
\\
&\qquad\le C (1+\|b_1\|_{\C^\alpha}^{\frac{20}\eps})(1+|x_1|)\|b_1-b^{n,1}\|_{\C^{\alpha-\eps}}^{\frac{\eps}8}\\
&\qquad\phantom{\le}+ C (1+\|b_1\|_{\C^\alpha}^{\frac{20}\eps}+\|b_2\|_{\C^\alpha}^{\frac{20}\eps})(1+|x_1|+|x_2|)(\|b_2-b^{n,1}\|_{\C^{\alpha-\eps}}^{\frac{\eps}8}+|x_1-x_2|^{\frac\eps8}).
\end{align*}	
By passing to the limit as $n\to\infty$, we get the desired bound \eqref{mmmssde}.
\end{proof}

\subsection{Strong existence and  uniqueness in the case $d=1$}

The proof of strong well-posedness of \Gref{x;b} in the case $d=1$ is based on the following key lemma.
\begin{lemma}\label{l:sud1}
Let $d=1$, $x\in\R$,  $b\in\C^\alpha$, $p\in[1,2]$. Suppose that
\begin{equation}\label{condalphap}
\alpha>\frac{p}2-\frac1{2H}.
\end{equation}
Let $(X^1,B^H)$, $(X^2,B^H)$ be two solutions to 
\Gref{x;b} defined on the same probability space and adapted to the same filtration $(\F_t)_{t\in[0,1]}$.  Suppose that $X^1,X^2\in\Vcl$ and 
\begin{equation}\label{pvarb}
[X^i-B^H]_{p-\var;[0,1]}<\infty,\quad a.s.,\,\,i=1,2.
\end{equation}
Then 
\begin{equation}\label{rezd1}
\P(X^1_t=X^2_t\, \text{ for any }\, t\in[0,1])=1.
\end{equation}
\end{lemma}	

\begin{proof}
 Consider the process $Y_t:=X_t^1\wedge X_t^2$, $t\in[0,1]$, and put 
\begin{equation*}
\psi^i:=X^i-x-B^H,\,\,i=1,2;\qquad \phi:=\psi^1\wedge \psi^2=Y-x-B^H.
\end{equation*} 
We claim that $(Y,B^H)$ is also a weak solution to \Gref{x;b}. This would imply by weak uniqueness that $\law(X_t^1\wedge X_t^2)=\law(X_t^1)$ for $t\in[0,1]$, and thus $X_t^1=X_t^2$ a.s. 

For the case when $b$ is a continuous function this strategy was carried out in \cite[Proposition~1.1]{bib:ttw74}. However their proof method does not work for the case when $b$ is a distribution, as the integration by parts formula of \cite[Proposition~1.1]{bib:ttw74} does not hold. Therefore we come with a different technique.

Note that the process $Y$ is clearly continuous and adapted to the filtration $(\F_t)_{t\in[0,1]}$. Therefore, it remains only to verify that condition \ref{cond2md} of \cref{D:sol} holds.

Assume the contrary. Then there exist $\eps,\delta>0$ and a sequence $(b^n)_{n\in\Z_+}$ of $\C^\infty(\R^d,\R^d)$ functions converging to $b$ in $\C^{\alpha-}$ such that for any $n\in\Z_+$
\begin{equation}\label{contra}
\P\Bigl(\sup_{t\in[0,1]}\Bigl|\int_0^t b^n(Y_r)dr -\phi_t\Bigr|>\eps\Bigr)>\delta.
\end{equation}
Denote now for $n\in\Z_+$, $t\in[0,1]$
\begin{equation*}
\psi_t^{i,n}:=\int_0^t b^n(X^i_r)\,dr,\,\,i=1,2;\qquad\phi_t^{n}:=\int_0^t b^n(Y_r)\,dr.
\end{equation*}
Since $X^1$ and $X^2$ are solutions to \Gref{x;b}, by passing to a common 
subsequence, if necessary, we have 
\begin{equation}\label{convn}
\|\psi^{1,n}- \psi^1\|_{\C[0,1]}\to0,\,\,\, \|\psi^{2,n}-\psi^2\|_{\C[0,1]}\to0 \quad a.s.,\quad  \text{as $n\to\infty$}. 	
\end{equation}
Consider a set $\Omega'\subset\Omega$ 
\begin{equation*}
\Omega':=\{\omega\in\Omega\colon \text{relation \eqref{convn} holds};\,\,   \|\psi^i(\omega)\|_{\C^{\frac12}([0,1])}\!<\!\infty;\,\, [\psi^i(\omega)]_{p-\var;[0,1]}\!<\!\infty; i=1,2\}.
\end{equation*}	
Since $X^1,X^2\in\Vcl$ the Kolmogorov continuity theorem, \eqref{convn}, and the assumptions of the lemma  implies that $\P(\Omega')=1$.

 Let $\omega\in\Omega'$. Since $\psi^1$ and $\psi^2$ are continuous functions, the set 
$$
A_\omega:=\{t\in[0,1]:\psi^1_t(\omega)>\psi^2_t(\omega)\}=
\{t\in[0,1]:X^1_t(\omega)>X^2_t(\omega)\}
$$ is open. Hence, it can be written as an at most countable union of mutually disjoint open intervals $A_\omega=\cup (s_k(\omega),t_k(\omega))$. We fix $K\in\N$ and consider the set 
\begin{equation*}
A^K_\omega:=\bigcup_{k\le K}(s_k(\omega),t_k(\omega))\subset A_\omega.
\end{equation*}
Put 
\begin{equation*}
Y^K_t(\omega):=X_t^1\1(t\in [0,1]\setminus A^K_\omega)+X_t^2\1(t\in A^K_\omega);\quad \phi^K:=Y^K-x-B^H,
\end{equation*}
and for arbitrary $m\in\N$ decompose the desired approximation error as 
\begin{align}\label{appe}
\sup_{t\in[0,1]}\Bigl|\int_0^t b^n(Y_r)dr -\phi_t\Bigr|&\le
\sup_{t\in[0,1]}\Bigl|\int_0^t (b^n(Y_r)-b^n(Y_r^K))dr\Bigr|+
\sup_{t\in[0,1]}\Bigl|\int_0^t (b^n(Y_r^K)-b^m(Y_r^K))dr\Bigr|\nn\\
&\phantom{\le}+\sup_{t\in[0,1]}\Bigl|\int_0^t b^m(Y_r^K)dr-\phi_t^K\Bigr|+\sup_{t\in[0,1]}|\phi^K_t -\phi_t\Bigr|\nn\\
&\le (1+\|b^n\|_{\Lip(\R)})\sup_{t\in[0,1]}|Y^K_t -Y_t\Bigr|+
\sup_{t\in[0,1]}\Bigl|\int_0^t (b^n(Y_r^K)-b^m(Y_r^K))dr\Bigr|\nn\\
&\phantom{\le}+\sup_{t\in[0,1]}\Bigl|\int_0^t b^m(Y_r^K)dr-\phi_t^K\Bigr|\nn\\
&=:I_1(n,K)+I_2(n,m,K)+I_3(m,K).
\end{align}
We treat now consequently all the three terms on the right-hand side of the above equation.
It is easy to bound $I_1$. Note that for any $k\in\N$ we have by continuity $X^1_{s_k(\omega)}=X^2_{s_k(\omega)}$. Thus, if $t\in[s_k(\omega),t_k(\omega)]$, then
\begin{equation*}
|\psi^1_t-\psi^2_t|=|\psi^1_t-\psi^1_{s_k(\omega)}-(\psi^2_t-\psi^2_{s_k(\omega)})|\le |t_k(\omega)-s_k(\omega)|^{\frac12}(\|\psi^1\|_{\C^{\frac12}([0,1])}+\|\psi^2\|_{\C^{\frac12}([0,1])}).
\end{equation*}
Hence 
\begin{align}\label{i1ss1s}
I_1(n,K)&= (1+\|b^n\|_{\Lip(\R)})\sup_{t\in A_\omega\setminus A^K_\omega}|Y^K_t -Y_t\Bigr|=
(1+\|b^n\|_{\Lip(\R)})\sup_{t\in A_\omega\setminus A^K_\omega}|\psi_t^1 -\psi_t^2\Bigr|\nn\\
&\le  (1+\|b^n\|_{\Lip(\R)})(\|\psi^1(\omega)\|_{\C^{\frac12}([0,1])}+\|\psi^2(\omega)\|_{\C^{\frac12}([0,1])})\max_{k>K}|t_k(\omega)-s_k(\omega)|.
\end{align}
Since the intervals $(s_k(\omega), t_k(\omega))$ are mutually disjoint, we have ${\max_{k>K} |t_k(\omega) - s_k(\omega)| \to 0}$ as $K \to \infty$. Hence, we deduce from \eqref{i1ss1s} and the definition of $\Omega'$ that for any fixed $n\in\N$, $\omega\in\Omega'$
\begin{equation}\label{i1ss}
\lim_{K\to\infty}I_1(n,K)=0.
\end{equation}

Next, we move on to $I_2$. We note that the process $Y^K$ can be written as $Y^K=x+B^H+\phi^K$, where the process $\phi^K$ is not necessarily adapted to the filtration $(\F_t)$. Therefore instead of bound \eqref{sb1} which requires the drift to be adapted we apply  \eqref{bound1termnoreg} with $\beta=\alpha$, $f=b^n-b^m$ and $\phi^K+x$ in place of $\phi$. We see that condition \eqref{condalphap} implies \eqref{dparamcond}. Therefore, recalling the definition of $\phi^K$, we deduce
that there exists a random variable $\xi_{n,m,K}$ such that
\begin{align*}
I_2(n,m,K)&\le \xi_{n,m,K}(1+|x|)(1+[\phi^K]_{p-\var;[0,1]})\\
&\le \xi_{n,m,K}(1+|x|)(1+[\psi^1]_{p-\var;[0,1]}+[\psi^2]_{p-\var;[0,1]})
\end{align*}
and 
$$
\|\xi_{n,m,K}\|_{L_2(\Omega)}\le C \|b^n-b^m\|_{\C^\alpha}
$$
for $C=C(p)$.
This implies that 
\begin{equation}\label{i2ss}
\liminf_{K\to\infty}	\liminf_{m\to\infty} I_2(n,m,K)\le \xi_n (1+|x|)(1+[\psi^1]_{p-\var;[0,1]}+[\psi^2]_{p-\var;[0,1]}),
\end{equation}
where we put 
$$
\xi_n:=\liminf_{K\to\infty}	\liminf_{m\to\infty} \xi_{n,m,K}. 
$$
By Fatou's lemma 
\begin{equation}\label{i2ssmom}
\|\xi_{n}\|_{L_2(\Omega)}\le C \|b^n-b\|_{\C^\alpha}
\end{equation}
for $C=C(p)$.

To treat $I_3$ we consider two cases. If $t\notin A^K_\omega$, then 
\begin{equation}\label{phipsi}
\phi_t^K=\psi_t^1.
\end{equation}
 Note that for each $k\in\N$ we have $\psi^1_{t_k(\omega)}(\omega)=\psi^2_{t_k(\omega)}(\omega)$, $\psi^1_{s_k(\omega)}(\omega)=\psi^2_{s_k(\omega)}(\omega)$ thanks to  continuity of $\psi^1,\psi^2$. Therefore, for any $k,m\in\N$ we get 
\begin{align*}
\Bigl|\int_{[s_k(\omega),t_k(\omega)]}(b^m(X^1_r)-b^m(X^2_r))\, dr\Bigr|&\le
\Bigl|\int_{[s_k(\omega),t_k(\omega)]}(b^m(X^1_r)\, dr-(\psi^1_{t_k(\omega)}-\psi^1_{s_k(\omega)})\Bigr|\\
&\phantom{\le}
+ 
\Bigl|\int_{[s_k(\omega),t_k(\omega)]}(b^m(X^2_r)\, dr-(\psi^2_{t_k(\omega)}-\psi^2_{s_k(\omega)})\Bigr|\\
&\le2\|\psi^{1,m}-\psi^1\|_{\C([0,1])}+2\|\psi^{2,m}-\psi^2\|_{\C([0,1])}.
\end{align*}
Recall that $Y^K_r=X^1_r$ for $r\not\in A^K_\omega$ and $Y^K_r=X^2_r$ for $r\in A^K_\omega$. Since  the set $A^K_\omega\cap[0,t]$ is just a union of no more than $K$ intervals $(s_k(\omega),t_k(\omega))$, we deduce using \eqref{phipsi}
\begin{align}\label{abovebound}
\Bigl|\int_0^tb^m(Y^K_r)\,dr-\phi_t^K\Bigr|&\le 
\Bigl|\int_0^tb^m(X^1_r)\,dr-\psi^1_t\Bigr|+\Bigl|\int_{A^K_\omega\cap[0,t]}(b^m(X^1_r)-b^m(X^2_r))\, dr\Bigr|\nn\\
&\le (2K+1) \|\psi^{1,m}-\psi^1\|_{\C([0,1])}+2K \|\psi^{2,m}-\psi^2\|_{\C([0,1])}.
\end{align}
Very similarly, if $t\in A^K_\omega$, then $\phi_t^K=\psi_t^2$.  The  set $[0,t]\setminus A^K_\omega$ is also a union of no more than $K$ intervals $[s'_k(\omega),t'_k(\omega)]$ and for each of such intervals we have  $\psi^1_{t'_k(\omega)}(\omega)=\psi^2_{t'_k(\omega)}(\omega)$, $\psi^1_{s'_k(\omega)}(\omega)=\psi^2_{s'_k(\omega)}(\omega)$ and as above we have
\begin{equation*}
\Bigl|\int_{[s'_k(\omega),t'_k(\omega)]}(b^m(X^1_r)-b^m(X^2_r))\, dr\Bigr|\le
2\|\psi^{1,m}-\psi^1\|_{\C([0,1])}+2\|\psi^{2,m}-\psi^2\|_{\C([0,1])}.
\end{equation*}
Thus, 
\begin{align*}
\Bigl|\int_0^tb^m(Y^K_r)\,dr-\phi_t^K\Bigr|&\le 
\Bigl|\int_0^tb^m(X^2_r)\,dr-\psi^2_t\Bigr|+\Bigl|\int_{[0,t]\setminus A^K_\omega}(b^m(X^1_r)-b^m(X^2_r))\, dr\Bigr|\\
&\le 2K \|\psi^{1,m}-\psi^1\|_{\C([0,1])}+(2K+1) \|\psi^{2,m}-\psi^2\|_{\C([0,1])}.
\end{align*}
Combining this with the above bound \eqref{abovebound} and recalling \eqref{convn}, we get that for any fixed $K\in\N$
\begin{equation}\label{i3ss}
\lim_{m\to\infty} I_3(m,K)\le (K+1)\lim_{m\to\infty} (\|\psi^{1,m}-\psi^1\|_{\C([0,1])}+\|\psi^{2,m}-\psi^2\|_{\C([0,1])})=0
\end{equation}

Now it remains to combine \eqref{i1ss}, \eqref{i2ss}, \eqref{i3ss} and substitute them into \eqref{appe}. We pass to the limit first as $m\to\infty$, then as $K\to\infty$ to get that on the set of full measure $\Omega'$
\begin{equation*}
	\sup_{t\in[0,1]}\Bigl|\int_0^t b^n(Y_r)dr -\phi_t\Bigr|\le  \xi_n (1+|x|)(1+[\psi^1]_{p-\var;[0,1]}+[\psi^2]_{p-\var;[0,1]}).
\end{equation*}
Recalling  \eqref{i2ssmom}, we see that the right-hand side of the above inequality converges to $0$ in probability; hence its left-hand side also converges to 0 in probability.  This contradicts \eqref{contra}. Hence the process $Y:=X^1\wedge X^2$ is indeed a solution to \Gref{x;b}.

Now we apply \cref{t:wu}(i). We see  from definition of $Y$ that $Y\in\Vcl$. Hence $\law(Y)=\law(X^1)$. Then if $f$ is a bounded smooth strictly increasing function $\R\to\R$, $t\in[0,1]$ we have
$$
\E f (Y_t)=\E f (X^1_t\wedge X^2_t)=\E f (X^1_t),
$$ 
which implies that $X^1_t\wedge X^2_t=X^1_t$ a.s. and similarly $X^1_t\wedge X^2_t=X^2_t$ a.s. Thus $X^1_t=X^2_t$ a.s., which by continuity of $X^1$, $X^2$ implies \eqref{rezd1}.
 \end{proof}
 
The conditions of the above lemma are sufficient to establish strong uniqueness under the same conditions as in \cref{t:su}.
\begin{corollary}\label{c:sud1}
		Let $d=1$, $x\in\R$,  $b\in\C^\alpha$, $H\in(0,1/2]$. Suppose that \eqref{alphacond} is satisfied and either \eqref{bcond1} or \eqref{bcond2} holds.
		Let $(X^1,B^H)$, $(X^2,B^H)$ be two solutions to 
		\Gref{x;b} defined on the same probability space and adapted to the same filtration $(\F_t)_{t\in[0,1]}$.  Suppose that $X^1,X^2\in\Vcl$. 
		Then 
		\begin{equation*}
			\P(X^1_t=X^2_t\, \text{ for any }\, t\in[0,1])=1.
		\end{equation*}
\end{corollary}	
\begin{proof}
\textbf{Case (i): condition \eqref{bcond1} holds}.  In this case we can fix $\eps>0$ such that 
\begin{equation}\label{condepsstr}
	1+\alpha H-\eps>\frac12,\qquad (1+\alpha H-\eps)(\alpha+\frac1{2H})>\frac12.
\end{equation}
We would like to apply \cref{l:sud1} with $p:=\frac{1}{1+\alpha H-\eps}\in[1,2]$. We see that  condition  \eqref{condepsstr} can be rewritten as $\alpha+\frac1{2H}>\frac{p}2$ and thus it implies \eqref{condalphap}. We also see that  $X^1,X^2\in\Vcl$ and thus  by \cref{l:apriori} for any $m\ge2$
$$
[ X^1-B^H]_{\C^{1+\alpha H}L_m([0,1])}\le C(1+\|b\|^2_{\C^\alpha}),\qquad
[X^2-B^H]_{\C^{1+\alpha H}L_m([0,1])}\le C(1+\|b\|^2_{\C^\alpha}).
$$
Therefore, by  the Kolmogorov continuity theorem, we have  
\begin{equation*}
	[X^1- B^H]_{p-\var;[0,1]}\le
	[X^1- B^H]_{\C^{1+\alpha H-\eps}([0;1])}<\infty,\quad a.s.,
\end{equation*}
and similarly $[X^2- B^H]_{p-\var;[0,1]}<\infty$ a.s.. Therefore, condition \eqref{pvarb} holds. Thus, all the conditions of  \cref{l:sud1} are satisfied and \eqref{rezd1} implies $X^1=X^2$ a.s.

\textbf{Case (ii): condition \eqref{bcond2} holds}. In this setting we would like to apply \cref{l:sud1} with $p=1$. Then  condition \eqref{condalphap} coincides with our standing assumption \eqref{alphacond}. Further, since the function $b$ is a nonnegative measure, the functions $X^1- B^H$ and $X^2- B^H$ are nondecreasing, which implies \eqref{pvarb}. 
Thus, all the conditions of \cref{l:sud1} are met and therefore $X^1=X^2$ a.s.
\end{proof}

We deduce the strong existence of solutions to \Gref{x;b} from the strong uniqueness just established in \cref{l:sud1} and from weak existence using Gy\"ongy-Krylov's method. We recall the main technical tool.

\begin{proposition}[{\cite[Lemma~1.1]{MR1392450}}]\label{L:GP}
Let $(Z_n)$ be a sequence of random elements in a Polish space $(E, \rho)$ equipped with the Borel $\sigma$-algebra. Assume that for every pair of subsequences
$(Z_{l_k})$ and $(Z_{m_k})$ there exists a further sub-subsequence  $(Z_{l_{k_r}},Z_{m_{k_r}})$ which converges weakly in the space $E\times E$ to a random element $w=(w^1,w^2)$ such that $w^1=w^2$ a.s. 

Then there exists an $E$-valued random element $Z$ such that $(Z_n)$ converges in probability
to $Z$.
\end{proposition}

\begin{proof}[Proof of \cref{t:su}]

\textbf{Strong existence}. We fix as usual $b\in\C^\alpha$, $x\in\R^d$ and put $b_n:=G^d_{1/n}b$. Let $X_n$ be a strong solution to \Gref{x;b_n}. We would like to apply \cref{L:GP} to the sequence $(X_n)_{n\in\N}$. Therefore we consider two arbitrary subsequences $(X'_n)_{n\in\N}$ and $(X''_n)_{n\in\N}$ of this sequence which correspond to functions $(b'_n)_{n\in\N}$ and $(b''_n)_{n\in\N}$.

We note that   by \cref{l:apriori} for any $m\ge2$
$$
[X'_n-B^H]_{\C^{1+\alpha H}L_m([0,1])}\le C(1+\|b\|^2_{\C^\alpha}),\qquad
[X''_n-B^H]_{\C^{1+\alpha H}L_m([0,1])}\le C(1+\|b\|^2_{\C^\alpha})
$$
where $C=C(m)$.
Therefore, Kolmogorov's continuity theorem implies tightness of $(X_n',X_n'',B^H,W)_{n\in\Z_+}$ in the space $\C([0,1],\R^{4d})$. Hence, by Prokhorov's theorem, this sequence has a weakly converging subsequence.  By passing to an appropriate subsequence and applying the Skorokhod representation theorem, we see that there exists a random element $(\wh X',\wh X'',\wh B^H, \wh W)$ and a sequence of random elements $(\wh X'_n,\wh X''_n,\wh B^H_n,\wh W_n)$ defined on a common probability space $(\wh\Omega, \wh\F, \wh P)$
such that $\law(\wh X'_n,\wh X''_n,\wh B^H_n,\wh W_n)=\law( X'_n, X''_n, B^H, W)$ and
\begin{equation*}
	\|(\wh X'_n,\wh X''_n,\wh B^H_n,\wh W_n)-(\wh X',\wh X'',\wh B^H,\wh W)\|_{\C([0,1])}\to0 \,\,\,\text{as $n\to\infty$ a.s.}
\end{equation*}
Consider the  filtration $\wh\F_t:=\sigma(\wh X'_s,\wh X''_s,\wh W_s, s\le t)$, $t\in[0,1]$. It is standard to check (see, e.g, \cite[proof of Theorem~8.2]{GG22}, \cite[proof of Corollary~5.4]{blm23}, \cite[proof of Theorem 2.5]{ART21}) that $\wh W$ is an $(\wh \F_t)$--Brownian motion and thus $B^H$ is $(\wh \F_t)$--fractional Brownian motion. By \cref{t:wu}(ii)
both $(\wh X',\wh B^H)$ and $(\wh X'',\wh B^H)$ are weak solutions to \Gref{x;b} and $\wh X',\wh X''\in\Vcl$. 
By construction,  both $(\wh X',\wh B^H)$ and $(\wh X'',\wh B^H)$ are adapted to the same filtration $(\wh \F_t)$. Thus, all conditions of \cref{c:sud1} hold and $\wh X' = \wh X''$.  

By Gy\"ongy-Krylov's lemma (\cref{L:GP}) we see that the initial sequence $(X_n)_{n\in\N}$ converges in probability to a random element $X$. Therefore, $X$ is adapted to the natural filtration of $B^H$ as a limit of adapted processes. By \cref{t:wu}(ii) $X$ is a solution to \Gref{x;b}, thus $X$ is a strong solution and $X\in\Vcl$.

\textbf{Pathwise uniqueness}. Let $(X^1,B^H)$, $(X^2,B^H)$ be two weak solutions to \Gref{x;b} constructed on the same probability space and adapted to possibly different filtrations $(\F^1_t)_{t\in[0,1]}$ and $(\F^2_t)_{t\in[0,1]}$. Suppose that $X^1,X^2\in\Vcl$. Let $X\in\Vcl$ be a strong solution to the same SDE constructed on the same space, which exists by above. Since $X$ is a strong solution, $X$ also adapted to $(\F^1_t)_{t\in[0,1]}$. Applying now \cref{c:sud1} to the pairs $(X^1,B^H)$, $(X,B^H)$ we see that $X=X^1$ a.s. Similarly, $X=X^2$ a.s.  and thus 
$$
\P(X^1_t=X^2_t \text{ for all $t\in[0,1]$})=1
$$
and pathwise uniqueness holds.
\end{proof}

\section{Proofs of the main results for SHE}\label{s:she}

We fix  $\alpha\in(-\frac32,0)$. Without loss of generality we assume that the length of the time interval $T=1$. Recall the definition of the process $V$ (space-time convolution of white noise with the heat kernel) in \eqref{def.V}.

The weak uniqueness proof for SHE uses the same ideas as the proof of weak uniqueness for SDEs given in \cref{s:sde}. However, a number of changes are needed in order to incorporate the infinite-dimensional nature of the problem. In particular, we have to work with the weighted norm in time and space, apply a different version of Girsanov's theorem, and use backward uniqueness for solutions to the heat equation. Therefore, for the convenience of the reader, we have decided to present these two week uniqueness proofs separately.

For a measurable function  $f\colon[0,1]\times D \to\R$ consider a weighted norm
\begin{equation}\label{stwd}
\|f\|_w:=\sup_{t\in[0,1]}\sup_{x\in D}|P_{2-t}f(x)|.
\end{equation}	
We note that if $P_s f(x)=0$ for some $s>0$ and all $x\in D$, then  $f(x)=0$ for all $x\in D$, see, e.g., \cite[Theorem~2.3.11, p.63]{Evans}. Therefore, $\|f\|_w=0$ if and only if $f\equiv0$ and we see that $\|\cdot\|_w$ is a norm. Let $W_{\|\cdot\|_w\wedge1}$ be the corresponding Wasserstein distance defined as in \eqref{wrho}.
The following stability bound is crucial for the proof.

\begin{lemma}\label{sl:ml}
	Let $u_0\in\BB(D)$, $b\in\C^\alpha(\R)$, $g\in\Lip(\R,\R)$. Let $(\Omega,\F,(\F_t)_{t\in[0,1]},\P)$ be a probability space. Let $(u,W)$ be \textbf{any} weak solution  to  \Sref{u_0;b} in the class $\Vscl$ on this space. Let $(v,W)$ be a strong soluton to  \Sref{u_0;g} defined on the same space. Then there exist constants $C,\eps>0$ which depends only on $\alpha$ and $D$ such that 
	\begin{equation}\label{smainressde}
		\W_{\|\cdot\|_w\wedge1}(\law(u),\law(v))\le C \Gamma \|b-g\|_{\C^{\alpha-\eps}}^{\frac\eps2}.
	\end{equation}
	for 
	$$
	\Gamma:=(1+\|g\|_{\C^{\alpha}}^{\frac{4}{\eps}}+\|b\|_{\C^{\alpha}}^{\frac{4}{\eps}}).
	$$
\end{lemma}

In order to prove \cref{sl:ml}, we fix $u_0\in\BB(D)$, $b\in\C^\alpha(\R)$, $g\in\Lip(\R,\R)$. Let  $(u,W)$ be any weak solution to \Sref{u_0;b} on a space $(\Omega,\F,(\F_t)_{t\in[0,1]},\P)$ in the class $\Vscl$. 
Let $v$ be the strong solution to \Sref{u_0;g}. Let us now implement the generalized coupling strategy with stochastic sewing. We fix  a parameter $\lambda\ge1$ and let $\wt v$ be the solution to the following equation
\begin{equation}\label{SPDEwty}
	\wt v_t(x)=P_t u_0(x) +\int_0^t\int_D p_{t-r}(x,y) \bigl(g(\wt v_r(y))+\lambda(u_r(y)-\wt v_r(y)) \bigr)\,dydr+V_t(x),
\end{equation}	 
where $t\in[0,1]$, $x\in D$. First, we show that $\wt v$ is well-defined.
\begin{lemma}\label{sl:ks0}
Equation \eqref{SPDEwty} has an adapted solution on the same space $(\Omega,\F,(\F_t)_{t\in[0,1]},\P)$. Moreover, ${\|\wt v-V\|_{\Ctimespace{1}{m}{[0,1]}}<\infty}$.
\end{lemma}
\begin{proof}
We apply \cref{l:ebound} with $\D=D$, $s_t=p_t$ (recall \cref{c:conv}), $h(x)=g(x)-\lambda x$, $R(t,x)=V_t(x)+\lambda \int_0^t\int_D p_{t-r}(x,y) u_r(y)\,dydr$, $Z=\wt v$. We see that $h$ is Lipschitz and 
\begin{equation*}
	\|R\|_{\Ctimespacezerom{[0,1]}}\le \|V\|_{\Ctimespacezerom{[0,1]}}+\lambda \|u\|_{\Ctimespacezerom{[0,1]}}<\infty
\end{equation*}
since $u\in\Vscl$. Thus, all the conditions of \cref{l:ebound} are satisfied and applying this lemma we see that equation  \eqref{SPDEwty} has an adapted solution $\wt v$ and 
\begin{align*}
	\|\wt v-V\|_{\C^1L_m([0,1])}&\le \|\wt v-R\|_{\C^1L_m([0,1])}+\|R-V\|_{\C^1L_m([0,1])}\\
	&\le \|\wt v -R\|_{\C^1L_m([0,1])}+ \lambda \|u\|_{\Ctimespacezerom{[0,1]}}<\infty.\qedhere
\end{align*}
\end{proof}

Similar to \cref{s:sde}, we will show that $\law(\wt v)$ is close to $\law(v)$ in the total variation norm, and for each $\omega \in \Omega$, the distance between $\wt v(\omega)$ and $u(\omega)$ is also small. This would imply that $\law(v)$ and $\law(u)$ are close in an appropriate Wasserstein distance, and thus \cref{sl:ml} holds.

Recall the definition of $\|\cdot\|_{\Ctimespacezerom{[0,1]}}$ norm in \eqref{fullsnorm}. Below we denote by $\Leb(D)$
the Lebesgue measure of the set $D$.
\begin{lemma}\label{sl:ks1}
	There exists a constant $C>0$ such that 
	\begin{equation}\label{sdtvbound}
		\dtv(\law( v),\law(\wt v))\le \frac12 \lambda \Leb(D)^{1/2} \|u-\wt v\|_{\Ctimespace{0}{2}{[0,1]}}.
	\end{equation}	
\end{lemma}
\begin{proof}
We argue similarly to the proof of \cref{l:ks1} and make appropriate modifications to extend that result to the infinite-dimensional setting.
	
For $N>0$, $t\in[0,1]$, $x\in D$ define
	\begin{align*}
		&\beta_N(t,x):=\lambda (u_t(x)-\wt v_t(x))\1_{ |u_t(x)-\wt v_t(x)|\le N};\\
		&\wt W_t(f):=W_t(f)+	\int_0^t\int_D f(y) \beta_N(r,y) dydr,\quad f\in L_2(D).
	\end{align*}	
	
We would like to apply  the Girsanov theorem for space-time white noise in the form of  \cite[Theorem~2.7.1, Proposition~2.7.4]{Marta}. We note that $|\beta_N|\le \lambda N$. Hence, the corresponding version of the Novikov condition for SPDEs (\cite[condition~(2.7.17)]{Marta}) holds. Therefore by \cite[Proposition~2.7.4]{Marta} we see that all the conditions of the Girsanov theorem for space-time white noise  \cite[Theorem~2.7.1]{Marta} are satisfied. Thus, there exists an equivalent probability measure $\wt \P^{N}$ such that the process $\wt W$ is an $(\F_t)_{t\in[0,1]}$-white noise  under this measure and by \cite[formula~(2.7.2)]{Marta}
\begin{equation}\label{sdenratio}
	\frac{d \wt \P^{N}}{d \P}=\exp\Bigl(-\int_0^1\int_D \beta_N(r,y) W(dr,dy)-\frac12\int_0^1\int_D \beta_N^2(r,y)\,dydr \Bigr).
\end{equation}	

Let  $(\wt v_N,\wt W)$ be a strong solution to \Sref{u_0;g} on $(\Omega,\F,\wt \P^N)$. That is, let $\wt v_N$ be the solution of the following equation
\begin{equation*}
\wt v_N(t,x)=P_t u_0(x) +\int_0^t\int_D p_{t-r}(x,y) g(\wt v_N(r,y))\,dydr+\int_{0}^t\int_D p_{t-r}(x,y)\wt W(dr,dy),
\end{equation*}
where $t\in[0,1],\,x\in D$.

To get the desired bound \eqref{sdtvbound}, we note first that by the triangle inequality,
\begin{equation}\label{streq}
	\dtv(\law_\P( v),\law_\P(\wt v))\le \dtv(\law_\P( v),\law_\P(\wt v_N))+\dtv(\law_\P(\wt v_N),\law_\P(\wt v)).
\end{equation}

We recall that $(v,W)$ is a strong solution to \Sref{u_0;g} on $(\Omega,\F,\P)$ Therefore, by weak uniqueness for \Sref{u_0;g}  (recall that $g\in\Lip(\R^d,\R^d)$) we have $\law_{\P}(v)=\law_{\wt \P^N}(\wt v_N)$.
Thus, using again Pinsker's inequality and the expression for the density \eqref{sdenratio}, we can bound the first term in the right-hand side of \eqref{streq} in the following way
\begin{align}\label{streq1}
	\dtv(\law_{\P}( v),\law_{\P}(\wt v_N))&=
	\dtv(\law_{\wt \P^N}(\wt v_N),\law_{\P}(\wt v_N))\le \dtv(\P,\wt\P^N) \nn\\
	&\le \frac1{\sqrt2}\Bigl(\E^{\P}\log \frac{d \P}{d \wt \P^{N}}\Bigr)^{1/2}=\frac12\Bigl(\int_0^1\int_D \E^\P\beta_N^2(r,y)\,dydr\Bigr)^{1/2}\nn\\
	&\le \frac12\lambda \Bigl(\int_0^1\int_D \E^\P (u_r(y)-\wt v_r(y))^2\,dydr\Bigr)^{1/2}.
\end{align}	
Here we used that $|\beta_t^N|\le C \lambda N$ to conclude that the expected value of the stochastic integral is $0$.
	
To bound the second term in \eqref{streq}, we use \cite[Proposition 2.7.3]{Marta} to rewrite equation for $\wt v_N$ on $(\Omega,\F,\wt \P^N)$ as
\begin{align*}
	 \wt v_N(t,x)=&P_t u_0 (x)+\int_0^t\int_D p_{t-r}(x,y) \bigl(g(\wt v_N(r,y))+\lambda(u_r(y)-\wt v_r(y))\1_{ |u_t(x)-\wt v_t(x)|\le N}\bigr)\,dydr\\
	 &+\int_{0}^t\int_D p_{t-r}(x,y)W(dr,dy),\quad t\in[0,1],\, x\in D.
\end{align*}
Recall equation \eqref{SPDEwty} for $\wt v$. We have on the event $\{\|u-\wt v\|_{\C([0,1]\times D)}\le N\}$ for $t\in[0,1]$
\begin{align*}
\|\wt v(t)-\wt v_N(t)\|_{\C(D)}&\le \|g\|_{\Lip(\R)}\int_0^t\int_D p_{t-r}(x,y) \|\wt v(r)-\wt v_N(r)\|_{\C(D)}\,dydr\\
&=
\|g\|_{\Lip(\R)}\int_0^t\|\wt v(r)-\wt v_N(r)\|_{\C(D)}\,dr.
\end{align*}
Therefore, by Gronwall's lemma, $\wt v_N=\wt v$ on  $\{\|u-\wt v\|_{\C([0,1]\times D)}\le N\}$. Hence we can bound the second term in \eqref{streq} as 
	\begin{equation*}
		\dtv(\law_\P(\wt v_N),\law_\P(\wt v))\le \P(\|u-\wt v\|_{\C([0,1]\times D)}>N).
	\end{equation*}

Now we combine this with \eqref{streq1} and substitute into \eqref{streq}. We get 
	\begin{equation*}
			\dtv(\law_\P( v),\law_\P(\wt v))\le  \frac12 \lambda \, \Leb(D)^{1/2} \|u-\wt v\|^2_{\Ctimespace{0}{2}{[0,1]}}+
	 \P(\|u-\wt v\|_{\C([0,1]\times D)}>N), 
	\end{equation*}
	for $C>0$. Since $N$ was arbitrary, by passing to the limit as $N\to\infty$, we get \eqref{sdtvbound}.
\end{proof}

The next lemma shows that $\wt v$ is close to $u$ in space. Before we continue, let us explain how the proof differs from the proof of \cref{L:ks2}, which establishes a bound of a similar type for the SDE case. In the first part of the proof (Steps 1--3), we bound
\begin{equation}\label{expl}
\sup_{t\in [0,1]}\sup_{x\in D}\|u_{t,x}-\wt v_{t,x}\|_{L_m(\Omega)}.
\end{equation}
This is done along more or less the same lines as in the proof of \cref{L:ks2}. However, the second part of the proof, which shows that the supremums can be squeezed inside the expectation, uses quite different arguments from the ones used in the proof of \cref{L:ks2}. Indeed, in the proof of \cref{L:ks2}, to estimate $\| \sup_{t\in [0,1]}|u_{t}-\wt v_{t}|\|_{L_m(\Omega)}$, we first bounded $\| (u_{t}-\wt v_{t})-(u_{s}-\wt v_{s})\|_{L_m(\Omega)}$ for $s,t\in[0,1]$ in terms of $|t-s|^\eps$ and then applied the Kolmogorov contniuity theorem. An analogue of this here would be to bound
$$
\| (u_{t,x}-\wt v_{t,x})-(u_{s,y}-\wt v_{s,y})\|_{L_m(\Omega)},
$$
where $s,t\in[0,1]$, $x,y\in D$, in terms of $|t-s|^\eps+|x-y|^\eps$ and then apply the Kolmogorov theorem.  However, such a bound seems quite tricky to obtain. All we have is a weaker bound on
\begin{equation}\label{expl2}
\| (u_{t,x}-\wt v_{t,x})-P_{t-s}(u_{s,x}-\wt v_{s,x})\|_{L_m(\Omega)},
\end{equation}
in terms of $|t-s|^\eps$, see \eqref{holdspde}. This is not sufficient to show that the operations of supremum and taking expectation in \eqref{expl} can be interchanged. Nevertheless, by introducing weighted processes in time and space, we show that the bound on \eqref{expl2} allows us to bound the supremum norm of these processes and eventually leads to \eqref{smainboundxmy}. This is the subject of Steps 4--5 of the proof.
\begin{lemma}\label{sl:ks2}
	For any $\eps>0$ such that 
	\begin{equation}\label{sepscond}
	\alpha+\frac32>24\eps,
	\end{equation}
	there exist constants $C=C(\eps),C_0=C_0(\eps)>1$ such that for any
	\begin{equation}\label{slambdacond}
	\lambda> C_0(1+\|g\|_{\C^\alpha}^{\frac3\eps}+ \|b\|_{\C^{\alpha}}^{\frac3\eps})
	\end{equation}
	one has 
	\begin{align}
			\label{spdeee} 
&\|u-\wt v\|_{\Ctimespace{0}{2}{[0,1]}}\le  C  \lambda^{-1-\eps}+C \|b-g\|_{\C^{\alpha-\eps}};\\
\label{smainboundxmy}
	&\bigl\| \sup_{\substack{t\in[0,1]\\ x\in D}}|P_{2-t}(u_t-\wt v_t)(x) |\bigr\|_{L_2(\Omega)}\le C\lambda^{-1}+C\lambda^\eps \|b-g\|_{\C^{\alpha-\eps}}.
	\end{align}
\end{lemma}	
We see that since $\alpha>-3/2$, one can choose $\eps>0$ such that condition \eqref{sepscond} is satisfied. 

\begin{proof}
Fix  $\eps>0$ satisfying \eqref{sepscond}. Denote $m:=2/\eps\ge2$ and let 
\begin{equation}\label{sgmdef}
\Theta:=1+\|g\|_{\C^\alpha}^3+ \|b\|_{\C^{\alpha}}^3.
\end{equation}
	Put
	\begin{equation*}
		\phi(t,x):=u(t,x)-V(t,x),\qquad \psi(t,x):=\wt v(t,x)- V(t,x).
	\end{equation*}
	Let $(b_n)_{n\in\Z_+}$ be a sequence of $\C^\infty(\R,\R)$ functions  converging to $b$ in $\C^{\alpha-}$.  For $n\in\Z_+$ define
	\begin{equation*}
	z^n_t=P_t u_0(x) +\int_0^t\int_D p_{t-r}(x,y) b_n(u_r(y))\,dydr+V_t(x).\quad t\in[0,1],\,x\in D.
	\end{equation*}
	It follows from the  definition of a solution to \Sref{u_0;b} (\cref{def:ssol}) that by passing to a subsequence if necessary, we have 
	\begin{equation}\label{spdeconv}
		\|z^n - u\|_{\C([0,1]\times D)}\to0\quad a.s.,\quad  \text{as $n\to\infty$}. 	
	\end{equation}	
	
	\textbf{Step~1}. 
	Let $t\in[0,1]$, $x\in D$. Then for any $T>t$ we have
	\begin{equation*}
	P_{T-t}(z_t^n-\wt v_t)(x)=\int_0^t\int_D p_{T-r}(x,y) \bigl(b_n(u_r(y))+\lambda(\wt v_r(y)-u_r(y)) -g(\wt v_r(y))\bigr)\,dydr.
	\end{equation*}
	Therefore, we derive 	
	\begin{align*}
		e^{\lambda t}P_{T-t}(z_t^n-\wt v_t)(x)&=
		\lambda\int_0^t e^{\lambda r } P_{T-r}(z_r^n-\wt v_r)(x)\,dr\\
		&\phantom{=}+
		\int_0^t\int_D e^{\lambda r } p_{T-r}(x,y) \bigl(b_n(u_r(y))+\lambda(\wt v_r(y)-u_r(y)) -g(\wt v_r(y))\bigr)\,dydr\\
		&= \int_0^t\int_D e^{\lambda r } p_{T-r}(x,y)\bigl(b_n(u_r(y))-g(\wt v_r(y))
		+\lambda(z_r^n(y)-u_r(y))\bigr)\,dy dr.
	\end{align*}
	Thus, by letting $T\searrow t$ and using continuity of $z^n$, $\wt v$ we get for $t\in[0,1]$, $x\in D$
	\begin{align*}
		z_t^n(x)-\wt v_t(x)&= \int_0^t\int_D e^{-\lambda (t- r) } p_{t-r}(x,y)\bigl(b_n(u_r(y))-g(\wt v_r(y))
		+\lambda(z_r^n(y)-u_r(y))\bigr)\,dy dr
	\end{align*}
	
This implies for $(s,t)\in\Delta_{[0,1]}$, $x\in D$
	\begin{align*}
		&|(z^n_t(x)-\wt v_t(x))-e^{-\lambda(t-s)}P_{t-s}(z^n_s-\wt v_s)(x)|\\
		&\qquad\le\Bigl|\int^t_s\int_D e^{-\lambda (t-r)}p_{t-r}(x,y) \bigl(b_n(u_r(y))-g(\wt v_r(y))\bigr)\,dydr\Bigr|+\lambda
		\Leb(D)\|z^n-u\|_{\C([0,1]\times D)} \\
		&\qquad\le \Bigl|\int^t_s\int_D e^{-\lambda (t-r)}p_{t-r}(x,y) \bigl(g(u_r(y))-g(\wt v_r(y))\bigr)\,dydr\Bigr|\\
	&\qquad\phantom{\le}		+\Bigl|\int_s^t\int_D e^{-\lambda (t-r)}p_{t-r}(x,y)\bigl(b_n(u_r(y))-g(u_r(y))\bigr) \, dydr\Bigr|\\
		&\qquad\phantom{\le}+\lambda\Leb(D)\|z^n-u\|_{\C([0,1]\times D)}\\
		&\qquad=: I_{1}(s,t,x)+I_{2,n}(s,t,x)+\lambda\Leb(D)\|z^n-u\|_{\C([0,1]\times D)}.
	\end{align*}
	We pass to the limit in the above inequality as $n\to\infty$. Recalling \eqref{spdeconv}, we get	
	\begin{equation}\label{slimeq}
		|(u_t(x)-\wt v_t(x))-e^{-\lambda(t-s)}P_{t-s}(u_s-\wt v_s)(x)|\le  I_{1}(s,t,x)+\liminf_{n\to\infty} I_{2,n}(s,t,x).
	\end{equation}

	\textbf{Step~2}. 	Let us now analyze each of the terms in the above inequality. 	
	
	We begin with $I_{1}$. We apply  \cref{c:sfinfin} with $\beta=\alpha$, $\tau=5/8$, $\mu=\frac12$,  $f=g$. We see that conditions \eqref{sparam} and 
	\eqref{sparamcond2} hold thanks to our choice of parameters and the standing assumption \eqref{spdecond}. Therefore,  we get by \eqref{spde2w}
	\begin{align}\label{sivanpred}
		\|I_{1}(s,t,x)\|_{L_m(\Omega)}&\le C \|g\|_{\C^\alpha}
		(t-s)^\eps\lambda^{-\frac78-\frac\alpha4+3\eps}\|u-\wt v\|_{\C^{0,0}L_m([0,1])}^{\frac12}\nn\\
		&\quad+C\|g\|_{\C^\alpha}(t-s)^\eps\lambda^{-\frac{11}8-\frac\alpha4+3\eps}\bigl([\phi]_{ \C^{\frac{5}8,0}L_m([0,1])} +[\psi]_{ \C^{\frac{5}8,0}L_m([0,1])}\bigr)\nn\\
		&\le C\|g\|_{\C^\alpha}
		(t-s)^\eps\lambda^{-\frac12-2\eps}\|u-\wt v\|_{\C^{0,0}L_m([0,1])}^{\frac12}\nn\\
		&\quad+C\|g\|_{\C^\alpha}(t-s)^\eps\lambda^{-1-2\eps}\bigl([\phi]_{ \C^{\frac{5}8,0}L_m([0,1])} +[\psi]_{ \C^{\frac{5}8,0}L_m([0,1])}\bigr),
 	\end{align}
	where the last inequality follows from \eqref{sepscond}. Here  $C=C(\eps)$ is independent of $\lambda$. We bound $[\phi]_{\Ctimespace{\frac58}{m}{[0,1]}}$ using 
	\cref{l:apspde}. Since $\alpha>-3/2$, we get
	\begin{equation}\label{phibspde}
		[\phi]_{\Ctimespace{\frac58}{m}{[0,1]}}\le
		[\phi]_{\Ctimespace{1+\frac\alpha4}{m}{[0,1]}}\le C(1+\|b\|^2_{\C^\alpha}).
	\end{equation}	
	Next, to bound $[\psi]_{\Ctimespace{\frac58}{m}{[0,1]}}$,  we apply \cref{l:sapriori}  with $f=g$, $\rho(t,x)=\lambda (u(t,x)-\wt v(t,x))$, $Z=\wt v$. By \cref{sl:ks0}, we have $[\wt v-V]_{\Ctimespace{1}{m}{[0,1]}}<\infty$ and hence condition \eqref{zspdecond} holds. We get from \eqref{claimspde}
	\begin{equation}\label{psibspde}
		[\psi]_{\Ctimespace{\frac58}{m}{[0,1]}}\le
		[\psi]_{\Ctimespace{1+\frac\alpha4}{m}{[0,1]}}\le C(1+\|g\|^2_{\C^\alpha})(1+\lambda \|u-\wt v\|_{\Ctimespacezerom{[0,1]}}).
	\end{equation}
	We combine now \eqref{phibspde} and \eqref{psibspde} with \eqref{sivanpred} and use again the  inequality $ xy\le x^2+y^2$ valid for any $x,y\in\R$ to bound the first term in \eqref{sivanpred}. Recalling the definition of $\Theta$ in \eqref{sgmdef}, we obtain
	\begin{align}\label{sivanpred2}
		\|I_{1}(s,t,x)\|_{L_m(\Omega)}&\le C(t-s)^\eps \lambda^{-2\eps }\|u-\wt v\|_{\Ctimespacezerom{[0,1]}}+ C (t-s)^\eps\|g\|_{\C^\alpha}^2 \lambda^{-1-2\eps}\nn\\
		&\phantom{\le}+C \|g\|_{\C^\alpha}(1+\|b\|_{\C^\alpha}^2+\|g\|_{\C^\alpha}^2)(t-s)^\eps\lambda^{-1-2\eps}\nn\\
		&\phantom{\le}+C\|g\|_{\C^\alpha}(1+\|g\|_{\C^\alpha}^2)(t-s)^\eps\lambda^{-2\eps}\|u-\wt v\|_{\Ctimespacezerom{[0,1]}}\nn\\
		&\le C \Theta(t-s)^\eps \lambda^{-2\eps}\|u-\wt v\|_{\Ctimespacezerom{[0,1]}}+ C\Theta (t-s)^\eps \lambda^{-1-2\eps}.
	\end{align}
	
	Now we continue with the analysis of $I_{2,n}$. We apply  \cref{c:sfinfin} with $\beta=\alpha-\eps$, $\tau=5/8$, $f=b_n-g$. We see that $\alpha-\eps>1-4\tau=-3/2$ thanks to \eqref{sepscond}, and thus condition \eqref{sparam} holds. Therefore, we get  
	from \eqref{spdew} and \eqref{phibspde}
	\begin{align*}
		\|I_{2,n}(s,t,x)\|_{L_m(\Omega)}&\le 	C \|b_n-g\|_{\C^{\alpha-\eps}}(t-s)^\eps		\lambda^{-1-\frac\alpha4+4\eps}\Bigl(1+[\phi]_{ \Ctimespace{\frac58}{m}{[0,1]}}\lambda^{-\frac38}\Bigr)\nn\\
		&\le C \|b_n-g\|_{\C^{\alpha-\eps}}(t-s)^\eps	\nn\\
		&\phantom{\le}+C (\|b_n\|_{\C^{\alpha-\eps}}+\|g\|_{\C^{\alpha-\eps}})(t-s)^\eps(1+\|b\|^2_{\C^\alpha})\lambda^{-1-2\eps}.
	\end{align*}
	for $C=C(\eps)$ independent of $\lambda$ and $n$.
	
	Now we substitute this together with \eqref{sivanpred2} into \eqref{slimeq}. We note that $\|b_n -g\|_{\C^{\alpha-\eps}}\to \|b -g\|_{\C^{\alpha-\eps}}$ as $n\to\infty$. Therefore, by Fatou's lemma we derive
	for any $(s,t)\in\Delta_{[0,1]}$, $x\in D$
	\begin{align}\label{sxmy1}
		&\|(u_t(x)-\wt v_t(x))-e^{-\lambda(t-s)}P_{t-s}(u_s-\wt v_s)(x)\|_{L_m (\Omega)}\nn\\
		&\quad\le C \Theta(t-s)^\eps \lambda^{-2\eps}\|u-\wt v\|_{\Ctimespacezerom{[0,1]}}+ C\Theta(t-s)^\eps  \lambda^{-1-2\eps}+C(t-s)^\eps \|b-g\|_{\C^{\alpha-\eps}},
	\end{align}
	where $C=C(\eps)$ and we used again the definition of  constant $\Theta$ in \eqref{sgmdef}.	
	
	\textbf{Step~3: buckling for the supremum in time-space norm}. 
	Choose now any $\lambda\ge1$ such that 
	\begin{equation}\label{sxmy2}
		C\Theta\lambda^{-\eps}\le\frac12,
	\end{equation}
	where $C$ is as in \eqref{sxmy1}. We choose $s=0$ in \eqref{sxmy1}, use that $u(0)=\wt v(0)=u_0$ take supremum over $t\in[0,1]$, $x\in D$ and derive
		\begin{equation*}
		\|u-\wt v\|_{\Ctimespacezerom{[0,1]}}\le \frac12\|u-\wt v\|_{\Ctimespacezerom{[0,1]}}+   \lambda^{-1-\eps}+C \|b-g\|_{\C^{\alpha-\eps}},
	\end{equation*}
	where we used that $C \Theta\le \frac12 \lambda^{\eps}$.
	We note that $\|u-V\|_{\Ctimespacezerom{[0,1]}}<\infty$ since $u\in\Vscl$. We see also that by \cref{sl:ks0} $\|\wt v-V\|_{\Ctimespacezerom{[0,1]}}<\infty$. Therefore $\|u-\wt v\|_{\Ctimespacezerom{[0,1]}}<\infty$, and we get from the above inequality 
		\begin{equation}\label{spdefinal} 
	\|u-\wt v\|_{\Ctimespacezerom{[0,1]}}\le  C  \lambda^{-1-\eps}+C \|b-g\|_{\C^{\alpha-\eps}}.
\end{equation}
	for $C=C(\eps)$. This shows \eqref{spdeee}.
	
	\textbf{Step~4: bounding the H\"older in time norm}. 
	Now we substitute \eqref{spdefinal} into \eqref{sxmy1}. We use again \eqref{sxmy2} to replace $\Theta$ by $\lambda^{\eps}$. We get 	for any $(s,t)\in\Delta_{[0,1]}$, $x\in D$
	\begin{align}\label{holdspde}
	&\|(u_t(x)-\wt v_t(x))-e^{-\lambda(t-s)}P_{t-s}(u_s-\wt v_s)(x)\|_{L_m (\Omega)}\nn\\
	&\qquad \le C(t-s)^\eps  \lambda^{-1-\eps}+C(t-s)^\eps \|b-g\|_{\C^{\alpha-\eps}}
	\end{align}
	for $C=C(\eps)$.
	
%
	\textbf{Step~5: derivation of the final bound.} 
	Now we have all the ingredients to derive \eqref{smainboundxmy}. We put 
	\begin{equation*}
	\uu(t,x):=P_{2-t}u_t(x),\qquad\wt{\textbf{v}}(t,x):=P_{2-t}\wt v_t(x),\qquad t\in[0,1], x\in D.
	\end{equation*}	
	Using Jensen's inequality, for  $s,t\in\Delta_{[0,1]}$, $x\in D$ we can  rewrite \eqref{holdspde} as 
		\begin{align}\label{prevholder}
		&\|(\uu_t(x)-\vv_t(x))-(\uu_s(x)-\vv_s(x))\|_{L_m (\Omega)}\nn\\
		&\quad\le \sup_{x\in D}\|(u_t(x)-\wt v_t(x))-P_{t-s}(u_s-\wt v_s)(x)\|_{L_m (\Omega)}\nn\\
		&\quad \le C(t-s)^\eps  \lambda^{-1-\eps}+C(t-s)^\eps \|b-g\|_{\C^{\alpha-\eps}}+
		(1-e^{-\lambda(t-s)})\sup_{x\in D}\|P_{t-s}(u_s-\wt v_s)(x)\|_{L_m (\Omega)}\nn\\
		&\quad \le C(t-s)^\eps  \lambda^{-1-\eps}+C(t-s)^\eps \|b-g\|_{\C^{\alpha-\eps}}+
		\lambda^\eps(t-s)^\eps\|u-\wt v\|_{\Ctimespacezerom{[0,1]}}\nn\\
		&\quad \le C(t-s)^\eps ( \lambda^{-1}+\lambda^\eps \|b-g\|_{\C^{\alpha-\eps}}),
	\end{align}
	where the last inequality follows from \eqref{spdefinal}.
	
	Next, we apply 	\cref{p:spacess}. We get for  $t\in\Delta_{[0,1]}$, $x,y\in D$
		\begin{align*}
		\|(\uu_t(x)-\vv_t(x))-(\uu_t(y)-\vv_t(y))\|_{L_m (\Omega)}&\le C (2-t)^{-1/2} |x-y| \|u-\wt v\|_{\Ctimespacezerom{[0,1]}}\\
		&\le C |x-y|( \lambda^{-1-\eps}+ \|b-g\|_{\C^{\alpha-\eps}}).
	\end{align*}
	where we used again  \eqref{spdefinal} and that $(2-t)^{-1/2}\le 1$ because $t\le1$. Combining this with \eqref{prevholder}, we get for any $s,t\in[0,1]$, $x,y\in D$	
	\begin{equation*}
		\|(\uu_t(x)-\vv_t(x))-(\uu_s(y)-\vv_s(y))\|_{L_m (\Omega)} \le C\bigl((t-s)^\eps+|x-y|\bigr) ( \lambda^{-1}+\lambda^\eps \|b-g\|_{\C^{\alpha-\eps}}).
	\end{equation*}
	
	Since $m\eps>1$ and the processes $\uu$ and $\vv$ are continuous, we can apply the Kolmogorov continuity theorem to derive	
	\begin{equation*}
		\|\sup_{t\in[0,1]}|\uu_t-\vv_t|\|_{L_m (\Omega)}\le C\lambda^{-1}+C\lambda^\eps \|b-g\|_{\C^{\alpha-\eps}}.
	\end{equation*}
	for $C=C(\eps)$, which is  the desired bound \eqref{smainboundxmy}. Here we also used that $\uu_0\equiv\vv_0\equiv0$.
\end{proof}

Now we can combine the bounds of \cref{sl:ks1,sl:ks2} in order to prove \cref{sl:ml}.

\begin{proof}[Proof of \cref{sl:ml}] 
Take $\eps>0$ satisfying \eqref{sepscond} and $\lambda>1$ satisfying  \eqref{slambdacond}. 	By the triangle inequality, \cref{sl:ks1} and \cref{sl:ks2} we have
	\begin{align}\label{spdemain}
		\W_{\|\cdot\|_w\wedge1}(\law(u),\law(v))&\le \W_{\|\cdot\|_w\wedge 1}(\law(u),\law(\wt v))+
		\W_{\|\cdot\|_w\wedge 1}(\law(\wt v),\law(v))\nn\\
		&\le \bigl\|\,\|u-\wt v\|_{w}\bigr\|_{L_1(\Omega)}+ \dtv(\law(\wt v),\law(v))\nn\\
		&\le C \lambda^{-1}+C\lambda^{\eps} \|b-g\|_{\C^{\alpha-\eps}}+C \lambda  \|u-\wt v\|_{\Ctimespace{0}{2}{[0,1]}}\nn\\
		&\le C \lambda^{-\eps}+C\lambda \|b-g\|_{\C^{\alpha-\eps}}.
	\end{align}
Let $C_0=C_0(\eps)$ be as in \eqref{slambdacond}.  We choose 
	\begin{equation*}
		\lambda:=\|b-g\|_{\C^{\alpha-\eps}}^{-1/2}+C_0 (1+\|g\|_{\C^\alpha}^{\frac3\eps}+ \|b\|_{\C^{\alpha}}^{\frac3\eps}).
	\end{equation*}
	Clearly such $\lambda$ satisfies \eqref{slambdacond}.  Substituting this $\lambda$ into \eqref{spdemain} we get 
	\begin{align*}
		\W_{\|\cdot\|_w\wedge1}(\law(u),\law(v))&\le C  \|b-g\|_{\C^{\alpha-\eps}}^{\frac\eps2} 	+C  \|b-g\|_{\C^{\alpha-\eps}}^{\frac12}\\
		&\phantom{\le} 	+C C_0 (1+\|g\|_{\C^\alpha}^{\frac3\eps}+ \|b\|_{\C^{\alpha}}^{\frac3\eps}) \|b-g\|_{\C^{\alpha-\eps}},
	\end{align*}
	which implies \eqref{smainressde}.
\end{proof}	

Now we are ready to complete the weak uniqueness proof.
\begin{proof}[Proof of \cref{t:swu}]
	(i). Fix $b\in\C^\alpha$, $u_0\in\BB(D)$ and let $b^n:=G_{1/n}b$, $n\in\N$. 
	Let $u^n$ be the strong solution to \Sref{u_0;b^n}. Let $(u,W)$ be a weak solution to \Sref{u_0;b} in the class $\Vscl$. 	
	By \cref{sl:ml},  there exist constants  $C,\eps>0$ such that for any $n\in\N$
	\begin{equation*}
\W_{\|\cdot\|_w\wedge1}(\law(u),\law(u^n))\le C (1+\|b\|_{\C^\alpha}^{\frac{4}\eps}+\|b^n\|_{\C^\alpha}^{\frac{4}\eps})\|b-b^n\|_{\C^{\alpha-\eps}}^{\frac{\eps}2}.
	\end{equation*}
	By passing to the limit as $n\to\infty$ in the above inequality, we get
	\begin{equation*}
			\lim_{n\to\infty}\W_{\|\cdot\|_w\wedge1}(\law(u),\law(u^n))=0.
	\end{equation*}	
If $(\overline{u},\overline{W})$ is  another weak solution to \Sref{u_0;b},
	then by above 
	\begin{equation*}
		\lim_{n\to\infty}\W_{\|\cdot\|_w\wedge1}(\law(\overline{u}),\law(u^n))=0,
	\end{equation*}
	which implies that $\W_{\|\cdot\|_w\wedge1}(\law(\overline{u}),\law(u))=0$ and thus 
	$\W_{\|\cdot\|\wedge1}(\law(\overline{u}),\law(u))=0$ (see again \cite[Theorem~2.3.11, p.63]{Evans}). Thus $\law(u)=\law(\overline{u})$ and thus weak uniqueness holds for \Sref{u_0;b}.

	(ii). By \cite[Proposition~3.3]{ABLM}, the sequence $\law(u^n)$ is tight in $\C([0,1]\times D,\R)$ and by \cite[Proposition~3.4]{ABLM} any of its limiting points is a weak solution to \Sref{u_0;b} and lies in $\Vscl$. By part (i) of the theorem all of the limiting points have the same law. Therefore the sequence $\law(u^n)$ weakly converges to the unique weak solution of \Sref{u_0;b}.
\end{proof}

\appendix \section{Proofs of technical results} \label{a:app}

\begin{proof}[Proof of \cref{l:sappr}]
Fix $N\in\N$, $(s,t)\in\Delta_{[0,T]}$, $x\in\D$. Denote the approximation error
\begin{align*}
Err_N:=&\int_s^t\int_\D s_{t-r}(x,y)w_r f(Y_r(y)+\phi_r(y))\,dydr	\\
&-\sum_{i=0}^{k(N)-1}\int_{t^N_i}^{t_{i+1}^N}\int_\D s_{t-r}(x,y)w_r \E^{t^N_i}[f(Y_r(y)+S_{r-t^N_i}\phi_{t^N_i}(y))]\,dydr.
\end{align*}	
Then, by the triangle inequality
\begin{align}\label{sconv}
		\E|Err_N|&\le \E\Bigl|\sum_{i=0}^{k(N)-1}\int_{t^N_i}^{t_{i+1}^N}\int_\D s_{t-r}(x,y)w_r\bigl(f(Y_r(y)+\phi_{r}(y))- \E^{t^N_i}[f(Y_r(y)+\phi_{r}(y))]\bigr)\,dydr\Bigr|\nn\\
		&\phantom{\le}+
		\sum_{i=0}^{k(N)-1}\int_{t^N_i}^{t_{i+1}^N}\int_\D s_{t-r}(x,y)w_r \E\Bigl| f(Y_r(y)+\phi_r(y))-f(Y_r(y)+S_{r-t^N_i}\phi_{t^N_i}(y))\Bigr|\,dydr\nn\\
		&=:\E\Bigl|\sum_{i=0}^{k(N)-1} I_{1,i}\Bigr|+\sum_{i=0}^{k(N)-1} I_{2,i}.
	\end{align}
We note that the sequence $(I_{1,i})$ is a martingale difference sequence with respect to the filtration 
$(\G_{i}):=(\F_{t_{i+1}^N})$. Indeed, each $I_{1,i}$ is $\G_i$ measurable, and $\E [I_{1,i}|\G_{i-1}]=0$. Therefore, by the Burkholder-Davis-Gundy inequality, we have 
	\begin{align}\label{si1n}
		\E\Bigl|\sum_{i=0}^{k(N)-1} I_{1,i}\Bigr|^2&\le C \sum_{i=0}^{k(N)-1} \E I_{1,i}^2\le C \|f\|^2_{\C(\R^d)}\|w\|_{\C([s,t])}^2\sum_{i=0}^{k(N)-1} (t_{i+1}^N-t_{i}^N)^2\nn\\
		&\le  C T \|f\|^2_{\C(\R)}\|w\|_{\C([s,t])}^2|\Pi_N|\to0\quad \text{as $N\to\infty$}.
	\end{align}
Here we used that the kernel $s$ integrates to $1$ by \eqref{convprop}. 
Next, using \eqref{stechncond}, we easily deduce for any $i=0,\hdots, k(N)-1$
	\begin{align*}
		I_{2,i}&\le \|w\|_{\C([s,t])} \|f\|_{\C^\eps(\R)}\int_{t^N_i}^{t_{i+1}^N}\int_\D s_{t-r}(x,y) \E[|\phi_{r}(y)-S_{r-t_i^N}\phi_{t_i^N}(y)|^\eps]\,dydr\\
		&\le \|w\|_{\C([s,t])} \|f\|_{\C^\eps(\R)}[\phi]_{\Ctimespaced{\eps}{1}{[s,t]}}\int_{t^N_i}^{t_{i+1}^N}\int_\D s_{t-r}(x,y) |r-t_i^N|^{\eps^2}\,dydr\\
		&\le \|w\|_{\C([s,t])} \|f\|_{\C^\eps(\R)}[\phi]_{\Ctimespaced{\eps}{1}{[s,t]}}(t^N_{i+1}-t_{i}^N)^{1+\eps^2},
	\end{align*}	
	where we used again that $s$ integrates on $\D$ to $1$. 
	This yields 
	\begin{equation*}
		\sum_{i=0}^{k(N)-1} I_{2,i}\le \|w\|_{\C([s,t])} \|f\|_{\C^\eps(\R)}[\phi]_{\Ctimespaced{\eps}{1}{[s,t]}} T |\Pi_N|^{\eps^2}\to0\quad \text{as $N\to\infty$}.
	\end{equation*}
	 Combining this with \eqref{si1n} and substituting into \eqref{sconv}, we get \eqref{spde:s3}.
\end{proof}	

\begin{proof}[Proof of \cref{l:ebound}]
	Fix $m\ge1$. 
	Define
	\begin{equation*}
		\phi:=Z-R.	
	\end{equation*}	
	Then \eqref{SPDEgeneric} can be rewritten as
	\begin{equation}\label{wtyeq}
		\phi(t,x)=S_tz(x) +\int_0^t\int_{\D} s_{t-r}(x,y) h(\phi(r,y)+R(r,y)) dydr,\quad t\in[0,1],\,\,x\in\D.
	\end{equation}
	Consider now a  sequence of Picard iterations: $\phi^{(0)}(t,x):=S_tz(x)$, and for $n\in\Z_+$ put
	\begin{equation*}
		\phi^{(n+1)}(t,x):=S_tz(x)+\int_0^t\int_{\D}s_{t-r}(x,y) h(\phi^{(n)}(r,y)+R(r,y))\,dydr, \,\,  t\in[0,1],\,\,x\in\D.
	\end{equation*}
	It is immediate that for $n\ge1$
	\begin{align*}
		\|\phi^{(n+1)}(t,x)-\phi^{(n)}(t,x)\|_{\lm}&\le [h]_{\Lip(\R^d)}\int_0^t\int_{\D}s_{t-r}(x,y) \|\phi^{(n)}(r,y)-\phi^{(n-1)}(r,y)\|_{\lm}\,dr\\
		&\le [h]_{\Lip(\R^d)}\int_0^t\sup_{y\in \D} \|\phi^{(n)}(r,y)-\phi^{(n-1)}(r,y)\|_{\lm}\,dr,
	\end{align*}	
	which implies 
	\begin{equation}\label{ndif}
		\sup_{x\in\D}\|\phi^{(n+1)}(t,x)-\phi^{(n)}(t,x)\|_{\lm}\le \frac1{n!}\|\phi^{(1)}-\phi^{(0)}\|_{\Ctimespacedzerom{[0,1]}}[h]_{\Lip(\R^d)}^n t^n.
	\end{equation}	
	Hence for each fixed $t\in[0,1]$, $x\in\D$ the sequence $(\phi^{(n)}(t,x))_{n\in\Z_+}$ converges in $\lm$. By continuity of $h$, its limit $\phi$ solves \eqref{wtyeq} and thus $Z$ solves \eqref{SPDEgeneric}. We also see that $Z$ is adapted to $(\F_t)$ as an $L_m(\Omega)$ limit of $(\F_t)$-adapted processes.
	
	Further, we also have from \eqref{ndif}
	\begin{align}\label{infnormpic}
		\|\phi\|_{\Ctimespacedzerom{[0,1]}}&\le \|\phi^{(0)}\|_{\Ctimespacedzerom{[0,1]}}+ \|\phi^{(1)}-\phi^{(0)}\|_{\Ctimespacedzerom{[0,1]}}e^{[h]_{\Lip(\R^d)}}\nn\\
		&\le \|z\|_{L_\infty(\D)}+(h(0)+[h]_{\Lip(\R^d)}\|z\|_{L_\infty(\D)} +[h]_{\Lip(\R^d)}\|R\|_{\Ctimespacedzerom{[0,1]}})e^{[h]_{\Lip(\R^d)}}.
	\end{align}	
	Recall that by assumptions we have $z\in\BB(\D)$ and $\|R\|_{\Ctimespacedzerom{[0,1]}}<\infty$. Therefore,   ${\|Z-R\|_{\Ctimespacedzerom{[0,1]}}<\infty}$.
	
	Next, we have
	\begin{align*}
		[\phi]_{\Ctimespaced{1}{m}{[0,1]}}&\le \|h(\phi(\cdot,\cdot)+R(\cdot,\cdot)) \|_{\Ctimespacedzerom{[0,1]}}\\
		&\le h(0)+[h]_{\Lip(\R^d)}(	\|\phi\|_{\Ctimespacedzerom{[0,1]}}+\|R\|_{\Ctimespacedzerom{[0,1]}})<\infty,
	\end{align*}
	where in the last line we used that $\|\phi\|_{\Ctimespacedzerom{[0,1]}}<\infty$ thanks to \eqref{infnormpic}. Using \eqref{infnormpic}, we get 
	$\|\phi\|_{\Ctimespaced{1}{m}{[0,1]}}<\infty$.
\end{proof}

\begin{proposition}\label{p:spacess}
Let $f\colon D\times\Omega\to\R$ be a bounded measurable function. Then there exists $C>0$ such that for any 
$t\in(0,1]$, $m\ge1$, $x,y\in D$ one has 
\begin{equation}\label{mainrezh}
	\sup_{x,y\in D}\|P_t f(x)-P_t f(y)\|_{\lm}\le  C |x-y|t^{-1/2} \sup_{z\in D}\| f(z)\|_{\lm}.
\end{equation}	
\end{proposition}
\begin{proof}
Let $t>0$, $x,y\in D$. Then, by Jensen's inequality
\begin{align*}
\|P_tf(x)-P_tf(y)\|_{\lm}&=\Bigl\|\int_D (p_t(x,z)-p_t(y,z))f(z)\,dz\Bigr\|_{\lm}\\
&\le 
\int_D |p_t(x,z)-p_t(y,z)|\,\|f(z)\|_{\lm}\,dz\\
&\le \int_D |p_t(x,z)-p_t(y,z)|\,dz\,  \sup_{z\in D} \|f(z)\|_{\lm}.
\end{align*}
By \cite[Lemma C.2]{ABLM}, we have 	$ \int_D |p_t(x,z)-p_t(y,z)|\,dz\le C |x-y|t^{-1/2}$ for some $C>0$, which yields \eqref{mainrezh}.
\end{proof}

\bibliographystyle{alpha1}
\bibliography{biblio}

\newcommand{\etalchar}[1]{$^{#1}$}
\begin{thebibliography}{ABLM24b}

\bibitem[ART21]{ART21}
Lukas Anzeletti, Alexandre Richard, and Etienne Tanr{\'e}.
\newblock Regularisation by fractional noise for one-dimensional differential
  equations with nonnegative distributional drift.
\newblock {\em arXiv preprint arXiv:2112.05685}, 2021.

\bibitem[ABLM24a]{ABLMmw}
Siva Athreya, Oleg Butkovsky, Khoa L\^{e}, and Leonid Mytnik.
\newblock Analytically weak and mild solutions to stochastic heat equation with
  irregular drift.
\newblock {\em arXiv preprint arXiv:2410.06599}, 2024.

\bibitem[ABLM24b]{ABLM}
Siva Athreya, Oleg Butkovsky, Khoa L\^{e}, and Leonid Mytnik.
\newblock Well-posedness of stochastic heat equation with distributional drift
  and skew stochastic heat equation.
\newblock {\em Comm. Pure Appl. Math.}, 77(5):2708--2777, 2024.

\bibitem[ABM20]{ABM2020}
Siva Athreya, Oleg Butkovsky, and Leonid Mytnik.
\newblock Strong existence and uniqueness for stable stochastic differential
  equations with distributional drift.
\newblock {\em Ann. Probab.}, 48(1):178--210, 2020.

\bibitem[AT00]{AT00}
Siva Athreya and Roger Tribe.
\newblock Uniqueness for a class of one-dimensional stochastic {PDE}s using
  moment duality.
\newblock {\em Ann. Probab.}, 28(4):1711--1734, 2000.

\bibitem[BM05]{Yura}
Yuri Bakhtin and Jonathan~C. Mattingly.
\newblock Stationary solutions of stochastic differential equations with memory
  and stochastic partial differential equations.
\newblock {\em Commun. Contemp. Math.}, 7(5):553--582, 2005.

\bibitem[BMS24]{BMS24}
Clayton Barnes, Leonid Mytnik, and Zhenyao Sun.
\newblock Wright-{F}isher stochastic heat equations with irregular drifts.
\newblock {\em arXiv preprint arXiv:2402.11160}, 2024.

\bibitem[BC01]{BC}
Richard~F. Bass and Zhen-Qing Chen.
\newblock Stochastic differential equations for {D}irichlet processes.
\newblock {\em Probab. Theory Related Fields}, 121(3):422--446, 2001.

\bibitem[BC03]{BC03}
Richard~F Bass and Zhen-Qing Chen.
\newblock Brownian motion with singular drift.
\newblock {\em The Annals of Probability}, 31(2):791--817, 2003.

\bibitem[BZ14]{BZ14}
Said~Karim Bounebache and Lorenzo Zambotti.
\newblock A skew stochastic heat equation.
\newblock {\em Journal of Theoretical Probability}, 27(1):168--201, 2014.

\bibitem[BDG21]{BDG}
Oleg Butkovsky, Konstantinos Dareiotis, and M\'{a}t\'{e} Gerencs\'{e}r.
\newblock Approximation of {SDE}s: a stochastic sewing approach.
\newblock {\em Probab. Theory Related Fields}, 181(4):975--1034, 2021.

\bibitem[BKS20]{BKS18}
Oleg Butkovsky, Alexei Kulik, and Michael Scheutzow.
\newblock Generalized couplings and ergodic rates for {SPDE}s and other
  {M}arkov models.
\newblock {\em Ann. Appl. Probab.}, 30(1):1--39, 2020.

\bibitem[BLM23]{blm23}
Oleg Butkovsky, Khoa L{\^e}, and Leonid Mytnik.
\newblock Stochastic equations with singular drift driven by fractional
  brownian motion.
\newblock {\em arXiv preprint arXiv:2302.11937}, 2023.

\bibitem[BS20]{BScmp}
Oleg Butkovsky and Michael Scheutzow.
\newblock Couplings via comparison principle and exponential ergodicity of
  {SPDE}s in the hypoelliptic setting.
\newblock {\em Comm. Math. Phys.}, 379(3):1001--1034, 2020.

\bibitem[CG16]{CG16}
R.~Catellier and M.~Gubinelli.
\newblock Averaging along irregular curves and regularisation of {ODE}s.
\newblock {\em Stochastic Process. Appl.}, 126(8):2323--2366, 2016.

\bibitem[CGL{\etalchar{+}}24]{choukplus}
Khalil Chouk, Massimiliano Gubinelli, Guopeng Li, Jiawei Li, and Tadahiro Oh.
\newblock Nonlinear pdes with modulated dispersion ii: Korteweg--de vries
  equation.
\newblock {\em arXiv preprint arXiv:1406.7675v2}, 2024.

\bibitem[DSS24]{Marta}
Robert~C Dalang and Marta Sanz-Sol{\'e}.
\newblock Stochastic partial differential equations, space-time white noise and
  random fields.
\newblock {\em arXiv preprint arXiv:2402.02119}, 2024.

\bibitem[DG24]{DGmult24}
Konstantinos Dareiotis and M\'at\'e Gerencs\'er.
\newblock Path-by-path regularisation through multiplicative noise in rough,
  {Y}oung, and ordinary differential equations.
\newblock {\em Ann. Probab.}, 52(5):1864--1902, 2024.

\bibitem[DHL24]{D24}
Konstantinos Dareiotis, Teodor Holland, and Khoa L\^{e}.
\newblock Regularisation by multiplicative noise for reaction–diffusion
  equations.
\newblock {\em arXiv preprint arXiv:2409.11130}, 2024.

\bibitem[Doe38]{doeblin}
Wolfgang Doeblin.
\newblock Expos{\'e} de la th{\'e}orie des cha{\i}nes simples constantes de
  markova un nombre fini d'{\'e}tats.
\newblock {\em Math{\'e}matique de l'Union Interbalkanique}, 2(77-105):78--80,
  1938.

\bibitem[EGZ19]{Eberle}
Andreas Eberle, Arnaud Guillin, and Raphael Zimmer.
\newblock Quantitative {H}arris-type theorems for diffusions and
  {M}c{K}ean-{V}lasov processes.
\newblock {\em Trans. Amer. Math. Soc.}, 371(10):7135--7173, 2019.

\bibitem[Eva98]{Evans}
Lawrence~C. Evans.
\newblock {\em Partial differential equations}, volume~19 of {\em Graduate
  Studies in Mathematics}.
\newblock American Mathematical Society, Providence, RI, 1998.

\bibitem[Fla11]{F11}
Franco Flandoli.
\newblock {\em Random perturbation of {PDE}s and fluid dynamic models}, volume
  2015 of {\em Lecture Notes in Mathematics}.
\newblock Springer, Heidelberg, 2011.
\newblock Lectures from the 40th Probability Summer School held in Saint-Flour,
  2010.

\bibitem[Fla13]{flandoli2013topics}
Franco Flandoli.
\newblock Topics on regularization by noise.
\newblock {\em Lecture Notes, University of Pisa. Available at
  \url{http://users.dma.unipi.it/~flandoli/Berlino_Lectures_Flandoli.pdf}},
  2013.

\bibitem[FIR17]{FIR17}
Franco Flandoli, Elena Issoglio, and Francesco Russo.
\newblock Multidimensional stochastic differential equations with
  distributional drift.
\newblock {\em Trans. Amer. Math. Soc.}, 369(3):1665--1688, 2017.

\bibitem[FHL21]{friz2021}
Peter~K Friz, Antoine Hocquet, and Khoa L{\^e}.
\newblock Rough stochastic differential equations.
\newblock {\em arXiv preprint arXiv:2106.10340}, 2021.

\bibitem[FV10]{FV2010}
Peter~K. Friz and Nicolas~B. Victoir.
\newblock {\em Multidimensional stochastic processes as rough paths}, volume
  120 of {\em Cambridge Studies in Advanced Mathematics}.
\newblock Cambridge University Press, Cambridge, 2010.
\newblock Theory and applications.

\bibitem[FZ18]{FZ18}
Peter~K. Friz and Huilin Zhang.
\newblock Differential equations driven by rough paths with jumps.
\newblock {\em J. Differential Equations}, 264(10):6226--6301, 2018.

\bibitem[Gal23]{LucioNLY}
Lucio Galeati.
\newblock Nonlinear {Y}oung differential equations: a review.
\newblock {\em J. Dynam. Differential Equations}, 35(2):985--1046, 2023.

\bibitem[GG22]{GG22}
Lucio Galeati and M{\'a}t{\'e} Gerencs{\'e}r.
\newblock Solution theory of fractional {SDE}s in complete subcritical regimes.
\newblock {\em arXiv preprint arXiv:2207.03475}, 2022.

\bibitem[Ges18]{bengess}
Benjamin Gess.
\newblock Regularization and well-posedness by noise for ordinary and partial
  differential equations.
\newblock In {\em Stochastic partial differential equations and related
  fields}, volume 229 of {\em Springer Proc. Math. Stat.}, pages 43--67.
  Springer, Cham, 2018.

\bibitem[GHMR17]{GMR}
Nathan Glatt-Holtz, Jonathan~C. Mattingly, and Geordie Richards.
\newblock On unique ergodicity in nonlinear stochastic partial differential
  equations.
\newblock {\em J. Stat. Phys.}, 166(3-4):618--649, 2017.

\bibitem[GP24]{grafner2024}
Lukas Gr{\"a}fner and Nicolas Perkowski.
\newblock Weak well-posedness of energy solutions to singular sdes with
  supercritical distributional drift.
\newblock {\em arXiv preprint arXiv:2407.09046}, 2024.

\bibitem[GLT06]{Tindel}
Massimiliano Gubinelli, Antoine Lejay, and Samy Tindel.
\newblock Young integrals and {SPDE}s.
\newblock {\em Potential Anal.}, 25(4):307--326, 2006.

\bibitem[GK96]{MR1392450}
Istv\'{a}n Gy\"{o}ngy and Nicolai Krylov.
\newblock Existence of strong solutions for {I}t\^{o}'s stochastic equations
  via approximations.
\newblock {\em Probab. Theory Related Fields}, 105(2):143--158, 1996.

\bibitem[GP93a]{GP93a}
Istv\'{a}n Gy\"{o}ngy and \'{E}. Pardoux.
\newblock On quasi-linear stochastic partial differential equations.
\newblock {\em Probab. Theory Related Fields}, 94(4):413--425, 1993.

\bibitem[GP93b]{GP93b}
Istv\'{a}n Gy\"{o}ngy and \'{E}. Pardoux.
\newblock On the regularization effect of space-time white noise on
  quasi-linear parabolic partial differential equations.
\newblock {\em Probab. Theory Related Fields}, 97(1-2):211--229, 1993.

\bibitem[Hai02]{H02}
M.~Hairer.
\newblock Exponential mixing properties of stochastic {PDE}s through asymptotic
  coupling.
\newblock {\em Probab. Theory Related Fields}, 124(3):345--380, 2002.

\bibitem[HMS11]{HMS}
M.~Hairer, J.~C. Mattingly, and M.~Scheutzow.
\newblock Asymptotic coupling and a general form of {H}arris' theorem with
  applications to stochastic delay equations.
\newblock {\em Probab. Theory Related Fields}, 149(1-2):223--259, 2011.

\bibitem[HM06]{HMAOM}
Martin Hairer and Jonathan~C. Mattingly.
\newblock Ergodicity of the 2{D} {N}avier-{S}tokes equations with degenerate
  stochastic forcing.
\newblock {\em Ann. of Math. (2)}, 164(3):993--1032, 2006.

\bibitem[Han22]{Han}
Yi~Han.
\newblock Exponential ergodicity of stochastic heat equations with {H}\" older
  coefficients.
\newblock {\em arXiv preprint arXiv:2211.08242}, 2022.

\bibitem[Han24]{Hanwave}
Yi~Han.
\newblock Stochastic wave equation with {H}\"{o}lder noise coefficient:
  well-posedness and small mass limit.
\newblock {\em J. Funct. Anal.}, 286(3):Paper No. 110224, 46, 2024.

\bibitem[HZ23]{hao2023sdes}
Zimo Hao and Xicheng Zhang.
\newblock {SDE}s with supercritical distributional drifts.
\newblock {\em arXiv preprint arXiv:2312.11145}, 2023.

\bibitem[HP21]{HP20}
Fabian~Andsem Harang and Nicolas Perkowski.
\newblock {$C^\infty$}-regularization of {ODE}s perturbed by noise.
\newblock {\em Stoch. Dyn.}, 21(8):Paper No. 2140010, 29, 2021.

\bibitem[KP25]{kremp2023}
Helena Kremp and Nicolas Perkowski.
\newblock Rough weak solutions for singular l\'{e}vy sdes.
\newblock {\em Probab. Theory Related Fields, arXiv preprint arXiv:2309.15460},
  2025.

\bibitem[KR05]{kr_rock05}
N.~V. Krylov and M.~R\"ockner.
\newblock Strong solutions of stochastic equations with singular time dependent
  drift.
\newblock {\em Probab. Theory Related Fields}, 131(2):154--196, 2005.

\bibitem[Kul18]{Kulikbook}
Alexei Kulik.
\newblock {\em Ergodic behavior of {M}arkov processes}, volume~67 of {\em De
  Gruyter Studies in Mathematics}.
\newblock De Gruyter, Berlin, 2018.
\newblock With applications to limit theorems.

\bibitem[KS20]{Ksch}
Alexei Kulik and Michael Scheutzow.
\newblock Well-posedness, stability and sensitivities for stochastic delay
  equations: a generalized coupling approach.
\newblock {\em Ann. Probab.}, 48(6):3041--3076, 2020.

\bibitem[L{\^e}20]{LeSSL}
Khoa L{\^e}.
\newblock A stochastic sewing lemma and applications.
\newblock {\em Electron. J. Probab.}, 25:Paper No. 38, 55, 2020.

\bibitem[Lig93]{lig}
Thomas~M. Liggett.
\newblock The coupling technique in interacting particle systems.
\newblock In {\em Doeblin and modern probability ({B}laubeuren, 1991)}, volume
  149 of {\em Contemp. Math.}, pages 73--83. Amer. Math. Soc., Providence, RI,
  1993.

\bibitem[Mat02]{M02}
Jonathan~C. Mattingly.
\newblock Exponential convergence for the stochastically forced
  {N}avier-{S}tokes equations and other partially dissipative dynamics.
\newblock {\em Comm. Math. Phys.}, 230(3):421--462, 2002.

\bibitem[Nua06]{Nu}
David Nualart.
\newblock {\em The Malliavin calculus and related topic}, volume~24 of {\em
  Cambridge Studies in Advanced Mathematics}.
\newblock Cambridge University Press, Cambridge, 2006.

\bibitem[NO02]{OuNu}
David Nualart and Youssef Ouknine.
\newblock Regularization of differential equations by fractional noise.
\newblock {\em Stochastic Process. Appl.}, 102(1):103--116, 2002.

\bibitem[Per18]{perkowski2018energy}
Nicolas Perkowski.
\newblock Energy solutions for the stochastic burgers equation: uniqueness and
  applications.
\newblock {\em Lecture notes. Available at
  \url{www.mathematik.hu-berlin.de/~perkowsk/files/energy-minicourse.pdf}},
  2018.

\bibitem[Pri12]{Pr12}
Enrico Priola.
\newblock Pathwise uniqueness for singular {SDE}s driven by stable processes.
\newblock {\em Osaka J. Math.}, 49(2):421--447, 2012.

\bibitem[Sch05]{ScheutzowCE}
Michael Scheutzow.
\newblock Exponential growth rates for stochastic delay differential equations.
\newblock {\em Stoch. Dyn.}, 5(2):163--174, 2005.

\bibitem[TTW74]{bib:ttw74}
Hiroshi Tanaka, Masaaki Tsuchiya, and Shinzo Watanabe.
\newblock Perturbation of drift-type for {L}\'evy processes.
\newblock {\em J. Math. Kyoto Univ.}, 14:73--92, 1974.

\bibitem[Tsy09]{Tsy}
Alexandre~B. Tsybakov.
\newblock {\em Introduction to nonparametric estimation}.
\newblock Springer Series in Statistics. Springer, New York, 2009.
\newblock Revised and extended from the 2004 French original, Translated by
  Vladimir Zaiats.

\bibitem[Ver80]{ver80}
A.~Ju. Veretennikov.
\newblock Strong solutions and explicit formulas for solutions of stochastic
  integral equations.
\newblock {\em Mat. Sb. (N.S.)}, 111(153)(3):434--452, 480, 1980.

\bibitem[Ver88]{ver88}
A~Yu Veretennikov.
\newblock Bounds for the mixing rate in the theory of stochastic equations.
\newblock {\em Theory of Probability \& Its Applications}, 32(2):273--281,
  1988.

\bibitem[Vil09]{Villani}
C\'{e}dric Villani.
\newblock {\em Optimal transport}, volume 338 of {\em Grundlehren der
  mathematischen Wissenschaften [Fundamental Principles of Mathematical
  Sciences]}.
\newblock Springer-Verlag, Berlin, 2009.
\newblock Old and new.

\bibitem[ZZ17]{bib:zz17}
Xicheng Zhang and Guohuan Zhao.
\newblock Heat kernel and ergodicity of {SDE}s with distributional drifts.
\newblock {\em arXiv preprint arXiv:1710.10537}, 2017.

\bibitem[Zvo74]{zvonkin74}
A.~K. Zvonkin.
\newblock A transformation of the phase space of a diffusion process that will
  remove the drift.
\newblock {\em Mat. Sb. (N.S.)}, 93(135):129--149, 152, 1974.

\end{thebibliography}

\end{document}